\documentclass{aastex6}

\usepackage{amssymb}
\usepackage{graphicx}\usepackage{subfigure}
\usepackage{amsmath}
\usepackage[title]{appendix}
\usepackage{mathrsfs}
\usepackage{amsbsy}
\usepackage{amsfonts}
\usepackage{mathrsfs}
\usepackage{latexsym}
\usepackage{psfrag}
\usepackage{multirow}

\usepackage{paralist}

\fullcollaborationName{The Friends of AASTeX Collaboration}

\def\pt{\partial}\def\dfr#1#2{\displaystyle{\frac{#1}{#2}}}
\renewcommand{\vec}[1]{\mbox{\boldmath \small $#1$}}
\newcommand{\svec}[1]{\mbox{\boldmath \scriptsize $#1$}}
\newcommand{\ssvec}[1]{\mbox{\boldmath \tiny $#1$}}

\newtheorem{Def}{Definition}[section]
\newtheorem{example}{Example}[section]
\newtheorem{lemma}{Lemma}[section]
\newtheorem{thm}{Theorem}[section]
\newtheorem{remark}{Remark}[section]
\numberwithin{equation}{section}
\numberwithin{figure}{section}
\numberwithin{table}{section}

\begin{document}


\title{Physical-constraints-preserving central discontinuous Galerkin methods for
special relativistic hydrodynamics with a general equation of state}


\author{Kailiang Wu\altaffilmark{1} and Huazhong Tang\altaffilmark{2}}
%


\affil{HEDPS, CAPT \& LMAM, School of Mathematical Sciences\\
Peking University, Beijing 100871, P.R. China}


\altaffiltext{1}{wukl@pku.edu.cn}
\altaffiltext{2}{hztang@math.pku.edu.cn (Corresponding author)}

\begin{abstract}

The ideal gas equation of state (EOS) with a constant adiabatic index
is a poor approximation for most relativistic astrophysical flows,
although it is commonly used in relativistic hydrodynamics.
The paper develops high-order accurate physical-constraints-preserving (PCP)
central discontinuous Galerkin (DG) methods for the one- and two-dimensional  special relativistic
hydrodynamic (RHD) equations with a general EOS.
It is built on the theoretical  analysis of the admissible states for the RHD
and the  PCP limiting procedure enforcing the admissibility
of  central DG solutions.
The convexity, scaling and orthogonal invariance, and Lax-Friedrichs splitting property of
the admissible state set are first proved with the aid of its equivalent form,
and then  the high-order central DG methods with the PCP limiting procedure and strong stability preserving time discretization are proved to preserve the positivity of the density, pressure, and specific internal
energy, and the bound of the fluid velocity,
 maintain the high-order accuracy, and be $L^1$-stable.
The accuracy, robustness, and effectiveness of the proposed methods
are demonstrated by several 1D and 2D numerical examples
involving large Lorentz factor, strong discontinuities, or low density or pressure etc.

\end{abstract}




\section{Introduction}
\label{sec:intro}

In many cases, high-energy physics and astrophysics may involve fluid flows
where  the velocities are close to the speed of light or the influence of large gravitational potentials cannot be ignored such that the relativistic effect should be taken into account.
Relativistic hydrodynamics (RHD) is important in investigating numerous astrophysical phenomena, from stellar to galactic scales, e.g. gamma-ray bursts, astrophysical jets, core collapse super-novae, coalescing
neutron stars, formation of black holes, etc.

The RHD equations  are highly nonlinear so that their analytical treatment is extremely difficult.
The numerical simulation has become a primary and powerful approach to understand the physical mechanisms in the RHDs. The pioneering numerical work may date back to
the May and White finite difference code via artificial viscosity  for the spherically symmetric general RHD equations in the Lagrangian coordinate \citep{May1966,May1967}.
Wilson first attempted to solve multi-dimensional RHD equations in the Eulerian coordinate by using the finite difference method with the artificial viscosity technique \citep{Wilson:1972}.
Since 1990s, the numerical study of the RHDs began to attract considerable attention, and various modern shock-capturing methods with an exact
or approximate Riemann solver have been developed for the RHD equations,
the readers are referred to the early review articles \citep{MME:2003,Font2008}
and more recent works \citep{WuTang2014,WuTang2015,WuTang2016GRP} as well as references therein.

Most existing methods do not preserve the positivity of the density and pressure as well as the specific internal
energy and the bound of the fluid velocity at the same time,
although they have been used to solve some RHD problems successfully.
There exists the big risk of failure when a numerical scheme
is applied to the RHD problems with large Lorentz
factor, low density or pressure, or strong discontinuity,
because as soon as the negative
density or pressure, or the superluminal
fluid velocity may be obtained, the eigenvalues
of the Jacobian matrix become imaginary so that the discrete problem becomes ill-posed.
It is of great significance to develop high-order accurate numerical schemes, whose solutions
 satisfy the intrinsic physical constraints.
Recent years have witnessed some advances in
developing high order accurate  bound-preserving type schemes for hyperbolic conservation laws.
Those schemes are mainly built on two types of  limiting procedures.
One is the simple scaling limiting procedures
for the reconstructed or evolved solution polynomials in
a finite volume or discontinuous Galerkin method, see e.g. \citep{zhang2010,zhang2012a,zhang2010b,zhang2011,Xing2010}.
Another is the flux-corrected limiting procedure,
which can be used on high-order finite difference, finite volume,
and discontinuous Galerkin methods, see e.g. \citep{Xu_MC2013,Hu2013,Liang2014,JiangXu2013,XiongQiuXu2014,Christlieb}.
A survey of the maximum-principle-satisfying or positivity-preserving
 high-order schemes based on the first type limiter was presented in \citep{zhang2011b}.
 The readers are also referred to \citep{xuzhang2016} for a review of
 those two approaches for enforcing the bound-preserving property in high order schemes.
Two works were recently made to develop the physical-constraints-preserving (PCP) schemes for the special RHD equations with an ideal equation of state (EOS) by extending the above bound-preserving techniques. One is the high-order accurate PCP finite difference weighted essentially non-oscillatory (WENO) schemes  proposed in \citep{WuTang2015}, another is the bound preserving discontinuous Galerkin methods presented in \citep{QinShu2016}. Recently, the extension of PCP schemes to the ideal relativistic magnetohydrodynamics was first studied in \citep{WuTang2016}.


Although the ideal gas EOS with a constant adiabatic index
is commonly used in relativistic hydrodynamics,
it is a poor approximation for most relativistic astrophysical flows, see e.g. \citep{Choi-Wiita2010,Falle-Kom,Ryu,Synge1957}.
The aim of the paper is to study the properties of the admissible states
and develop the high-order accurate PCP central DG methods for the special RHD equations
with a general EOS.
The central DG method was first introduced in \citep{Liu2007,Liu2008}
for the hyperbolic problems and well developed for
the Hamilton-Jacobi equations \citep{LiYa2010}, ideal  magnetohydrodynamic equations \citep{LiXu2011,Yakovlev2013,LiXu2012}, and  relativistic hydrodynamics and magnetohydrodynamics \citep{Zhao}, etc.
Recently positivity-preserving techniques for central DG method were discussed in \citep{ChengLiQiuXu} without rigorous proof for the ideal magnetohydrodynamics.
In comparison with the existing works in the non-relativistic or relativistic case,
it is not trivial to develop
 high-order accurate provable PCP central DG method for the RHDs with general EOS.
The technical challenge is mainly coming from  the inherent nonlinear coupling between the RHD equations due to the Lorentz factor and  general EOS,
no explicit expressions of the primitive variables and
flux vectors with respect to the conservative vector,
and  one more physical constraint for the fluid velocity in addition to the positivity of the density, pressure, and specific internal energy.

The paper is organized as follows.
Section \ref{sec:governingequ} introduces the governing equations and  the general equation of state.
Section \ref{sec:admis-state}  analyzes the admissible state set.
Section \ref{sec:scheme} presents the
high-order accurate PCP central DG methods for the 1D and 2D RHD equations with a
general EOS. 
Section \ref{sec:experiments}
gives several  numerical examples to verify the accuracy,
robustness, and effectiveness of the proposed methods for ultra-relativistic problems with large Lorentz factor, strong discontinuities, or low density or pressure, etc.
Concluding remarks are presented in Section \ref{sec:conclude}.

\section{Governing equations}
\label{sec:governingequ}
In the framework of special relativity,
the ideal fluid flow is governed by the laws of particle number conservation
and energy-momentum conservation \citep{Landau}.
In the laboratory frame of reference, the $d$-dimensional special RHD equations
may be written into a system of conservation laws as follows
\begin{equation}\label{eqn:coneqn3d}
\displaystyle\frac{\partial \vec{U}}{\partial t} +
\sum^d_{i=1}\frac{\partial \vec{F}_i(\vec{U})}{\partial x_i}= \vec 0,
\end{equation}
where $d=1$, or 2, or 3, $\vec{U}=\left(D, \vec m, E\right)^{\rm T}$ denotes the conservative vector, and $\vec{F}_i$ is the flux
in the $x_i$-direction, which is defined by
\begin{align}
\vec{F}_i =& \left(Dv_i, v_i \vec m + p \vec e_i, m_i\right)^{\rm T}, ~i=1,\cdots, d. \label{eq:fluxdef}
\end{align}
Here the mass density $D=\rho W$,
the momentum density (row) vector  $\vec m=(m_1, \cdots, m_d)=Dh W \vec v$,
the energy density $E=DhW-p$,
and   $\rho$, $\vec v= (v_1,  \cdots,  v_d)$ , and $p$ denote
the rest-mass density, fluid velocity vector, and pressure,  respectively.
Moreover,  the row vector $\vec e_i$ denote the $i$-th row of the identity matrix of order $d$,
 $W=1/\sqrt{1-v^2}$ is the Lorentz factor
with $v:=(v_1^2+\cdots +v_d^2)^{1/2}$, and $h$ denotes the specific enthalpy defined by
\begin{equation}
\label{eq:h}
h = 1 + e + \dfr{p}{\rho},
\end{equation}
with units in which the speed of light $c$ is equal to one,
and $e$ is the specific internal energy.

An additional thermodynamic equation relating state variables,
 the so-called equation of
state (EOS),
is needed to close the system \eqref{eqn:coneqn3d}. A general EOS may be expressed as
\begin{equation}
\label{eq:EOS:e}
e = e(p,\rho),
\end{equation}
or 
\begin{equation}
\label{eq:EOS:h}
h =h(p,\rho)= 1+e(p,\rho) + p/\rho.
\end{equation}
The relativistic kinetic theory reveals \citep{Taub}  that the EOS \eqref{eq:EOS:h} should satisfy
$$
\left( h - \frac{p}{\rho} \right) \left( h - \frac{4p}{\rho} \right) \ge 1,
$$
which implies a weaker inequality
\begin{equation}\label{eq:hcondition1}
h(p,\rho) \ge \sqrt{1+p^2/\rho^2}+p/\rho.
\end{equation}
It can also be derived from the kinetic theory, see Appendix \ref{Appendix-A},
and will be  useful in analyzing the admissible state of RHD equations
\eqref{eqn:coneqn3d}.

For a general EOS, the local sound speed $c_s$ is defined by
\begin{equation}\label{eq:cs2def}
c_s^2 = \frac{1}{h}\frac{\pt p(\rho,S)}{\pt \rho} = \frac{1 }{  \frac{\pt \rho(p,S)}{\pt p} h},
\end{equation}
where the entropy $S$ is related to other thermodynamic variables  \citep{Landau} by
\begin{equation}\label{eq:diffrelation}
T {\rm d} S = {\rm d} e + p {\rm d} \Big( \frac{1}{\rho} \Big)
= {\rm d}h -  \frac{1}{\rho}  {\rm d} p,
\end{equation}
here $T$ denotes the thermodynamical temperature.

We will consider the causal EOS, i.e. those for which
\begin{equation}\label{eq:cscondition}
0<c_s < c = 1.
\end{equation}
 For such EOS, the hyperbolic property of   \eqref{eqn:coneqn3d}
 is preserved. The readers are referred to \citep{ZhaoTang2013}
 for the calculation of eigenvalues and (left and right) eigenvectors for the system \eqref{eqn:coneqn3d}
 with $d=1$ and 2 and  a general EOS.



\begin{lemma}\label{lem:gEOSC}
If the fluid's coefficient of thermal expansion $\beta=- \frac{\pt {\ln \rho(T,p)}}{\pt T} >0$, then the following inequality holds
\begin{equation}\label{eq:gEOSC}
h\left(\frac1{\rho} - \frac{\pt h(p,\rho)}{\pt p} \right) < \frac{\pt h(p,\rho)}{\pt \rho} < 0.
\end{equation}
\end{lemma}

\begin{proof}
Taking partial derivatives of $e(p,\rho(T,p))=e(T,\rho(T,p))$
with respect to $T$  gives
$$
\frac{\pt e(p,\rho) }{\pt \rho} \frac{\pt \rho(T,p) }{\pt T} = \frac{\pt e(T,\rho) }{\pt T} + \frac{\pt e(T,\rho) }{\pt \rho} \frac{\pt \rho(T,p) }{\pt T},
$$
being equivalent to
\begin{equation}\label{eq:generalEOSproof}
\frac{\pt e(p,\rho) }{\pt \rho} \frac{\pt \rho(T,p) }{\pt T} = C_p + \frac{p}{\rho^2} \frac{\pt \rho(T,p)}{\pt T},
\end{equation}
where
$$
C_p := \frac{\pt e(T,\rho) }{\pt T} + \left(  \frac{\pt e(T,\rho) }{\pt \rho} - \frac{p}{\rho^2}  \right) \frac{\pt \rho(T,p) }{\pt T} \in \mathbb{R}^+,
$$
denotes the specific heat capacity at constant pressure.
Using the definition of $\beta$ and \eqref{eq:generalEOSproof} gives
\begin{equation*} 
\frac{\pt e(p,\rho) }{\pt \rho} - \frac{p}{\rho^2}  = - \frac{C_p}{\rho \beta} <0.
\end{equation*}
Combining it with \eqref{eq:EOS:h} yields
\begin{equation}\label{eq:conditionB}
\frac{\partial h(p,\rho)}{\partial \rho} <0.
\end{equation}

Utilizing  $h=h(p,\rho)=h(p,\rho(p,S))$ and the chain rule of derivation gives
\begin{align*}
\frac1{\rho} \overset{\eqref{eq:diffrelation}}{=} \frac{\pt h(p,S)}{\pt p}
= \frac{\pt h(p,\rho)}{\pt p} + \frac{\pt h(p,\rho)}{\pt \rho}  \frac{\pt \rho(p,S)}{\pt p}
\overset{\eqref{eq:cs2def}}{=} \frac{\pt h(p,\rho)}{\pt p} +  \frac{\pt h(p,\rho)}{\pt \rho}  \frac{1}{h c_s^2}.
\end{align*}
It follows that
\begin{equation}\label{eq:newcs2}
h\left(\frac1{\rho} - \frac{\pt h(p,\rho)}{\pt p} \right) / \left( \frac{\pt h(p,\rho)}{\pt \rho} \right)
=  \frac{1}{c_s^2}   \overset{\eqref{eq:cscondition}}{>} 1,
\end{equation}
which completes the proof by \eqref{eq:conditionB}.
\end{proof}

The  hypothesis of Lemma \ref{lem:gEOSC} is valid
for  most of compressible fluids, e.g. the gases.

Before ending this section, we give several special EOS.
	The most commonly used EOS, which is called the ideal
	EOS, is given by
\begin{equation}\label{eq:EOSideal}
h = 1 + \frac{\Gamma p }{(\Gamma -1)\rho},
\end{equation}
where  $\Gamma$ denotes the adiabatic index. 
In general, the adiabatic index $\Gamma$  is taken as $5/3$
for mildly relativistic or subrelativistic cases
and as $4/3$ for ultrarelativistic cases
where $e\gg \rho$.
%
Although the EOS \eqref{eq:EOSideal} is commonly used in RHDs,
it is a poor approximation for most relativistic astrophysical flows.
It is borrowed from nonrelativistic thermodynamics and
inconsistent with relativistic kinetic theory, see \citep{Ryu}.
The EOS \eqref{eq:EOSideal}  is a reasonable approximation only
if the gas is either strictly subrelativistic or ultrarelativistic.
When the gas is semirelativistic or  two-component, 
\eqref{eq:EOSideal} is no longer correct.

Since the correct equation of state for the relativistic perfect
gas has been recognized as being important, several investigations
with a more general equation of state have been reported
in numerical relativistic hydrodynamics.
For the one-component perfect gases,
 several general EOS  have been used in the literature. For example,
the first      is  \citep{Mathews,Mignoneetal:2005}
 \begin{equation}\label{eq:EOS:Mathews}
 h =  \frac{5p}{2\rho}  +  \sqrt{\frac{9p^2}{4\rho^2} + 1}.
 \end{equation}
 %
and the second   \citep{Sokolov} is described as follows
\begin{equation}\label{eq:EOS:Sokolov}
h =  \frac{2p}{\rho} + \sqrt{\frac{4p^2}{\rho^2} + 1},
\end{equation}
Recently,   a new approximate EOS in \citep{Ryu} is given   as follows
\begin{equation}\label{eq:EOS:Ryu}
h =   \frac{ 2(6 p^2 + 4 p\rho + \rho^2) }{ \rho (3 p + 2 \rho) }.
\end{equation}

It is not difficult to verify that
besides the conditions \eqref{eq:hcondition1} and \eqref{eq:gEOSC},
the  EOS \eqref{eq:EOSideal}--\eqref{eq:EOS:Ryu}
 satisfy
 that $e(p,\rho)$ is continuously differentiable in ${\mathbb R}^+\times {\mathbb {R}}^+$ and satisfies
\begin{equation}\label{eq:epto0}
\mathop{\lim }\limits_{p \to 0^+ } e(p,\rho)=0,\quad \mathop {\lim }\limits_{p \to +\infty } e(p,\rho) =  + \infty,
\end{equation}
for any fixed positive $\rho$.

\section{Admissible states}
 \label{sec:admis-state}

For the RHD equations \eqref{eqn:coneqn3d}, it is very natural and intuitive
to {define}  the (physical) admissible states $ \vec U$.
\begin{Def}
 The set of admissible states of the RHD equations
\eqref{eqn:coneqn3d} is defined by
\begin{equation}\label{EQ-adm-set01}
{\mathcal G} := \Big\{ { \left. \vec U = (D,\vec m,E)^{\rm T} \right| {\rho (\vec U) > 0,p(\vec U) > 0,e(\vec U)>0,v(\vec U)<1} } \Big\}.
\end{equation}
\end{Def}
Unfortunately, four conditions in  \eqref{EQ-adm-set01} are much difficultly verified
by the given value of the conservative vector $\vec U$,
since there is no explicit expression for the  transformation $\vec U \mapsto  (\rho,p,e,\vec v)^{\rm}$.
It also indicates that it very difficult to study the properties
of  ${\mathcal G}$ and develop the PCP schemes  for he RHD equations
\eqref{eqn:coneqn3d} with the EOS \eqref{eq:EOS:e} or \eqref{eq:EOS:h}.
In practice,  if giving the value of $\vec U$,
then one has to  iteratively solve a nonlinear algebraic equation,
e.g. an equation for  the unknown pressure $p$
 \begin{equation}\label{eq:solvePgEOS}
  E + p = D h\Big(p, \rho^{[\svec U]}(p)  \Big) \left(1- |\vec m|^2/(E+p)^2\right)^{-1/2},\quad p \in \mathbb{R}^+,
 \end{equation}
 where
 \begin{equation*}
 \rho^{[\svec U]}(p) := D\sqrt {1 - |\vec m|^2 /{(E + p)^2 }}.
  \end{equation*}
  Once the positive solution of the above equation is obtained, denoted by   $p(\vec U)$,
   other variables may be  sequentially  calculated by
\begin{equation}\label{eq:solveVRHOgEOS}
  {v(\vec U)}  = \frac{{ |\vec m|}}{{E + p(\vec U)}},\quad
\rho (\vec U) = D\sqrt {1 -  { v^2(\vec U)}   }  , \quad e(\vec U) = e( p(\vec U),\rho(\vec U) ).
\end{equation}

For the ideal EOS \eqref{eq:EOSideal} with $\Gamma \in (1,2]$,
it  has 
been rigorously proved in \cite{WuTang2015} that the physical constraints  in \eqref{EQ-adm-set01}
are equivalent to two explicit constraints on conservative vector
\begin{equation}\label{eq:NEcondition}
D>0,\quad q(\vec U):= E-  \sqrt{D^2+|\vec m|^2}>0.
\end{equation}
Actually,  for a  general EOS \eqref{eq:EOS:h}, they are still necessary for $\vec U \in {\cal G}$.

\begin{lemma}\label{lemma:ne}
Under the condition \eqref{eq:hcondition1}, the admissible state $\vec U \in {\cal G}$ must satisfy \eqref{eq:NEcondition}.
\end{lemma}

\begin{proof}
Because $\rho$, $p$, and $e$ are positive  and  $0\le v<c=1$, it is easy to get
the following inequalities
\[
D = \frac{\rho }{{\sqrt {1 -   v^2 } }} > 0, \quad
E = \frac{{\rho h}}{{1 -  v^2 }} - p > \rho h - p \overset{\eqref{eq:h}}{=} \rho (1 + e) > 0.
\]
Using \eqref{eq:hcondition1} further gives
\begin{align*}
 E^2  - \left( {D^2  +  m^2 } \right) &= \left( {\frac{{\rho h}}{{1 - v^2 }} - p} \right)^2  - \frac{{\rho ^2 }}{{1 - v^2 }} - \left( {\frac{{\rho hv}}{{1 - v^2 }}} \right)^2  \\
 &= \left( {\frac{{\rho h}}{{1 - v^2 }}} \right)^2  +p^2 -2p{\frac{{\rho h}}{{1 - v^2 }}}
  - \frac{{\rho ^2 }}{{1 - v^2 }} - \left( {\frac{{\rho hv}}{{1 - v^2 }}} \right)^2
 \\
  &= \frac{1}{{1 - v^2 }}\left[ {\left( {\rho h - p} \right)^2  - \rho ^2  - p^2 v^2 } \right]\\
    &\overset{v< 1}{>}
    \frac{1}{{1 - v^2 }}\left[ {\rho ^2 \left( {1 + e} \right)^2  - \rho ^2  - p^2 } \right]
     \overset{\eqref{eq:hcondition1}}{>} 0.
\end{align*}
It follows that $q(\vec U)=E- \sqrt{D^2  +  m^2 }>0$.  The proof is completed.
\end{proof}


\begin{lemma}\label{thm:equDefgEOS}
If $\vec U=(D,\vec m,E)^{\rm T}$ satisfies  \eqref{eq:NEcondition}
and  $e(p,\rho)$ is continuously differentiable in ${\mathbb R}^+\times {\mathbb {R}}^+$, then  $\vec U$ belongs to ${\mathcal G}$ under the conditions \eqref{eq:hcondition1}, \eqref{eq:gEOSC}, and \eqref{eq:epto0}.
\end{lemma}

\begin{proof}
Consider the pressure function defined by
\[
\Psi^{[\ssvec U]} (p) :=   Dh\left( {p,  \rho^{[\ssvec U]}(p)  } \right)\sqrt {1 - \frac{{ |\vec m|^2 }}{{(E + p)^2 }}} -(E + p)\left(1- \frac{{ |\vec m|^2 }}{{(E + p)^2}} \right),\qquad p \in [0,+\infty),
\]
which is related to \eqref{eq:solvePgEOS}.
Obviously, for given $\vec U$ satisfying  \eqref{eq:NEcondition}, $\Psi^{[\ssvec U]} (p) \in C^1[0,+\infty)$  and its derivative satisfies
\begin{align}\nonumber
 \frac{ {\rm d} \Psi^{[\ssvec U]} (p)} {{\rm d} p }&= D \left[  {\frac{{\partial h }}{{\partial p}}} \left(p, \rho^{[\ssvec U]}(p)  \right)  \sqrt {1 - \frac{{|\vec m|^2 }}{{(E + p)^2 }}}  + \frac{{D |\vec m|^2 }}{{(E + p)^3 }}  {\frac{{\partial h}}{{\partial \rho }}}   \left(p, \rho^{[\ssvec U]}(p)  \right)   \right]
 \\ \nonumber
&~~~~~ + \frac{{D |\vec m|^2 }}{{(E + p)^3 }} h\left( {p,  \rho^{[\ssvec U]}(p)  } \right) \left( {1 - \frac{{|\vec m|^2 }}{{(E + p)^2 }}} \right)^{ - \frac{1}{2}}
- \frac{{ |\vec m|^2 }}{{(E + p)^2 }} -1
\\ \nonumber
& \overset{\eqref{eq:gEOSC}}{>}
D \left[\sqrt {1 - \frac{{ |\vec m|^2 }}{{(E + p)^2 }}}  - \frac{{ D |\vec m|^2 h\left( {p,  \rho^{[\ssvec U]}(p)  } \right)  }}{{(E + p)^3  }}     \right] \frac{{\partial h}}{{\partial p}} \left( {p,  \rho^{[\ssvec U]}(p)  } \right)
 \\ 
&  + \frac{{2 D |\vec m|^2 }}{{(E + p)^3 }}h\left( {p,  \rho^{[\ssvec U]}(p)  } \right)   \left( {1 - \frac{{ |\vec m|^2 }}{{(E + p)^2 }}} \right)^{ - \frac{1}{2}}
- \frac{{ |\vec m|^2 }}{{(E + p)^2 }} -1
=: \hat {\Psi}^{[\ssvec U]} (p). \label{eq:gEOSproof4}
\end{align}
Thanks to \eqref{eq:h} and \eqref{eq:epto0}, one  yields
$$
\mathop{\lim }\limits_{p \to 0^+ }  h\left( {p,  \rho^{[\ssvec U]}(p)  } \right) = 1,\quad
\mathop {\lim }\limits_{p \to +\infty }  e\left( {p,  \rho^{[\ssvec U]}(p)  } \right) = +\infty,
$$
which implies
\begin{align*}
& \mathop{\lim }\limits_{p \to 0^+ } \Psi^{[\ssvec U]} (p) = D \sqrt {1 - \frac{{ |\vec m|^2 }}{{E ^2 }}} + \frac{{ |\vec m|^2 }}{{E }} -E = \left( D-\sqrt{E^2-|\vec m|^2}\right) \sqrt {1 - \frac{{ |\vec m|^2 }}{{E ^2 }}} < 0, \\
& \mathop{\lim }\limits_{p \to +\infty } \Psi^{[\ssvec U]} (p) =
\mathop{\lim }\limits_{p \to +\infty }  D \left [1+ e \left( {p,  \rho^{[\ssvec U]}(p)  } \right) \right] \sqrt {1 - \frac{{|\vec m|^2 }}{{(E + p)^2 }}}
+ \frac{{ |\vec m|^2 }}{{E + p}} -E  = + \infty.
\end{align*}
By the {\em intermediate value theorem}, $\Psi^{[\ssvec U]} (p)$ has at least one positive zero, that is to say, there exist at least one positive solution to the algebraic equation $\Psi^{[\ssvec U]} (p)=0$ or \eqref{eq:solvePgEOS}.

The following task is to  prove the uniqueness of positive zero of $\Psi^{[\ssvec U]} (p)$.
The proof  by contradiction is used here.
Assume that $\Psi^{[\ssvec U]} (p)$ has more than one positive zeros and the smallest two are respectively denoted  by $p_1(\vec U)$ and $p_2(\vec U)$ satisfying $p_2(\vec U)>p_1(\vec U)>0$.
Because the equation $\Psi^{[\ssvec U]} (p) =0$ is equivalent to \eqref{eq:solvePgEOS},
one has  the identity
\begin{equation}\label{eq:gEOSproof5}
D  h\left( {p_i, \rho^{[\ssvec U]} (p_i)} \right) =(E+p_i)\sqrt{ 1-\frac{ |\vec m|^2}{(E+p_i)^2} }, \quad i = 1, 2.
\end{equation}
Combing such identity and the condition \eqref{eq:NEcondition} gives $h\left( {p_i, \rho^{[\ssvec U]} (p_i)} \right) >0$,
and further using \eqref{eq:gEOSC} yields
\begin{equation}\label{eq:proof66}
\frac{{\partial h}}{{\partial p}} \left(p_i, \rho^{[\ssvec U]}(p_i) \right)>\frac{1}{\rho^{[\ssvec U]} (p_i)}>0.
\end{equation}
Combining \eqref{eq:gEOSproof5}--\eqref{eq:proof66} with \eqref{eq:gEOSproof4} gives
\begin{align*}
\frac{ {\rm d} \Psi^{[\ssvec U]} } { {\rm d} p } (p_i) & > \hat {\Psi}^{[\ssvec U]} (p_i) \overset{\eqref{eq:gEOSproof5}}{=} D  \left( {1 - \frac{{|\vec m|^2 }}{{(E + p_i)^2 }}} \right)^{ \frac{3}{2}}
   \frac{{\partial h}}{{\partial p}} \left(p_i, \rho^{[\ssvec U]}(p_i) \right)+  \frac{{ |\vec m|^2 }}{{(E + p_i)^2 }} -1
   \\
   & \overset{\eqref{eq:proof66}}{>}  D  \left( {1 - \frac{{ |\vec m|^2 }}{{(E + p_i)^2 }}} \right)^{ \frac{3}{2}}
    \frac{1}{\rho^{[\ssvec U]} (p_i)} +  \frac{{ |\vec m|^2 }}{{(E + p_i)^2 }} -1 =0,
    \quad i = 1, 2.
\end{align*}
It indicates
\[
\mathop {\lim }\limits_{\delta p \to 0} \frac{{\Psi^{[\ssvec U]} (p_i  + \delta p) - \Psi^{[\ssvec U]} (p_i )}}{\delta p} =  \frac{ {\rm d} \Psi^{[\ssvec U]} } { {\rm d} p }  (p_i ) > 0,\quad i = 1, 2.
\]
By $\Psi^{[\ssvec U]} (p_i )=0$ and the $(\varepsilon, \delta)$-definition of limit, for $\varepsilon_i= \frac{1}{2} \frac{ {\rm d} \Psi^{[\ssvec U]} } { {\rm d} p } (p_i ) >0 $, there exists $\delta_i >0$ such that
$$\left | \frac{{ \Psi^{[\ssvec U]} (p_i  + \delta p)}}{\delta p} - \frac{ {\rm d} \Psi^{[\ssvec U]} } { {\rm d} p } (p_i ) \right| < \varepsilon_i,\quad \forall \delta p \in ( -\delta_i, \delta_i ),$$
which is equivalent to
$$  \varepsilon_i < \frac{{ \Psi^{[\ssvec U]} (p_i  + \delta p) }}{\delta p} < 3 \varepsilon_i ,\quad \forall \delta p \in ( -\delta_0, \delta_0 ),$$
where $\delta_0 = \min \left\{  \delta_1,\delta_2, \frac{p_2-p_1}{2} \right\}>0$.
 Therefore it holds that $(p_1 +  \frac{\delta_0}{2},~p_2 -  \frac{\delta_0}{2}) \subset (p_1,p_2) $ and
\begin{align*}
\Psi^{[\ssvec U]} \left( p_1 +  \frac{\delta_0}{2} \right) > 0,\quad  \Psi^{[\ssvec U]} \left( p_2 -  \frac{\delta_0}{2} \right) <0.
\end{align*}
 Thanks to the {\em intermediate value theorem}, $\Psi^{[\ssvec U]} (p)$ has zero in the interval $\left( p_1 +  \frac{\delta_0}{2},p_2 -  \frac{\delta_0}{2}\right)$. It conflicts  with the assumption that $p_1$ and $p_2$ are the smallest two positive zeros of $\Psi^{[\ssvec U]} (p)$. Hence the assumption does not hold and $\Psi^{[\ssvec U]}  (p)$ has unique positive zero, denoted by $p(\vec U)$.
 Substituting the  positive pressure $p(\vec U)$
 into \eqref{eq:solveVRHOgEOS} and using \eqref{eq:NEcondition} gives
\[
  {v(\vec U)}  = \frac{{ |\vec m|}}{{E + p(\vec U)}} <  \frac{{ |\vec m|}}{E} < 1,\quad
\rho (\vec U) = D\sqrt {1 -  { v^2(\vec U)}   }  > 0.
\]
For any $p,\rho \in \mathbb{R}^+$, utilizing \eqref{eq:gEOSC} gives
$$
\frac{\pt e(p,\rho)}{\pt p}>0,
$$
 which implies
$$
e(\vec U) = e( p(\vec U),\rho(\vec U) ) >  \mathop{\lim }\limits_{p \to 0^+ }  e( p,\rho(\vec U) )  \overset{  \eqref{eq:epto0} }{=}   0.
$$
In conclusion, $\vec U \in {\mathcal G}$. The proof is completed.
\end{proof}

\begin{remark}
Under the EOS conditions \eqref{eq:hcondition1}, \eqref{eq:gEOSC}, and \eqref{eq:epto0},
Lemmas \ref{lemma:ne} and \ref{thm:equDefgEOS} indicate that
the admissible set ${\mathcal G}$
 is equivalent to the set
\begin{equation}\label{EQ-adm-set02}
\hat {\mathcal G} := \left\{ { \left. \vec U = (D,\vec m,E)^{\rm T} \right| {D>0,~q(\vec U)>0} } \right\}.
\end{equation}
In comparison with     ${\mathcal G}$, 
two constraints in the set $\hat{\mathcal G}$  
are directly imposed on the conservative variables such that they are very easy to be verified when the value of $\vec U$ is given.
For that reason, the further discussion will be always
performed under the conditions \eqref{eq:hcondition1}, \eqref{eq:gEOSC} and \eqref{eq:epto0}.  
\end{remark}

With the help of the equivalence between ${\mathcal G} $ and $\hat{\mathcal G} $,
 the convexity of admissible state set   ${\mathcal G} $  may be proved by exactly following the proof of Lemma 2.2
 in \cite{WuTang2015}.

\begin{lemma}
\label{lam:convex}
The function $q(\vec U)$ is  concave and Lipschitz continuous with respect to  $\vec U$. The admissible set $\hat {\mathcal G}$ is a open convex set.
Moreover, $\lambda {\vec U}_1 + (1-\lambda) {\vec U}_0 \in \hat {\mathcal G} $ for any
${\vec U}_1 \in \hat {\mathcal G}$, ${\vec U}_0 \in  \hat {\mathcal G} \cup \partial \hat {\mathcal G}$, and  $\lambda \in(0,1)$.
\end{lemma}


By   the convexity of ${\mathcal G}$,  
some  properties of ${\mathcal G}$ can be further obtained.

\begin{lemma}
\label{lam:propertyG}
If assuming $\vec U \in {\mathcal G}$, then one has
\begin{itemize}[\hspace{0em}$\bullet$]
  \item[(i). {(Scaling invariance)}] $\lambda \vec U \in {\mathcal G}$, for all scalar $\lambda  > 0$.
  \item[(ii). { (Orthogonal invariance)}] $\vec T \vec U \in {\mathcal G}$,
  where $\vec T=\mbox{\rm diag}\{1,\vec T_{d},1\}$ and
  $\vec T_{d}$ denotes any orthogonal matrix of size $d$.
  \item[(iii). {(Lax-Friedrichs splitting)}]  $\vec U \pm c^{-1}{{\vec F_i (\vec U)}} \in {\mathcal G}\cup \partial {\mathcal G}$
  and $\vec U \pm {\alpha}^{-1}{{\vec F_i (\vec U)}} \in {\mathcal G}$
  for any $\alpha > c=1$, $i=1, \cdots, d$,
 where $\partial {\mathcal G}$ denotes the boundary of ${\mathcal G}$.
\end{itemize}
\end{lemma}

\begin{proof}
The proof of the properties (i) and (ii)  is   direct and easy  via
the definition  of $\hat{\mathcal G}$ and omitted here.
The following task is to prove the property (iii).

For any given $i \in \{1,2,\cdots,d\}$, if using $D^ \pm,\vec m^ \pm$, and $E^ \pm $ to denote
three  component of the vector
$\vec U \pm  c^{-1} {{\vec F_i (\vec U)}} $, then it is convenient to yield
\begin{align*}
D^ \pm  &= D\left( {1 \pm v_i} \right) > 0,
\\[1mm]
 E^ \pm  & = E \pm m_i
 =(\rho hW^2-p)\pm \rho h W^2v_i \ge \rho h W^2 ( 1-|v_i| ) -p \\[2mm]
 & \ge  \frac{\rho h} {1+v} - p  > \frac{\rho h} { 2} -p  \overset {\eqref{eq:hcondition1}} {\ge} \frac{1}{2}\big( \sqrt{\rho^2+p^2} - p \big) >0.
\end{align*}
Further using {\eqref{eq:hcondition1}} gives
\begin{align*}
 \left( {D^ \pm  } \right)^2  + |\vec m^ \pm  |^2  - \left( {E^ \pm  } \right)^2
&  = \left( {1 \pm v_i } \right)^2 W^2 \left[ {\rho ^2  + p^2  - \left( \rho+\rho e \right)^2 } \right]  {\le}  0.
\end{align*}
It follows that $q(\vec U^{\pm}) \ge 0$, and  $\vec U \pm  c^{-1} {{\vec F_i (\vec U)}} \in \hat{\cal G} \cup  \pt \hat{\cal G}$.
On the other hand, for any~$\alpha > c=1$, using the convexity of ${\cal G}$ {and}
the above result gives
$$
\vec U \pm {\alpha}^{-1}{{\vec F_i (\vec U)}} = \left(1-\frac{c}{\alpha} \right) \vec U + \frac{c}{\alpha} \vec U^\pm \in \hat{\cal G}.
$$
The proof is completed.
\end{proof}

\section{Numerical methods}
\label{sec:scheme}

This section begins   to develop PCP central discontinuous Galerkin methods for the 1D and 2D special RHD equations \eqref{eqn:coneqn3d}.

\subsection{1D case}\label{sec:1Dscheme}

For the sake of convenience, this subsection will use the symbol $x$ to replace the independent variable $x_1$ in \eqref{eqn:coneqn3d}.
Let $\{ I_j:=( x_{j-\frac12},x_{j+\frac12})\}$
be a uniform partition of the 1D spatial domain $\Omega$ with a constant spatial step-size
$\Delta x=x_{j+\frac12}-x_{j-\frac12}$. With $x_j=\frac12(x_{j+\frac12}+x_{j-\frac12})$,
define a dual partition $\{ J_{j+\frac12}:=\left( x_j,x_{j+1}\right)\}$.
The central DG methods seek two approximate solutions
$\vec U_h^I(t,x)$ and $\vec U_h^J(t,x)$  on those mutually dual meshes $\{ I_j\}$ and
 $\{ J_{j+\frac12}\}$,
 where for each  $t \in (0,T_f]$,
each component of $\vec U_h^I$ (resp. $\vec U_h^J$) belongs to the finite dimensional space of discontinuous piecewise polynomial functions, ${\cal V}_h^I$ (resp. ${\cal V}_h^J$), defined by
\begin{align*}
&{\cal V}_h^I :=\left\{ \left. w(x) \in L^1(\Omega) \right|  w(x) |_{I_j} \in  {\mathbb{P}}^K (I_j)  \right\} ,\\
&{\cal V}_h^J :=\left\{ \left. w(x) \in L^1(\Omega) \right|  w(x) |_{J_{j+\frac12}} \in  {\mathbb{P}}^K (J_{j+\frac12})  \right\} ,
\end{align*}
here ${\mathbb{P}}^K (I_j)$ and $ {\mathbb{P}}^K (J_{j+\frac12}) $ denote two spaces of polynomial of degree at most $K$  on the cells $I_j$ and  $J_{j+\frac12}$, respectively, and
 $K$ is assumed to be a constant over the whole meshes.

Consider the central DG spatial discretization for $\vec U_h^I$.
Using  a test function $w(x)\in  {\mathbb{P}}^K (I_j)$ to multiply \eqref{eqn:coneqn3d} with $d=1$
and integrating by parts over the cell $I_j$ give
\begin{align}
\displaystyle
\frac{\rm d}{{\rm d}t} \int_{I_j }  \vec Uw {\rm d}x =  \int_{I_j } \vec F_1 \left( \vec U\right) \frac{{\rm d}w}{{\rm d}x}  {\rm d}x
+ \vec F_1 \left( \vec U(t,x_{j-\frac12})\right) w(x_{j-\frac12})
- \vec F_1 \left( \vec U(t,x_{j+\frac12})\right) w(x_{j+\frac12}) . \label{eq:integrating}
\end{align}
Different from the standard DG discretization, the central DG discretization on the mesh $\{I_j\}$
(resp.  $\{ J_{j+\frac12}\}$) use its dual solution  $\vec U_h^J$ (resp. $\vec U_h^I$)  to compute the volume and surface integrals related to the flux $\vec F$. 
Specifically,  replacing the exact solution $\vec U$ at the left- and right-hand sides of \eqref{eq:integrating} with  the approximate solutions $\vec U_h^I$ and $\vec U_h^J$, respectively, gives
\begin{align}\nonumber
\displaystyle
\frac{\rm d}{{\rm d}t} \int_{I_j } \vec U_h^I w {\rm d}x = &  \frac{1}{\tau_{\max}} \int_{I_j } \left( \vec U_h^J - \vec U_h^I \right) w {\rm d} x +  \int_{I_j } \vec F_1 \left( \vec U_h^J \right) \frac{{\rm d}w}{{\rm d}x}  {\rm d}x \\[2mm]
&+ \vec F_1 \left( \vec U_h^J (t,x_{j-\frac12})\right) w(x_{j-\frac12})
- \vec F_1 \left( \vec U_h^J (t,x_{j+\frac12})\right) w(x_{j+\frac12}), \label{eq:semi}
\end{align}
where the first term at the right-hand side is an additional numerical dissipation term
and important for the stability of  central DG methods \citep{Liu2008},
and $\tau_{\max}$ is the maximum time stepsize allowed by the CFL condition \citep{Liu2007}.
The resulting central DG discretization   \eqref{eq:semi}
does  not need numerical fluxes based on exact or approximate Riemann solvers,
since the solutions or fluxes are evaluated  at the cell  interface $x_{j\pm \frac12}$,
i.e. the centers of  dual cell  $ J_{j\pm \frac12}$,  where  the solutions $\vec U_h^J$ are
 continuous.
Due to the possible discontinuity of $\vec U_h^J$ at $x=x_j$, the second integration at the right-hand side of \eqref{eq:semi} is usually split into two parts
\begin{equation}\label{eq:Int2Parts}
\int_{I_j } \vec F_1 \left( \vec U_h^J \right) \frac{{\rm d}w}{{\rm d}x}  {\rm d}x =
\int_{x_{j-\frac12} }^{x_j} \vec F_1 \left( \vec U_h^J \right) \frac{{\rm d}w}{{\rm d}x}  {\rm d}x
+ \int_{x_{j} }^{x_{j+\frac12}} \vec F_1 \left( \vec U_h^J \right) \frac{{\rm d}w}{{\rm d}x}  {\rm d}x,
\end{equation}
which may    be evaluated approximately by numerical quadrature.

Let $\left\{ \Phi_j^{(\mu)}(x)\right\}_{\mu=0}^K$ denote a local orthogonal basis
of the polynomial space ${\mathbb {P}}^K(I_j)$,
and express the DG approximate solution $\vec U_h^I$  as
\begin{equation}\label{eq:expressI}
\vec U_h^I(t,x) = \sum\limits_{\mu = 0}^{K} \vec U_j^{I,(\mu)} (t)  \Phi_j^{(\mu)}(x) =:  \vec U_j^I(t,x), \quad  ~x \in I_j.
\end{equation}
If substituting \eqref{eq:expressI} into \eqref{eq:semi}, taking the test function $w(x)$ as $\Phi_j^{(\nu)}(x), \nu=0,1,\cdots,K$, respectively, and applying a $Q$-point Gaussian quadrature to the integrations in \eqref{eq:Int2Parts}, then the  semi-discrete central DG discretization on the mesh $\{I_j\} $
may be reformed as follows
\begin{align} \nonumber
& \sum\limits_{\mu = 0}^{K} \left( \int_{I_j} \Phi_j^{(\mu)} (x) \Phi_j^{(\nu)} (x) {\rm d}x  \right) \frac{{\rm d}\vec U_j^{I,(\mu)} (t)}{{\rm d}t}
  = \frac{1}{\tau_{\max}} \int_{I_j } \left( \vec U_h^J - \vec U_h^I \right)  \Phi_j^{(\nu)} (x) {\rm d} x \\[2mm] \nonumber
  & \quad  + \frac{\Delta x}{2} \sum\limits_{\alpha = 1}^{Q} \omega_{\alpha} \left(
  \vec F_1 \left( \vec U_h^J (t,x_{j-\frac14}^{\alpha}) \right) \frac{{\rm d}\Phi_j^{(\nu)}}{{\rm d}x} (x_{j-\frac14}^{\alpha})
  + \vec F_1 \left( \vec U_h^J (t,x_{j+\frac14}^{\alpha}) \right) \frac{{\rm d}\Phi_j^{(\nu)}}{{\rm d}x} (x_{j+\frac14}^{\alpha})
   \right)\\[2mm] \label{eq:schemeI:1D}
   & \quad + \vec F_1 \left( \vec U_h^J (t,x_{j-\frac12})\right) \Phi_j^{(\nu)}(x_{j-\frac12})
- \vec F_1 \left( \vec U_h^J (t,x_{j+\frac12})\right) \Phi_j^{(\nu)}(x_{j+\frac12}),\quad \nu=0,\cdots,K,
\end{align}
where $\big\{ x_{j\pm \frac14}^\alpha \big\}_{\alpha=1}^Q$ denote the Gaussian nodes transformed into the
interval $\big[x_{j\pm \frac14}-\frac{\Delta x}{4}, x_{j\pm \frac14}+\frac{\Delta x}{4} \big]$,
and the associated Gaussian quadrature weights $\{ \omega_\alpha \}_{\alpha=1}^Q$
satisfy    $\omega_\alpha>0$ and  $\sum\limits_{\alpha=1}^Q \omega_\alpha =1$. For the accuracy requirement, $Q$ should satisfy
$Q\ge K+1$ for the $\mathbb{P}^K$-based DG methods \citep{CockburnIII}.

The central DG spatial discretization for $\vec U_h^J$ is very similar. If using $\left\{ \Phi_{j+\frac12}^{(\mu)}(x)\right\}_{\mu=0}^K$ to denote a local orthogonal basis of the polynomial space ${\mathbb {P}}^K(J_{j+\frac12})$,
and expressing the DG approximate solution $\vec U_h^J$  as
\begin{equation}\label{eq:expressJ}
\vec U_h^J(t,x) = \sum\limits_{\mu = 0}^{K} \vec U_{j+\frac12}^{J,(\mu)} (t)  \Phi_{j+\frac12}^{(\mu)}(x)
=:  \vec U_{j+\frac12}^J(t,x), \quad  ~x \in J_{j+\frac12},
\end{equation}
then the semi-discrete central DG discretization  on the mesh $\{I_{j+\frac12}\} $  reads
\begin{align} \nonumber
& \sum\limits_{\mu = 0}^{K} \left( \int_{J_{j+\frac12}} \Phi_{j+\frac12}^{(\mu)} (x) \Phi_{j+\frac12}^{(\nu)} (x) {\rm d}x  \right) \frac{{\rm d}\vec U_{j+\frac12}^{J,(\mu)} (t)}{{\rm d}t}
  = \frac{1}{\tau_{\max}} \int_{J_{j+\frac12} } \left( \vec U_h^I - \vec U_h^J \right)  \Phi_{j+\frac12}^{(\nu)} (x) {\rm d} x \\[2mm] \nonumber
  & \quad  + \frac{\Delta x}{2} \sum\limits_{\alpha = 1}^{Q} \omega_{\alpha} \left(
  \vec F_1 \left( \vec U_h^I (t,x_{j+\frac14}^{\alpha}) \right) \frac{{\rm d}\Phi_{j+\frac12}^{(\nu)}}{{\rm d}x} (x_{j+\frac14}^{\alpha})
  + \vec F_1 \left( \vec U_h^I (t,x_{j+\frac34}^{\alpha}) \right) \frac{{\rm d}\Phi_{j+\frac12}^{(\nu)}}{{\rm d}x} (x_{j+\frac34}^{\alpha})
   \right)\\[2mm] \label{eq:schemeJ:1D}
   & \quad + \vec F_1 \left( \vec U_h^I (t,x_{j})\right) \Phi_{j+\frac12}^{(\nu)}(x_{j})
- \vec F_1 \left( \vec U_h^I (t,x_{j+1})\right) \Phi_{j+\frac12}^{(\nu)}(x_{j+1}),\quad \nu=0,\cdots,K.
\end{align}

If taking the bases as the scaled Legendre polynomials, e.g.
\begin{align*}
&\Phi_j^{(0)}(x) = 1,\quad \Phi_j^{(1)}(x) =  \frac{x-x_j}{\Delta x},\quad \Phi_j^{(2)}(x) = 12 \left( \frac{x-x_j}{\Delta x} \right)^2 -1, \cdots,\\
&\Phi_{j+\frac12}^{(0)}(x) = 1,\quad \Phi_j^{(1)}(x) =  \frac{x-x_{j+\frac12}}{\Delta x},\quad \Phi_{j+\frac12}^{(2)}(x) = 12 \left( \frac{x-x_{j+\frac12}}{\Delta x} \right)^2 -1, \cdots,
\end{align*}
then from \eqref{eq:schemeI:1D} and \eqref{eq:schemeJ:1D} with $\nu=0$,
one may derive the  evolution equations for the cell-averages of $\vec U_h^I$ and $\vec U_h^J$
as follows
\begin{align} \nonumber
 \frac{{\rm d}\vec U_j^{I,(0)} (t)}{{\rm d}t}
&  = \frac{1}{\Delta x} \left( \frac{1}{\tau_{\max}} \int_{I_j } \left( \vec U_h^J - \vec U_h^I \right)   {\rm d} x
   + \vec F_1 \left( \vec U_h^J (t,x_{j-\frac12})\right)
- \vec F_1 \left( \vec U_h^J (t,x_{j+\frac12})\right) \right)\\
& =: {\mathscr{L}}^I_j (\vec U_h^I,\vec U_h^J),
\end{align} \label{eq:schemeI:1D:ave}
and
\begin{align} \nonumber
 \frac{{\rm d}\vec U_{j+\frac12}^{J,(0)} (t)}{{\rm d}t}
 & = \frac{1}{\Delta x}\left( \frac{1}{\tau_{\max}} \int_{J_{j+\frac12} } \left( \vec U_h^I - \vec U_h^J \right)   {\rm d} x
 + \vec F_1 \left( \vec U_h^I (t,x_{j})\right)
- \vec F_1 \left( \vec U_h^I (t,x_{j+1})\right) \right) \\
& =: {\mathscr{L}}^J_{j+\frac12} (\vec U_h^J,\vec U_h^I).
\end{align}\label{eq:schemeJ:1D:ave}

Eqs. \eqref{eq:schemeI:1D} and \eqref{eq:schemeJ:1D} constitute
a nonlinear system of ordinary differential equations for $\vec U_j^{I,(\mu)} (t)$ and
$\vec U_{j+\frac12}^{J,(\mu)} (t)$, and may be rewritten into a compact form $\vec U'(t) = {{\mathscr L}} (\vec U)$.
The strong stability preserving (SSP) Runge-Kutta methods or multi-step methods \citep{Gottlieb2009} may be  further taken for the time discretization in order to obtain the fully discrete central DG methods.
For example, the third-order accurate SSP Runge-Kutta method
\begin{align} \label{eq:RK1} \begin{aligned}
& \vec U^ *   = \vec U^n  + \Delta t \mathscr{L} (\vec U^n ), \\[2mm]
& \vec U^{ *  * }  = \frac{3}{4}\vec U^n  + \frac{1}{4}\Big(\vec U^ *
 + \Delta t \mathscr{L}  (\vec U^ *  )\Big), \\[2mm]
& \vec U^{n+1}  = \frac{1}{3} \vec U^n  + \frac{2}{3}\Big(\vec U^{ *  * }
+ \Delta t  \mathscr{L}  (\vec U^{ *  * })\Big),
\end{aligned}\end{align}
and  the third-order accurate SSP multi-step method
\begin{equation}\label{eq:multi-step}
\vec U^{n+1} = \frac{16}{27} \left( \vec U^n  +3 \Delta t {\mathscr{L}}(\vec U^n ) \right) + \frac{11}{27} \left( \vec U^{n-3}  +\frac{12}{11} \Delta t {\mathscr{L}}(\vec U^{n-3} )  \right),
\end{equation}
where $\Delta t$ denotes the time stepsize in computations.



When $K =0$,  the above central DG methods reduce to
corresponding  first-order accurate 
central schemes on overlapping cells.

\begin{thm}\label{thm:1DK=0}
If	$K=0$ and $\vec U_j^I ,\vec U_{j+\frac12}^J \in {\cal G}$ for all $j$, 	
 then under the CFL type condition
 	\begin{equation}\label{eq:CFL-K=0}
 		0< \Delta t < \frac{ \theta \Delta x} {2c},\quad \theta := \frac{\Delta t}{\tau_{\max}}\in (0,1],
 	\end{equation}
 	one has
	\begin{equation*}
		   \vec U_j^I + \Delta t {\mathscr{L}}^I_j (\vec U_h^I,\vec U_h^J)  \in {\cal G},\quad \vec U_{j+\frac12}^J + \Delta t {\mathscr{L}}^J_{j+\frac12} (\vec U_h^J,\vec U_h^I) \in {\cal G},
	\end{equation*}
	for all $j$.
\end{thm}

\begin{proof}
Because both $\vec U_j^I$ and $\vec U_{j+\frac12}^J$ are constant vectors when  $K=0$,
one has
\begin{align} \nonumber
  \vec U_j^I + \Delta t {\mathscr{L}}^I_j (\vec U_h^I,\vec U_h^J)
& = \vec U_j^I + \frac{\Delta t}{\Delta x} \left( \frac{1}{\tau_{\max}} \int_{I_j } \left( \vec U_h^J - \vec U_j^I \right)   {\rm d} x
   + \vec F_1 \left( \vec U_{j-\frac12}^J \right)
- \vec F_1 \left( \vec U_{j+\frac12}^J \right)  \right)\\[2mm]   \nonumber
& = (1-\theta) \vec U_j^I + \frac{\theta}{2}\left( \vec U_{j+\frac12}^J + \vec U_{j-\frac12}^J \right) + \frac{\Delta t}{\Delta x} \left(
 \vec F_1 \left( \vec U_{j-\frac12}^J \right)
- \vec F_1 \left( \vec U_{j+\frac12}^J \right)  \right)\\[2mm]
& = (1-\theta) \vec U_j^I + \frac{\theta}{2} \vec U_{j+\frac12}^{J,-} + \frac{\theta}{2} \vec U_{j-\frac12}^{J,+}, \label{eq:convex:1D:K=0}
\end{align}
where
$$
\vec U_{j\pm\frac12}^{J,\mp} :=  \vec U_{j\pm\frac12}^J  \mp \left( \frac{\theta \Delta x}{2 \Delta t} \right)^{-1} \vec F_1 \left( \vec U_{j\pm\frac12}^J \right).
$$
Thanks to the Lax-Friedrichs splitting property in Lemma \ref{lam:propertyG},
 $\vec U_{j\pm\frac12}^{J,\mp} \in {\cal G}$  under the theorem hypothesis.
Combing those with  \eqref{eq:convex:1D:K=0}  and using
 the convexity of $\cal G$
 further yields $ \vec U_j^I + \Delta t {\mathscr{L}}^I_j (\vec U_h^I,\vec U_h^J) \in {\cal G}$.
  Similar arguments may show $\vec U_{j+\frac12}^J + \Delta t {\mathscr{L}}^J_{j+\frac12} (\vec U_h^J,\vec U_h^I) \in {\cal G}$. The proof is completed.
\end{proof}

Theorem \ref{thm:1DK=0} indicates that the first-order accurate ($K = 0$) central DG methods are
 PCP under the CFL type condition \eqref{eq:CFL-K=0} if the forward Euler method is used for time discretization.


When $K \ge 1$, the high-order accurate central DG methods may
work well for the 1D RHD problems whose solutions are either
smooth or contain weak discontinuities and do not involve low density or pressure and large Lorentz factor.
However, if the solution contains strong discontinuity,
the  high-order accurate central DG methods will generate significant spurious  oscillations
and even nonlinear instability.
Therefor, it is necessary to use some nonlinear limiter
to suppress or control possible spurious  oscillations.
Up to now, there exist some nonlinear limiters for the DG methods in the literature,
e.g.  the minmod-type limiter \citep{CockburnII},   moment-based limiter \citep{Biswas}, WENO limiter \citep{QiuWENOlimiter,Zhu2008,ZhaoTang2013,Zhao}, and so on.
Although those nonlinear limiters may effectively suppress spurious oscillations, they
cannot make the high-order accurate central DG methods become PCP
  in general.
To overcome such difficulty,
the positivity-preserving limiters  \citep{zhang2010b,ChengLiQiuXu} will
be extended to our central DG methods for the RHD equations:
 consider the scheme preserving the cell-averages  $\vec U^{I,(0)}_j(t)$ and $\vec U^{J,(0)}_{j+\frac12}(t)$ in $\cal G$, and then use those cell-averages
 to limit the polynomial vector $\vec U^{I}_j(t,x)$
 (resp. $\vec U^{J}_{j+\frac12}(t,x)$) as $\tilde{\vec U}^{I}_j(t,x)$
(resp. $\tilde{\vec U}^{J}_{j+\frac12}(t,x)$)
such that the values of $\tilde{\vec U}^{I}_j(t,x)$
(resp. $\tilde{\vec U}^{J}_{j+\frac12}(t,x)$)  at some critical
points in the cell $I_j$ (resp. $J_{j+\frac12}$) belong to $\cal G$.

Before presenting the  positivity-preserving limiter,
the PCP conditions for the 1D high-order accurate central DG methods is first studied.
For the sake of convenience, the independent variable $t$ will be temporarily omitted.
Let $\big\{ {\hat x}_{j\pm \frac14}^\alpha \big\}_{\alpha=1}^L$ be the Gauss-Lobatto nodes transformed into the
interval $\big[x_{j\pm \frac14}-\frac{\Delta x}{4}, x_{j\pm \frac14}+\frac{\Delta x}{4} \big]$, and $\{ {\hat{\omega}}_\alpha \}_{\alpha=1}^L$ be the associated Gaussian quadrature weights satisfying
${\hat{ \omega}}_\alpha>0$ and $\sum\limits_{\alpha=1}^L {\hat{ \omega}}_\alpha =1$,
where  $L$ is larger than  $(K+3)/2$
in order to ensure that the algebraic precision of corresponding quadrature rule  is at least $K$.

\begin{thm} \label{thm:PCP:1DRHD}
If 
${\vec U}^I_j(\hat x_{j\pm\frac14}^\alpha)\in {\cal G}$ and ${\vec U}^J_{j+\frac12}(\hat x_{j+\frac12 \pm\frac14}^\alpha)\in {\cal G}$ for all $j$ and $\alpha=1,2,\cdots,L$, then under the CFL type condition
 	\begin{equation}\label{eq:CFL-Kge1}
 		0< \Delta t \le \frac{ {\hat{\omega}}_1 \theta \Delta x} {2c},\quad \theta \in (0,1],
 	\end{equation}
 	one has
	\begin{equation*}
		   \vec U_j^{I,(0)} + \Delta t {\mathscr{L}}^I_j (\vec U_h^I,\vec U_h^J)  \in {\cal G},\quad \vec U_{j+\frac12}^{J,(0)} + \Delta t {\mathscr{L}}^J_{j+\frac12} (\vec U_h^J,\vec U_h^I) \in {\cal G},
	\end{equation*}
	for all $j$.
\end{thm}

\begin{proof}
Using the convexity of $\cal G$ and the exactness of the Gauss-Lobatto quadrature rule with $L$ nodes for the polynomials of degree $K$ yields
$$
\vec U_j^{I,(0)} = \frac{1}{\Delta x} \int_{I_j} \vec U_h^I dx
= \frac{1}{2} \sum \limits_{\alpha=1}^L \hat \omega_\alpha  \Big( \vec U_j^I (\hat x_{j-\frac14}^\alpha ) + \vec U_j^I (\hat x_{j+\frac14}^\alpha ) \Big) \in {\cal G},
$$
and
\begin{align*}
& \frac{1}{\Delta x} \int_{I_j} \vec U_h^J dx
= \frac{1}{2} \sum \limits_{\alpha=1}^L \hat \omega_\alpha  \Big( \vec U_{j-\frac12}^J (\hat x_{j-\frac14}^\alpha ) + \vec U_{j+\frac12}^J (\hat x_{j+\frac14}^\alpha ) \Big) \\
 & = \frac{\hat \omega_1}{2} \vec U_{j-\frac12}^J ( x_{j-\frac12} )
+ \frac{\hat \omega_L}{2} \vec U_{j+\frac12}^J ( x_{j+\frac12} )
+ \frac{1}{2} \sum \limits_{\alpha=2}^L \hat \omega_\alpha   \vec U_{j-\frac12}^J (\hat x_{j-\frac14}^\alpha )
+  \frac{1}{2} \sum \limits_{\alpha=1}^{L-1} \hat \omega_\alpha   \vec U_{j+\frac12}^J (\hat x_{j+\frac14}^\alpha )\\
&= \frac{\hat \omega_1}{2} \vec U_{j-\frac12}^J ( x_{j-\frac12} )
+ \frac{\hat \omega_1}{2} \vec U_{j+\frac12}^J ( x_{j+\frac12} ) + \left(1- \hat \omega_1 \right) \Xi  ,
\end{align*}
with
$$
\Xi := \frac{1}{2(1-\hat \omega_1 )} \left(  \sum \limits_{\alpha=2}^L \hat \omega_\alpha   \vec U_{j-\frac12}^J (\hat x_{j-\frac14}^\alpha )  + \sum \limits_{\alpha=1}^{L-1} \hat \omega_\alpha   \vec U_{j+\frac12}^J (\hat x_{j+\frac14}^\alpha ) \right) \in {\cal G},
$$
where $\hat \omega_1 = \hat \omega_L \le \frac12$.
Thus one has
\begin{align} \nonumber
 \vec U_j^{I,(0)} +  \Delta t {\mathscr{L}}^I_j (\vec U_h^I,\vec U_h^J)
& =  \vec U_j^{I,(0)}  +  \frac{\theta}{\Delta x}\nonumber\int_{I_j } \left( \vec U_h^J - \vec U_j^I \right)   {\rm d} x
\\ \nonumber
& \quad + \frac{\Delta t}{\Delta x} \left(
\vec F_1 \left( \vec U_{j-\frac12}^J ( x_{j-\frac12} )   \right)
- \vec F_1 \left( \vec U_{j+\frac12}^J ( x_{j+\frac12} )  \right)  \right)\\[2mm]   \nonumber
& = (1-\theta) \vec U_j^{I,(0)}
+ \theta \left( \frac{\hat \omega_1}{2} \vec U_{j-\frac12}^J ( x_{j-\frac12} ) + \frac{\hat \omega_1}{2} \vec U_{j+\frac12}^J ( x_{j+\frac12} ) + \left(1- \hat \omega_1 \right) \Xi  \right) \\[2mm]\nonumber
&\quad+ \frac{\Delta t}{\Delta x} \left(
\vec F_1 \left( \vec U_{j-\frac12}^J ( x_{j-\frac12} )   \right)
- \vec F_1 \left( \vec U_{j+\frac12}^J ( x_{j+\frac12} )  \right)  \right)
\\[2mm]
& = (1-\theta) \vec U_j^{I,(0)}  + \left(1- \hat \omega_1 \right) \theta \Xi +
\frac{\hat \omega_1 \theta}{2} \vec U_{j+\frac12}^{J,-} + \frac{\hat \omega_1 \theta}{2} \vec U_{j-\frac12}^{J,+}, \label{eq:convex:1D:Kge1}
\end{align}
where
$$
\vec U_{j\pm\frac12}^{J,\mp} :=  \vec U_{j\pm\frac12}^J ( x_{j \pm \frac12} )   \mp \left( \frac{ \hat \omega_1 \theta \Delta x}{2 \Delta t} \right)^{-1} \vec F_1 \left( \vec U_{j\pm\frac12}^J ( x_{j \pm \frac12} )  \right) \in {\cal G} \cup \partial {\cal G},
$$
due to  the Lax-Friedrichs splitting property in Lemma \ref{lam:propertyG}
and the theorem hypothesis. 
Using   \eqref{eq:convex:1D:Kge1}  and the convexity of $\cal G$
may further  yield $ \vec U_j^{I,(0)} + \Delta t {\mathscr{L}}^I_j (\vec U_h^I,\vec U_h^J) \in {\cal G}$.
Similar arguments may show
 $\vec U_{j+\frac12}^{J,(0)} + \Delta t {\mathscr{L}}^J_{j+\frac12} (\vec U_h^J,\vec U_h^I) \in {\cal G}$. The proof is completed.
\end{proof}

Theorem \ref{thm:PCP:1DRHD} gives a sufficient condition for the 1D central DG methods
which preserve the cell-averages  $\vec U_j^{I,(0)}$ and $\vec U_{j+\frac12}^{J,(0)}$ in $\cal G$
 when the forward Euler method is used for the time discretization.
 Since a high-order accurate SSP time discretization may be considered as
a convex combination of the forward Euler method,
Theorem \ref{thm:PCP:1DRHD} is valid for the high-order accurate SSP time discretization.

Let us present the PCP limiting procedure, which
limits ${\vec U}^I_j(x)$ and ${\vec U}^J_{j+\frac12}(x)$
as $\tilde{\vec U}^I_j(x)$ and $\tilde{\vec U}^J_{j+\frac12}(x)$
satisfying two requirements: (i) $\tilde{\vec U}^I_j(\hat x_{j\pm\frac14}^\alpha)\in {\cal G}$ and $\tilde{\vec U}^J_{j+\frac12}(\hat x_{j+\frac12 \pm\frac14}^\alpha)\in {\cal G}$ for $\alpha=1,2,\cdots,L$,
and  (ii) $\tilde{\vec U}^I_j( x_{j\pm\frac14}^\alpha)\in {\cal G}$ and $\tilde{\vec U}^J_{j+\frac12}( x_{j+\frac12 \pm\frac14}^\alpha)\in {\cal G}$ for $\alpha=1,2,\cdots,Q$.
The second requirement   does not appear in the non-relativistic case and is used to
ensure getting a physical solution of the pressure equation  \eqref{eq:solvePgEOS} by root-finding method
and  the successive calculations of $\vec F_1 \left( {\vec U}^I_j(x_{j\pm\frac14}^\alpha) \right)$ and
$\vec F_1 \left( {\vec U}^J_{j+\frac12}( x_{j+\frac12 \pm\frac14}^\alpha) \right)$ in \eqref{eq:schemeI:1D} and \eqref{eq:schemeJ:1D}.
%
%
 %
%
%
%
%
Because the PCP limiting procedures for   ${\vec U}^I_j(x)$ and ${\vec U}^J_{j+\frac12}(x)$
are the same and implemented  separately,
only the PCP limiter for ${\vec U}^I_j(x)$ is presented here.
Let ${\vec U}^I_j(x)=:\left( D_j(x),{\vec m}_j(x), E_j(x)\right)^{\rm T}$,
assume that $\vec U_j^{I,(0)}=:\left( \overline D_j, \overline{\vec m}_j, \overline E_j\right)^{\rm T}\in {\cal G}$,   
and introduce a sufficiently small positive number $\epsilon$
(taken as $10^{-13}$ in numerical computations)
such that $ \vec U_j^{I,(0)}\in {\cal G}_\epsilon$, where
\begin{align*} 
{\cal G}_\epsilon = \left\{   \vec U=(D,\vec m,E)^{\rm T} \big|  D\ge\epsilon,~q(\vec U)\ge\epsilon\right\}.
\end{align*}
Obviously, ${\cal G}_\epsilon \subset {\cal G}_0$ and $\mathop {\lim }\limits_{\epsilon  \to 0^ +  }  {\cal G}_\epsilon = {\cal G}_0$.

 The {\tt 1D PCP limiting procedure} is divided into the following two steps.

\noindent
{\bf Step (i)}: Enforce the positivity of $D(\vec U)$. Let $D_{\min} = \min \limits_{x \in {\mathcal S}_j}^{} D_j ( x )$,
where
$${\mathcal S}_j:=\left\{\hat{x}_{j-\frac14}^\alpha\right\}_{\alpha=1}^L  \bigcup \left\{\hat{x}_{j+\frac14}^\alpha\right\}_{\alpha=1}^L \bigcup \left\{x_{j-\frac14}^\alpha\right\}_{\alpha=1}^Q  \bigcup \left\{{x}_{j+\frac14}^\alpha\right\}_{\alpha=1}^Q .$$
If $D_{\min} < \epsilon$,
then  $D_j ( x )$ is limited as
$$
\hat D_j(x) = \theta_1 \big( D_j(x) - \overline D_j \big) + \overline D_j,
$$
where $\theta_1 = (\overline D_j - \epsilon)/ ( \overline D_j - D_{\min} ) <1$. Otherwise, take $\hat D_j(x) =  D_j(x)$ and $\theta_1=1$.
Denote $\hat {\vec U}_j(x) :=  \left( \hat D_j(x), \vec m_j(x), E_j(x) \right)^{\rm T}$.

\noindent
{\bf Step (ii)}: Enforce the positivity of $q(\vec U)$. Let
$q_{\min} = \min \limits_{x \in {\mathcal S}_j}^{} q(\hat {\vec U}_j ( x ))$. If $q_{\min} < \epsilon$, then  $\hat {\vec U}_j ( x )$ is limited as
$$
\tilde {\vec U}^I_j(x) = \theta_2 \big( \hat {\vec U}_j (x) -  {\vec U}_j^{I,(0)} \big) + {\vec U}_j^{I,(0)},
$$
where $\theta_2 = \left( q \big(  {\vec U}_j^{I,(0)} \big) - \epsilon \right)/ \left( q \big(  {\vec U}_j^{I,(0)} \big) - q_{\min} \right) <1$. Otherwise, set $\tilde {\vec U}^I_j(x) =  \hat {\vec U}_j(x)$ and $\theta_2=1$.

\begin{lemma}\label{lam:limiter}
If $ {\vec U}_j^{I,(0)} \in {\cal G}_\epsilon$, then
  $\tilde {\vec U}_j^I(x)$ given by the above PCP limiting procedure belongs to $ {\cal G}_\epsilon$ for all $x \in {\mathcal S}_j$.
\end{lemma}

\begin{proof}
For any $x\in {\mathcal S}_j$,  
it is obvious that
$\hat D_j (x)= D_j (x)\geq D_{\min} >\epsilon$ if $D_{\min} >\epsilon$.
 If $D_{\min} >\epsilon$, then one has
$$\hat D_j (x)= \theta_1 \big( D_j(x) - \overline D_j \big) + \overline D_j
\geq  \theta_1 \big( D_{\min} - \overline D_j \big) + \overline D_j=\epsilon.
$$
Thanks to $\theta_2 \in[0,1]$,  one yields
\begin{align*}
 \tilde D_j^I(x) =  \theta_2 \left( \hat D_j( x) - \overline D_j \right) +  \overline D_j  \ge \theta_2 \left( \epsilon - \overline D_j \right) +  \overline D_j \ge \epsilon.
\end{align*}

Similarly, if $q_{\min} \ge \epsilon$,
 then it is evident that $q\big( \tilde { \vec U}^I_j(x) \big) = q\big( \hat { \vec U}_j(x) \big) \ge q_{\min} \ge \epsilon  $ for any $x\in {\mathcal S}_j$. Otherwise, using the concavity of $q(\vec U)$ gives
\begin{align*}
q\big( \tilde { \vec U}_j^I (x) \big) & = q\big(  \theta_2  \hat {\vec U}_j (x) + (1-\theta_2)  {\vec U}_j^{I,(0)} \big)
 \ge  \theta_2 q\big(  \hat {\vec U}_j (x) \big) + (1-\theta_2) q \big(  {\vec U}_j^{I,(0)} \big)\\
 & \ge \theta_2 q_{\min} + (1-\theta_2) q \big(  {\vec U}_j^{I,(0)} \big) = \epsilon.
\end{align*}
The proof is completed.

\end{proof}

The above PCP limiting procedure preserves the conservation in the  sense that
$$
 {\vec U}_j^{I,(0)} = \frac{1}{\Delta x} \int_{I_j} {\vec U}^I_j( x)  dx = \frac{1}{\Delta x} \int_{I_j} \hat{\vec U}_j( x)  dx
= \frac{1}{\Delta x} \int_{I_j} \tilde{\vec U}^I_j( x)  dx,
$$
and maintains the high-order accuracy for smooth solutions, similar to the discussion at the end of Section 2.2 of \citep{zhang2010b}.
If replacing the solution polynomials  ${\vec U}^I_j(x)$ and ${\vec U}^J_{j+\frac12}(x)$
of high-order accurate central DG methods
with the limited polynomials $\tilde{\vec U}^I_j(x)$ and $\tilde{\vec U}^J_{j+\frac12}(x)$  at each stage of SSP Runge-Kutta method \eqref{eq:RK1} or each step of SSP muti-step method \eqref{eq:multi-step},
then the resulting fully discrete central DG methods are PCP under some CFL type conditions.

\begin{thm}\label{thm:1D}
If  the high-order accurate central DG solution polynomials are revised as the above limited polynomials at each stage of SSP Runge-Kutta method \eqref{eq:RK1} or each step of SSP muti-step method \eqref{eq:multi-step},
then
(i) the resulting Runge-Kutta central DG methods are PCP under the CFL type condition \eqref{eq:CFL-Kge1},
(ii) 
the resulting multi-step central DG methods are PCP under the CFL type condition
 	\begin{equation}\label{eq:CFL-Kge1MS}
 		0< \Delta t \le \frac{ {\hat{\omega}}_1 \theta \Delta x} {2c},\quad \theta\in \left(0,\frac13 \right].
 	\end{equation}
\end{thm}

Similar to   \citep{wang2012,ChengLiQiuXu,QinShu2016}, one may yield
 the $L^1$-stability of the proposed PCP central DG methods.

\begin{thm}\label{thm:stability}
Under the vanishing, reflective or periodic boundary conditions,
the PCP central DG methods are the $L^1$-stable in the sense that
$$
\left\| \tilde{\vec U}_h^{I} (t_n,x) \right\|_{L^1} + \left\| \tilde{ \vec U}_h^{J} (t_n,x) \right\|_{L^1}
<
2\left( \left\| \vec U_h^{I} (0,x) \right\|_{L^1} + \left\| \vec U_h^{J} (0,x) \right\|_{L^1} \right),
$$
where
$$
\left\|\vec w^I \right\|_{L^1} := \frac {\Delta x} {2} \sum \limits_{j} \sum \limits_{\alpha=1}^L \hat {\omega}_\alpha \left( \left\| \vec w_j^I(\hat x^\alpha_{j-\frac14}) \right\|_{l^1} + \left\| \vec w_j^I(\hat x^\alpha_{j+\frac14}) \right\|_{l^1} \right) \approx \int_{\Omega} \left\|\vec w^I(x) \right\|_{l^1} {\rm d}x,
$$
and
$$
\left\|\vec w^J \right\|_{L^1} :=  \frac {\Delta x} {2} \sum \limits_{j} \sum \limits_{\alpha=1}^L  \hat {\omega}_\alpha \left( \left\| \vec w_{j+\frac12}^J(\hat x^\alpha_{j+\frac14}) \right\|_{l^1} + \left\| \vec w_{j+\frac12}^J(\hat x^\alpha_{j+\frac34}) \right\|_{l^1} \right) \approx \int_{\Omega} \left\|\vec w^J(x) \right\|_{l^1} {\rm d}x.
$$
\end{thm}

\begin{proof}
It only needs to consider the forward Euler time discretization.
Because $ \tilde D_j^I(t_n,\hat x^\alpha_{j\pm\frac14})$ are larger than zero
and   the central DG methods
are conservative, one yields
\begin{align} \nonumber
\left\| \tilde{D}_h^{I} (t_n,x) \right\|_{L^1}& = \frac {\Delta x} {2} \sum \limits_{j} \sum \limits_{\alpha=1}^L \hat {\omega}_\alpha \left( \left| \tilde D_j^I(t_n,\hat x^\alpha_{j-\frac14}) \right| + \left| \tilde D_j^I(t_n,\hat x^\alpha_{j+\frac14}) \right| \right) \\ \nonumber
& = \frac {\Delta x} {2} \sum \limits_{j} \sum \limits_{\alpha=1}^L \hat {\omega}_\alpha \left(  \tilde D_j^I( t_n, \hat x^\alpha_{j-\frac14})  +  \tilde D_j^I(t_n, \hat x^\alpha_{ j+\frac14}) \right)\\ \nonumber
& = \sum \limits_{j} \int_{I_j} \tilde D_j^I(t_n,x) {\rm d}x =  \sum \limits_{j} \int_{I_j}  D_j^I(t_n,x) {\rm d}x = \Delta x \sum \limits_{j} D_j^{I,(0)}(t_n) \\ \nonumber
& = \Delta x \sum \limits_{j} \left( D_j^{I,(0)} (t_{n-1}) +  \Delta t {\mathscr{L}}^{I,D}_j ( \tilde {\vec U}_h^I(t_{n-1},x),\tilde{\vec U}_h^J(t_{n-1},x) \right)\\ \nonumber
& = \sum \limits_{j} \left(  \Delta x \tilde  D_j^{I,(0)} (t_{n-1}) +  \theta \int_{I_j } \left( \tilde {D}_h^J (t_{n-1},x) - \tilde {D}_j^I (t_{n-1},x) \right)   {\rm d} x \right)\\
& = (1-\theta)\left\| \tilde{D}_h^{I} (t_{n-1},x) \right\|_{L^1} + \theta \left\| \tilde{D}_h^{J} (t_{n-1},x) \right\|_{L^1}, \label{eq:wkl-01}
\end{align}
where  ${\mathscr{L}}^{I,D}_j$ denotes the first component of ${\mathscr{L}}^{I}_j$.
Similarly, one has
\begin{equation}\label{eq:wkl-02}
\left\| \tilde{D}_h^{J} (t_n,x) \right\|_{L^1} = (1-\theta)\left\| \tilde{D}_h^{J} (t_{n-1},x) \right\|_{L^1} + \theta \left\| \tilde{D}_h^{I} (t_{n-1},x) \right\|_{L^1}.
\end{equation}
Combining \eqref{eq:wkl-01} with \eqref{eq:wkl-02} gives
\begin{align*}
\left\| \tilde{D}_h^{I} (t_n,x) \right\|_{L^1}   &+ \left\| \tilde{D}_h^{J} (t_n,x) \right\|_{L^1}= \left\| \tilde{D}_h^{I} (t_{n-1},x) \right\|_{L^1}  + \left\| \tilde{D}_h^{J} (t_{n-1},x) \right\|_{L^1}   \\
=&  \cdots=\left\| \tilde{D}_h^{I} (0,x) \right\|_{L^1}  + \left\| \tilde{D}_h^{J} (0,x) \right\|_{L^1}
= \left\| {D}_h^{I} (0,x) \right\|_{L^1}  + \left\| {D}_h^{J} (0,x) \right\|_{L^1}.
\end{align*}
Similar argument may get
\begin{align*}
\left\| \tilde{E}_h^{I} (t_n,x) \right\|_{L^1}  + \left\| \tilde{E}_h^{J} (t_n,x) \right\|_{L^1} &= \left\| \tilde{E}_h^{I} (t_{n-1},x) \right\|_{L^1}  + \left\| \tilde{E}_h^{J} (t_{n-1},x) \right\|_{L^1}   \\
&= \cdots
= \left\| {E}_h^{I} (0,x) \right\|_{L^1}  + \left\| {E}_h^{J} (0,x) \right\|_{L^1}.
\end{align*}
Using $ q\left(\tilde {\vec U}_j^I(t_n ,\hat x^\alpha_{j\pm\frac14}) \right) >0$ gives
$$
\left| \left(\tilde m_1\right)_j^I(t_n ,\hat x^\alpha_{j\pm\frac14}) \right| < \tilde E_j^I(t_n ,\hat x^\alpha_{j\pm\frac14}) ,
$$
thus one has
$$
\left\| \left(\tilde m_1\right)_h^{I} (t_n,x) \right\|_{L^1} <  \left\| \tilde{E}_h^{I} (t_n,x) \right\|_{L^1}.
$$
Similarly, one also has
$$
\left\| \left(\tilde m_1\right)_h^{J} (t_n,x) \right\|_{L^1} <  \left\| \tilde{E}_h^{J} (t_n,x) \right\|_{L^1}.
$$
Therefore one has
\begin{align*}
&\left\| \tilde{\vec U}_h^{I} (t_n,x) \right\|_{L^1} + \left\| \tilde{ \vec U}_h^{J} (t_n,x) \right\|_{L^1}  \\
& < \left\| \tilde{D}_h^{I} (t_n,x) \right\|_{L^1} + \left\| \tilde{ D}_h^{J} (t_n,x) \right\|_{L^1} + 2\left(  \left\| \tilde{E}_h^{I} (t_n,x) \right\|_{L^1} +  \left\| \tilde{E}_h^{J} (t_n,x) \right\|_{L^1} \right) \\
&
= \left\| {D}_h^{I} (0,x) \right\|_{L^1} + \left\| { D}_h^{J} (0,x) \right\|_{L^1} + 2\left(  \left\| E_h^{I} (0,x) \right\|_{L^1} +  \left\| {E}_h^{J} (0,x) \right\|_{L^1} \right)
\\
& \le
2\left( \left\| \vec U_h^{I} (0,x) \right\|_{L^1} + \left\| \vec U_h^{J} (0,x) \right\|_{L^1} \right).
\end{align*}
The proof is completed.
\end{proof}

\subsection{2D case}
\label{sec:2Dscheme}

For the sake of convenience, this subsection will use
the symbol $\vec x=(x,y)$ to replace the independent variables $(x_1,x_2)$ in \eqref{eqn:coneqn3d}.
Let $\{ I_{i,j} = (x_{i-\frac{1}{2}},x_{i+\frac{1}{2}})\times
(y_{j-\frac{1}{2}},y_{j+\frac{1}{2}}) \}$ be a uniform partition of the 2D spatial domain $\Omega$ with a constant spatial step-sizes $\Delta x=x_{i+\frac12}-x_{i-\frac12}$ and $\Delta y=y_{j+\frac12}-y_{j-\frac12}$ in $x$ and $y$ directions respectively, and $\{J_{i+\frac12,j+\frac12}=\left( x_i,x_{i+1}\right)\times \left( y_j,y_{j+1}\right) \}$ be the dual partition.
The 2D central DG methods seek two approximate solutions $\vec U_h^I$ and $\vec U_h^J$ respectively defined on those  mutually dual meshes $\{ I_{i,j}\}$ and $\{J_{i+\frac12,j+\frac12}\}$,
where for each time $t \in (0,T_f]$,
each component of $\vec U_h^I$ (resp. $\vec U_h^J$)
belongs to the finite dimensional space
of discontinuous piecewise polynomial functions, ${\cal V}_h^I$ (resp. ${\cal V}_h^J$),
defined by
\begin{align*}
&{\cal V}_h^I :=\left\{ \left. w(\vec x) \in L^1(\Omega) \right|  w(\vec x) |_{I_{i,j}} \in  {\mathbb{P}}^K (I_{i,j})  \right\} ,\\
&{\cal V}_h^J :=\left\{ \left. w(\vec x) \in L^1(\Omega) \right|  w(\vec x) |_{J_{i+\frac12,j+\frac12}} \in  {\mathbb{P}}^K (J_{i+\frac12,j+\frac12})  \right\} ,
\end{align*}
here ${\mathbb{P}}^K (I_{i,j})$ and $ {\mathbb{P}}^K (J_{i+\frac12,j+\frac12}) $ denote two spaces of polynomial of degree at most $K$ on the cells $I_{i,j}$ and  $J_{i+\frac12,j+\frac12}$ respectively and
 their dimension is equal to $K_d :=(K+1)(K+2)/2$.

If letting $\left\{ \Phi_{i,j}^{(\mu)}(\vec x)\right\}_{\mu=0}^{K_d-1}$ and $\left\{ \Phi_{i+\frac12,j+\frac12}^{(\mu)}(\vec x)\right\}_{\mu=0}^{K_d-1}$ denote the local orthogonal bases of the spaces ${\mathbb {P}}^K(I_{i,j})$ and ${\mathbb {P}}^K(J_{i+\frac12,j+\frac12})$ respectively,
then  the central DG approximate solutions $\vec U_h^I$ and $\vec U_h^J$ may be expressed as
\begin{equation}\label{eq:expressI2D}
\vec U_h^I(t,\vec x) = \sum\limits_{\mu = 0}^{K_d-1} \vec U_{i,j}^{I,(\mu)} (t)  \Phi_{i,j}^{(\mu)}(\vec x) =:  \vec U_{i,j}^I(t,\vec x), \quad {}~\vec x \in I_{i,j},
\end{equation}
and
\begin{equation}\label{eq:expressJ2D}
\vec U_h^J(t,\vec x) = \sum\limits_{\mu = 0}^{K_d-1} \vec U_{i+\frac12,j+\frac12}^{J,(\mu)} (t)  \Phi_{i+\frac12,j+\frac12}^{(\mu)}(\vec x) =:  \vec U_{i+\frac12,j+\frac12}^J(t,\vec x), \quad {}~\vec x \in J_{i+\frac12,j+\frac12}.
\end{equation}
Similar to the 1D case, the semi-discrete 2D central DG methods for $\vec U_h^I$ and $\vec U_h^J$ may be respectively given by
\begin{align} \nonumber
& \sum\limits_{\mu = 0}^{K_d-1} \left( \iint_{I_{i,j}} \Phi_{i,j}^{(\mu)} (\vec x) \Phi_{i,j}^{(\nu)} (\vec x) {\rm d}\vec x  \right) \frac{{\rm d}\vec U_{i,j}^{I,(\mu)} (t)}{{\rm d}t}  = \frac{1}{\tau_{\max}} \iint_{I_{i,j} } \left( \vec U_h^J - \vec U_h^I \right)  \Phi_{i,j}^{(\nu)} (\vec x) {\rm d} \vec x \\[2mm] \nonumber
  & \quad  + \frac{\Delta x \Delta y}{4} \sum\limits_{\alpha = 1}^{Q} \sum\limits_{\beta = 1}^{Q}  \sum\limits_{\ell,m \in \{-1,1\}} \omega_{\alpha} \omega_{\beta} \Big(
  \vec F \left( \vec U_h^J (t,x_{i+\frac{\ell}4}^{\alpha},y_{j+\frac{m}4}^{\beta}) \right) \nabla  \Phi_{i,j}^{(\nu)}  (x_{i+\frac{\ell}4}^{\alpha},y_{j+\frac{m}4}^{\beta})
   \Big)\\[2mm] \nonumber
   & \quad - \frac{\Delta y}{2} \sum\limits_{\beta = 1}^{Q}   \sum\limits_{m,s\in \{-1,1\} } s \omega_{\beta}   \vec F_1 \left( \vec U_h^J (t,x_{i+\frac{s}{2}},y_{j+\frac{m}4}^{\beta})\right) \Phi_{i,j}^{(\nu)}(x_{i+\frac{s}{2}},y_{j+\frac{m}4}^{\beta}) \\[2mm]
   & \quad - \frac{\Delta x}{2} \sum\limits_{\alpha = 1}^{Q}  \sum\limits_{\ell,s\in \{-1,1\} } s  \omega_{\alpha}   \vec F_2 \left( \vec U_h^J (t,x_{i+\frac{\ell}4}^\alpha,y_{j+\frac{s}{2}})\right) \Phi_{i,j}^{(\nu)}(x_{i+\frac{\ell}4}^\alpha,y_{j+\frac{s}{2}}) , \label{eq:schemeI:2D}
\end{align}
and
\begin{align} \nonumber
& \sum\limits_{\mu = 0}^{K_d-1} \left( \iint_{J_{i+\frac12,j+\frac12}} \Phi_{i+\frac12,j+\frac12}^{(\mu)} (\vec x) \Phi_{i+\frac12,j+\frac12}^{(\nu)} (\vec x) {\rm d}\vec x  \right) \frac{{\rm d}\vec U_{i+\frac12,j+\frac12}^{J,(\mu)} (t)}{{\rm d}t}
\\[2mm] \nonumber
 &  = \frac{1}{\tau_{\max}} \iint_{J_{i+\frac12,j+\frac12} } \left( \vec U_h^I - \vec U_h^J \right)  \Phi_{i+\frac12,j+\frac12}^{(\nu)} (\vec x) {\rm d} \vec x \\[2mm] \nonumber
  & \quad  + \frac{\Delta x \Delta y}{4} \sum\limits_{\alpha = 1}^{Q} \sum\limits_{\beta = 1}^{Q}  \sum\limits_{\ell,m \in \{-1,1\}} \omega_{\alpha} \omega_{\beta} \Big(
  \vec F \left( \vec U_h^I (t,x_{i+\frac{\ell+2}4}^{\alpha},y_{j+\frac{m+2}4}^{\beta}) \right) \nabla  \Phi_{i,j}^{(\nu)}  (x_{i+\frac{\ell+2}4}^{\alpha},y_{j+\frac{m+2}4}^{\beta})
   \Big)\\[2mm] \nonumber
   & \quad - \frac{\Delta y}{2} \sum\limits_{\beta = 1}^{Q}  \sum\limits_{m,s\in \{-1,1\} }  s \omega_{\beta}    \vec F_1 \left( \vec U_h^I (t,x_{i+\frac{s+1}{2}},y_{j+\frac{m+2}4}^{\beta})\right) \Phi_{i+\frac12,j+\frac12}^{(\nu)}(x_{i+\frac{s+1}{2}},y_{j+ \frac{m+2}4 }^{\beta}) \\[2mm]
   & \quad - \frac{\Delta x}{2} \sum\limits_{\alpha = 1}^{Q} \sum\limits_{\ell,s\in \{-1,1\} } s  \omega_{\alpha}    \vec F_2 \left( \vec U_h^J (t,x_{i+\frac{\ell+2}4 }^\alpha,y_{j+\frac{s+1}{2}})\right) \Phi_{i+\frac12,j+\frac12}^{(\nu)}(x_{i+ \frac{\ell+2}4 }^\alpha,y_{j+\frac{s+1}{2}}) , \label{eq:schemeJ:2D}
\end{align}
where $ \nu=0,\cdots,K_d-1$, $\vec F=(\vec F_1,\vec F_2)$, $\big\{ x_{i\pm \frac14}^\alpha \big\}_{\alpha=1}^Q$ and $\big\{ y_{j\pm \frac14}^\alpha \big\}_{\alpha=1}^Q$ denote the Gaussian nodes transformed into the
interval $\big[x_{i\pm \frac14}-\frac{\Delta x}{4}, x_{i\pm \frac14}+\frac{\Delta x}{4} \big]$ and  $\big[y_{j\pm \frac14}-\frac{\Delta y}{4}, y_{j\pm \frac14}+\frac{\Delta y}{4} \big]$, respectively, and the associated Gaussian quadrature weights $\{ \omega_\alpha \}_{\alpha=1}^Q$ satisfy
$\omega_\alpha>0$ and $\sum\limits_{\alpha=1}^Q \omega_\alpha =1$.
For the accuracy requirement,  $Q$ should be not less than
$K+1$ for a $\mathbb{P}^K$-based central DG method \citep{CockburnIV}.

If taking the bases as the scaled Legendre polynomials such that $\Phi_{i,j}^{(0)}(\vec x) = \Phi_{i+\frac12,j+\frac12}^{(0)}(\vec x)=1$,
then  from \eqref{eq:schemeI:2D}--\eqref{eq:schemeJ:2D} with $\nu=0$,
one may derive the evolution equations for the cell-averages of $\vec U_h^I$ and $\vec U_h^J$
as follows
\begin{align} \nonumber
& \frac{{\rm d}\vec U_{i,j}^{I,(0)} (t)}{{\rm d}t}
=\frac{1}{\tau_{\max}} \frac{1}{\Delta x \Delta y}  \iint_{I_{i,j} } \left( \vec U_h^J - \vec U_h^I \right)   {\rm d} x {\rm d} y \\[2mm] \nonumber
      &  - \frac{1}{2} \sum\limits_{\beta = 1}^{Q}  \sum\limits_{m,s\in \{-1,1\} } s \omega_{\beta}
        \left( \frac{1}{\Delta x} \vec F_1 \left( \vec U_h^J (t,x_{i+\frac{s}{2}},y_{j+  \frac{m}4  }^{\beta})\right)
        +  \frac{1}{\Delta y} \vec F_2 \left( \vec U_h^J (t,x_{i+  \frac{m}4 }^\beta,y_{j+\frac{s}{2}})\right) \right)
          \\[2mm]
& \quad \quad \quad \quad=: {\mathscr{L}}^I_{i,j} (\vec U_h^I,\vec U_h^J),
\end{align} \label{eq:schemeI:2D:ave}
and
\begin{align} \nonumber
& \frac{{\rm d}\vec U_{i+\frac12,j+\frac12}^{J,(0)} (t)}{{\rm d}t}
=\frac{1}{\tau_{\max}} \frac{1}{\Delta x \Delta y}  \iint_{J_{i+\frac12,j+\frac12} } \left( \vec U_h^I - \vec U_h^J \right)   {\rm d} x {\rm d} y \\[2mm] \nonumber
& - \frac{1}{2} \sum\limits_{\beta = 1}^{Q}  \sum\limits_{m,s\in \{-1,1\} } s \omega_{\beta}
        \left( \frac{1}{\Delta x} \vec F_1 \left( \vec U_h^I (t,x_{i+\frac{s+1}{2}},y_{j+  \frac{m+2}4 }^{\beta})\right)
        +  \frac{1}{\Delta y} \vec F_2 \left( \vec U_h^J (t,x_{i+  \frac{m+2}4 }^\beta,y_{j+\frac{s+1}{2}})\right) \right)
          \\[2mm]
& \quad \quad \quad \quad =: {\mathscr{L}}^J_{i+\frac12,j+\frac12} (\vec U_h^J,\vec U_h^I),
\end{align} \label{eq:schemeJ:2D:ave}

If the time derivatives in \eqref{eq:schemeI:2D}--\eqref{eq:schemeJ:2D} are approximated by
using the SSP  Runge-Kutta  or multi-step methods, see e.g. \eqref{eq:RK1} or \eqref{eq:multi-step},
then the fully discrete 2D central DG methods may be obtained.
In the following the PCP technique is discussed  for the above 2D central DG methods.
First,  it may be proved that the 2D central DG methods with $K =0$ are
PCP under a CFL type condition.

\begin{thm}\label{thm:2DK=0}
	If  $K=0$ and $\vec U_{i,j}^I ,\vec U_{i+\frac12,j+\frac12}^J \in {\cal G}$ for all $i$ and $j$, 	
 then under the CFL type condition
 	\begin{equation}\label{eq:2DCFL-K=0}
 		0< \frac{\Delta t}{\Delta x} +\frac{\Delta t}{\Delta y}   < \frac{ \theta } {2c},\quad \theta \in (0,1],
 	\end{equation}
 	one has
	\begin{equation*}
		   \vec U_{i,j}^I + \Delta t {\mathscr{L}}^I_{i,j} (\vec U_h^I,\vec U_h^J)  \in {\cal G},\quad \vec U_{i+\frac12,j+\frac12}^J + \Delta t {\mathscr{L}}^J_{i+\frac12,j+\frac12} (\vec U_h^J,\vec U_h^I) \in {\cal G},
	\end{equation*}
	for all $i$ and $j$.
\end{thm}

\begin{proof}
Because both $\vec U_{i,j}^I$ and $\vec U_{i+\frac12,j+\frac12}^J$ are constant when $K=0$,
 one has
\begin{align} \nonumber
& \vec U_{i,j}^I + \Delta t {\mathscr{L}}^I_{i,j} (\vec U_h^I,\vec U_h^J) \\[2mm]  \nonumber
& =  \vec U_{i,j}^I + \frac{\theta}{\Delta x\Delta y} \iint_{I_{i,j}} \left( \vec U_h^J - \vec U_{i,j}^I \right)   {\rm d} x {\rm d} y \\[2mm]  \nonumber
& \quad  - \frac{1}{2}  \sum \limits_{m,s\in\{-1,1\}} s \left( \frac{\Delta t}{\Delta x} \vec F_1 \left( \vec U^J_{i+\frac{s}{2} , j+\frac{m}{2} } \right)
 + \frac{\Delta t}{\Delta y} \vec F_2 \left( \vec U^J_{i+\frac{m}{2} , j+\frac{s}{2} } \right)   \right)  \\[2mm]  \nonumber
& = (1-\theta) \vec U_{i,j}^I +\frac{\theta}{4}  \sum \limits_{m,s\in\{-1,1\}} \vec U^J_{i+\frac{s}{2} , j+\frac{m}{2} }  \\[2mm]  \nonumber
&\quad
- \frac{1}{2}  \sum \limits_{m,s\in\{-1,1\}}  \left( \frac{\Delta t}{\Delta x} s \vec F_1 \left( \vec U^J_{i+\frac{s}{2} , j+\frac{m}{2} } \right)  + \frac{\Delta t}{\Delta y} m \vec F_2 \left( \vec U^J_{i+\frac{s}{2} , j+\frac{m}{2} } \right)   \right)  \\[2mm]
& = (1-\theta) \vec U_{i,j}^I + \frac{\theta}{4}
 \sum \limits_{m,s\in\{-1,1\}} \left(  \frac{ \Delta y }{  \Delta x +\Delta y }
 \vec U^{J,[1]}_{i+\frac{s}{2} , j+\frac{m}{2} } + \frac{ \Delta x }{  \Delta x +\Delta y }
 \vec U^{J,[2]}_{i+\frac{s}{2} , j+\frac{m}{2} }  \right) , \label{eq:convex:2D:K=0}
\end{align}
where
\begin{align*}
& \vec U^{J,[1]}_{i+\frac{s}{2} , j+\frac{m}{2} } :=  \vec U^J_{i+\frac{s}{2} , j+\frac{m}{2} }  - \frac{2s}{\theta} \left( \frac{\Delta t}{\Delta x} + \frac{\Delta t}{\Delta y}  \right)  \vec F_1 \left(  \vec U^J_{i+\frac{s}{2} , j+\frac{m}{2} }  \right),
\\
 & \vec U^{J,[2]}_{i+\frac{s}{2} , j+\frac{m}{2} } :=  \vec U^J_{i+\frac{s}{2} , j+\frac{m}{2} }  - \frac{2m}{\theta} \left( \frac{\Delta t}{\Delta x} + \frac{\Delta t}{\Delta y}  \right)  \vec F_2 \left(  \vec U^J_{i+\frac{s}{2} , j+\frac{m}{2} }  \right).
\end{align*}
Thanks to the Lax-Friedrichs splitting property in Lemma \ref{lam:propertyG},
$ \vec U^{J,[1]}_{i+\frac{s}{2} , j+\frac{m}{2} }$,$ \vec U^{J,[2]}_{i+\frac{s}{2} , j+\frac{m}{2} }\in \mathcal G$ under the theorem hypothesis.
Combining those with \eqref{eq:convex:2D:K=0}
and using the convexity of $\cal G$ further yields
$ \vec U_{i,j}^I + \Delta t {\mathscr{L}}^I_{i,j} (\vec U_h^I,\vec U_h^J) \in {\cal G}$.
 Similar arguments may yield $\vec U_{i+\frac12,j+\frac12}^J + \Delta t {\mathscr{L}}^J_{i+\frac12,j+\frac12} (\vec U_h^J,\vec U_h^I) \in {\cal G}$. The proof is completed.
\end{proof}

Theorem \ref{thm:2DK=0} indicates that the first-order accurate 2D central DG method is PCP under the CFL type condition \eqref{eq:2DCFL-K=0}
if the forward Euler method is used for time discretization.
Similar to the 1D case, it is important to find out  a sufficient condition on the polynomial vectors $\vec U_{i,j}^I(\vec x)$ and $\vec U _{i+\frac12,j+\frac12}^J(\vec x)$  in a high-order accurate PCP central DG method.
 For the sake of convenience, omit the independent variable $t$  temporarily,
and let $\big\{ {\hat x}_{i\pm \frac14}^\alpha \big\}_{\alpha=1}^L$ and $\big\{ {\hat y}_{j\pm \frac14}^\alpha \big\}_{\alpha=1}^L$  be the Gauss-Lobatto nodes transformed into the
interval $\big[x_{i\pm \frac14}-\frac{\Delta x}{4}, x_{i\pm \frac14}+\frac{\Delta x}{4} \big]$ and $\big[y_{j\pm \frac14}-\frac{\Delta y}{4}, y_{j\pm \frac14}+\frac{\Delta y}{4} \big]$ respectively, and $\{ {\hat{\omega}}_\alpha \}_{\alpha=1}^L$ be the associated Gaussian quadrature weights satisfying
$\hat{ \omega}>0$ and $\sum\limits_{\alpha=1}^L {\hat{ \omega}}_\alpha =1$, where $L\ge (K+3)/2$.

\begin{thm} \label{thm:PCP:2DRHD}
If 
${\vec U}^I_{i,j}(\hat x_{i+\frac{s}{4}}^\alpha, y_{j+\frac{m}{4}}^\beta)\in {\cal G}$ and ${\vec U}^J_{i+\frac12,j+\frac12} ( x_{i+\frac12+\frac{s}{4}}^\beta,\hat y_{j+\frac12+\frac{m}{4}}^\alpha) \in {\cal G}$ for all $i,j\in \mathbb{Z}$, $s,m\in \{-1,1\}$, $\alpha=1,2,\cdots,L$, and $\beta=1,2,\cdots,Q$, then under the CFL type condition
 	\begin{equation}\label{eq:2DCFL-Kge1}
 		0< \frac{ \Delta t }{\Delta x} + \frac{ \Delta t }{\Delta y} \le \frac{ {\hat{\omega}}_1 \theta } {2c},\quad \theta \in (0,1],
 	\end{equation}
 	one has
	\begin{equation*}
		   \vec U_{i,j}^{I,(0)} + \Delta t {\mathscr{L}}^I_{i,j} (\vec U_h^I,\vec U_h^J)  \in {\cal G},\quad \vec U_{i+\frac12,j+\frac12}^{J,(0)} + \Delta t {\mathscr{L}}^J_{i+\frac12,j+\frac12} (\vec U_h^J,\vec U_h^I) \in {\cal G},
	\end{equation*}
	for all $i$ and $j$.
\end{thm}
\begin{proof}
Using the convexity of $\mathcal G$ and the exactness of the Gauss-Lobatto quadrature rule
with $L$ nodes and the Gauss quadrature rule with $Q$ nodes for the polynomials of degree $K$ yields
\begin{align}\nonumber
\frac{1}{\Delta x\Delta y} \iint_{I_{i,j}} \vec U_h^J {\rm d}x{\rm d}y
& = \frac{1}{\Delta x} \int_{x_{i-\frac12}}^{x_{i+\frac12}} \left( \frac12 \sum \limits_{\beta =1}^Q \sum \limits_{m\in \{-1,1\} } \omega_\beta \vec U_h^J (x,y_{j+\frac{m}{4}}^\beta)  \right) {\rm d} x \\ \nonumber
& = \frac12 \sum \limits_{\beta =1}^Q \sum \limits_{m\in \{-1,1\} } \omega_\beta \left( \frac{1}{\Delta x} \int_{x_{i-\frac12}}^{x_{i+\frac12}}  \vec U_h^J (x,y_{j+\frac{m}{4}}^\beta)  {\rm d} x \right) \\ \nonumber
& = \frac12 \sum \limits_{\beta =1}^Q \sum \limits_{m\in \{-1,1\} } \omega_\beta \left(    \frac12 \sum \limits_{\alpha =1}^L \sum \limits_{s\in \{-1,1\} } \hat{\omega}_\alpha \vec U_h^J (\hat x_{i+\frac{s}{4}}^\alpha,y_{j+\frac{m}{4}}^\beta)    \right) \\
& = \frac12 \sum \limits_{\beta =1}^Q \sum \limits_{m\in \{-1,1\} } \omega_\beta \left( \frac { \hat \omega_1}2   \sum \limits_{s\in \{-1,1\} } \vec U_h^J (x_{i+\frac{s}{2}}, y_{j+\frac{m}{4}}^\beta )  + (1-\hat \omega_1) \Xi_{i,j+\frac{m}{4}}^\beta \right), \label{eq:2DconvexcombA}
\end{align}
where
$$
\Xi_{i,j+\frac{m}{4}}^\beta :=\frac{1}{2\left(1-\hat \omega_1 \right)} \left(  \sum \limits_{\alpha=2}^L \hat \omega_\alpha \vec U_h^J (   \hat x_{i-\frac14}^\alpha, y_{j+\frac{m}{4}}^\beta )   +  \sum \limits_{\alpha=1}^{L-1} \hat \omega_\alpha \vec U_h^J (   \hat x_{i+\frac14}^\alpha, y_{j+\frac{m}{4}}^\beta )   \right) \in {\cal G},
$$
and $\hat \omega_1 = \hat \omega_L \le \frac12$ has been used.
Similarly, one has
\begin{align} \label{eq:2DconvexcombB}
\frac{1}{\Delta x\Delta y} \iint_{I_{i,j}} \vec U_h^J {\rm d}x{\rm d}y
= \frac12 \sum \limits_{\beta =1}^Q \sum \limits_{m\in \{-1,1\} } \omega_\beta \left( \frac { \hat \omega_1}2   \sum \limits_{s\in \{-1,1\} } \vec U_h^J ( x_{i+\frac{m}{4}}^\beta,  y_{j+\frac{s}{2}} )  + (1-\hat \omega_1) \Xi_{i+\frac{m}{4},j}^\beta \right),
\end{align}
with
$$
\Xi_{i+\frac{m}{4},j}^\beta :=\frac{1}{2\left(1-\hat \omega_1 \right)} \left(  \sum \limits_{\alpha=2}^L \hat \omega_\alpha \vec U_h^J (   x_{i+\frac{m}{4}}^\beta, \hat y_{j-\frac14}^\alpha )   +  \sum \limits_{\alpha=1}^{L-1} \hat \omega_\alpha \vec U_h^J (   x_{i+\frac{m}{4}}^\beta,  \hat y_{j+\frac14}^\alpha )   \right) \in {\cal G},
$$
and
$$
\vec U_{i,j}^{I,(0)}=\frac12 \sum \limits_{\beta =1}^Q \sum \limits_{m\in \{-1,1\} } \omega_\beta \left(    \frac12 \sum \limits_{\alpha =1}^L \sum \limits_{s\in \{-1,1\} } \hat{\omega}_\alpha \vec U_h^I (\hat x_{i+\frac{s}{4}}^\alpha,y_{j+\frac{m}{4}}^\beta)    \right) \in {\cal G}.
$$
Combining \eqref{eq:2DconvexcombA} and   \eqref{eq:2DconvexcombB} gives
\begin{align}\nonumber
& \frac{1}{\Delta x\Delta y} \iint_{I_{i,j}} \vec U_h^J {\rm d}x{\rm d}y  = \frac{\lambda_x} {\lambda_x+\lambda_y} \frac{1}{\Delta x\Delta y} \iint_{I_{i,j}} \vec U_h^J {\rm d}x{\rm d}y +  \frac{\lambda_y} {\lambda_x+\lambda_y} \frac{1}{\Delta x\Delta y} \iint_{I_{i,j}} \vec U_h^J {\rm d}x{\rm d}y \\ \nonumber
& = \frac12 \sum \limits_{\beta =1}^Q \sum \limits_{m\in \{-1,1\} } \omega_\beta \left( \frac { \hat \omega_1 }{2(\lambda_x+\lambda_y)}   \sum \limits_{s\in \{-1,1\} }  \left( \lambda_x  \vec U_h^J (x_{i+\frac{s}{2}}, y_{j+\frac{m}{4}}^\beta ) +
\lambda_y \vec U_h^J ( x_{i+\frac{m}{4}}^\beta,  y_{j+\frac{s}{2}} ) \right) \right) \\
& \quad \quad + \frac12 \sum \limits_{\beta =1}^Q \sum \limits_{m\in \{-1,1\} } \omega_\beta (1-\hat \omega_1) \Xi_{i,j}^{\beta,m}, \label{eq:2DconvexcombAB}
\end{align}
where $\lambda_x:=\Delta t/\Delta x, \lambda_y:=\Delta t/\Delta y$, and
$$
\Xi_{i,j}^{\beta,m} := \frac{\lambda_x} {\lambda_x+\lambda_y}  \Xi_{i,j+\frac{m}4}^\beta + \frac{\lambda_y} {\lambda_x+\lambda_y}  \Xi_{i+\frac{m}4,j}^\beta \in {\cal G}.
$$
Therefore, one gets
\begin{align} \nonumber
& \vec U_{i,j}^{I,(0)} + \Delta t {\mathscr{L}}^I_{i,j} (\vec U_h^I,\vec U_h^J) = (1- \theta)  \vec U_{i,j}^{I,(0)} + \frac{\theta}{\Delta x\Delta y} \iint_{I_{i,j}} \vec U_h^J {\rm d}x{\rm d}y \\ \nonumber
&  \quad - \frac{1}{2} \sum\limits_{\beta = 1}^{Q}   \sum\limits_{m,s\in \{-1,1\} } s \omega_{\beta}
        \left( \lambda_x \vec F_1 \left( \vec U_h^J (x_{i+\frac{s}{2}},y_{j+\frac{m}{4}}^{\beta})\right)
        +  \lambda_y \vec F_2 \left( \vec U_h^J (x_{i+\frac{m}{4}}^\beta,y_{j+\frac{s}{2}})\right) \right) \\ \nonumber
& \overset{\eqref{eq:2DconvexcombAB}} {=} (1- \theta)  \vec U_{i,j}^{I,(0)} + \frac{\theta}{2} \sum \limits_{\beta =1}^Q \sum \limits_{m\in \{-1,1\} } \omega_\beta (1-\hat \omega_1) \Xi_{i,j}^{\beta,m} \\
&\quad + \frac{\theta}{2}  \sum \limits_{\beta =1}^Q  \sum \limits_{m,s\in \{-1,1\} }  \frac{ \omega_\beta \hat \omega_1}{2}
\left( \frac{ \lambda_x }{\lambda_x+\lambda_y} \vec U^{J,\beta}_{i+\frac{s}{2},j+\frac{m}{4}}
+
\frac{ \lambda_y }{\lambda_x+\lambda_y} \vec U^{J,\beta}_{i+\frac{m}{4},j+\frac{s}{2}}
  \right), \label{eq:convexcom2DKge1}
\end{align}
where
\begin{align*}
& \vec U^{J,\beta}_{i+\frac{s}{2},j+\frac{m}{4}}   := \vec U_h^J (x_{i+\frac{s}{2}},y_{j+\frac{m}{4}}^{\beta}) -\frac{2s(\lambda_x+\lambda_y)}{\theta \hat \omega_1} \vec F_1 \left( \vec U_h^J (x_{i+\frac{s}{2}},y_{j+\frac{m}{4}}^{\beta}) \right) \in {\cal G} \cup \partial {\cal G},
\\
& \vec U^{J,\beta}_{i+\frac{m}{4},j+\frac{s}{2}}  := \vec U_h^J (x_{i+\frac{m}{4}}^\beta,y_{j+\frac{s}{2}}) -  \frac{2s(\lambda_x+\lambda_y)}{\theta \hat \omega_1} \vec F_2  \left( \vec U_h^J (x_{i+\frac{m}{4}}^\beta,y_{j+\frac{s}{2}})\right) \in {\cal G} \cup \partial {\cal G},
\end{align*}
due to  the Lax-Friedrichs splitting property in Lemma \ref{lam:propertyG}
and the theorem hypothesis.
Using  \eqref{eq:convexcom2DKge1} and the convexity of $\cal G$  further yields $ \vec U_{i,j}^{I,(0)} + \Delta t {\mathscr{L}}^I_{i,j} (\vec U_h^I,\vec U_h^J) \in {\cal G}$.
 Similar arguments yield $\vec U_{i+\frac12,j+\frac12}^{J,(0)} + \Delta t {\mathscr{L}}^J_{i+\frac12,j+\frac12} (\vec U_h^J,\vec U_h^I) \in {\cal G}$. The proof is completed.
\end{proof}


Although the sufficient condition for the 2D high-order accurate central DG methods
in Theorem \ref{thm:PCP:2DRHD} is  given  only for the forward Euler time discretization,
it is also  valid the high-order accurate SSP time discretization \eqref{eq:RK1} or  \eqref{eq:multi-step},
which has been expressed as  a convex combination of the forward Euler method.
Built on the above theoretical results,
the 2D PCP limiting procedure may be  presented and is very similar to the 1D case
so that its details may be omitted here.
The only difference is that the 2D PCP limiter is used to
ensure the admissibility of $\vec U_h^I(\vec x)$ and $\vec U_h^J(\vec x)$ at
the following points
$$
{\cal S}_{ij} = \left( \hat {\cal S}_{i}^x \otimes  {\cal S}_{j}^y \right) \bigcup  \left( {\cal S}_{i}^x \otimes \hat {\cal S}_{j}^y \right)   \bigcup  \left( {\cal S}_{i}^x \otimes  {\cal S}_{j}^y \right),
$$
for all $i$ and $j$, where $\bigotimes$ denotes the tensor product of sets, and
\begin{align*}
& \hat {\cal S}_{i}^x := \left\{\hat{x}_{i-\frac14}^\alpha\right\}_{\alpha=1}^L  \bigcup \left\{\hat{x}_{i+\frac14}^\alpha\right\}_{\alpha=1}^L, \quad
  {\cal S}_{i}^x := \left\{x_{i-\frac14}^\beta\right\}_{\beta=1}^Q  \bigcup \left\{{x}_{i+\frac14}^\beta\right\}_{\beta=1}^Q,\\[2mm]
& \hat {\cal S}_{j}^y := \left\{\hat{y}_{j-\frac14}^\alpha\right\}_{\alpha=1}^L  \bigcup \left\{\hat{y}_{j+\frac14}^\alpha\right\}_{\alpha=1}^L, \quad
 {\cal S}_{j}^y := \left\{y_{j-\frac14}^\beta\right\}_{\beta=1}^Q  \bigcup \left\{{y}_{j+\frac14}^\beta\right\}_{\beta=1}^Q.
\end{align*}

If replacing the solution polynomials of high-order accurate central DG methods
with the limited polynomials  at each stage of SSP Runge-Kutta method \eqref{eq:RK1}
or each step of SSP muti-step method \eqref{eq:multi-step},
then using Theorem \ref{thm:PCP:1DRHD} may prove that
the resulting 2D fully discrete central DG methods are PCP under some CFL type conditions.

%

\begin{thm}\label{thm:2D}
If  the 2D high-order accurate  central DG solution polynomials are revised
to the above limited polynomials at each stage of SSP Runge-Kutta method  \eqref{eq:RK1}
or each step of muti-step method \eqref{eq:multi-step},
then (i)  the resulting Runge-Kutta central DG methods are PCP under the CFL type condition \eqref{eq:2DCFL-Kge1},
{(ii)}  the resulting multi-step central DG scheme is PCP under the CFL type condition
 	\begin{equation}\label{eq:2DCFL-Kge1MS}
 	0< \frac{ \Delta t }{\Delta x} + \frac{ \Delta t }{\Delta y} \le \frac{ {\hat{\omega}}_1 \theta } {2c},
 	\quad \theta \in \left(0,\frac13 \right].
 	\end{equation}
\end{thm}
It is worth mentioning that the resulting 2D PCP central DG methods are also $L^1$-stable
similar to Theorem \ref{thm:stability}.

\section{Numerical experiments}
\label{sec:experiments}

This section conducts several numerical experiments on the 1D and 2D highly challenging
ultra-relativistic RHD problems
with large Lorentz factor, or strong discontinuities, or low rest-mass density or pressure,
to demonstrate  the accuracy, robustness, and effectiveness
of the proposed PCP central DG methods.
 To shorten the paper length, it  will only present
 the numerical results obtained by the ${\mathbb{P}}^2$-based central DG methods
  with the third-order accurate Runge-Kutta time discretization \eqref{eq:RK1} or multi-step time discretization \eqref{eq:multi-step}.
For convenience, abbreviate them as ``{\tt PCPRKCDGP2}'' and ``{\tt PCPMSCDGP2}'' respectively.
Unless otherwise stated, $\theta$ is taken as 1 for {\tt PCPRKCDGP2} and $\frac13$ for {\tt PCPMSCDGP2}.

\subsection{1D case}

This section is to  conduct four 1D numerical experiments.
In all computations, the time stepsize  $\Delta t$ will be taken
as $ 0.5 { {\hat{\omega}}_1 \theta \Delta x} c^{-1}$ with $\hat{\omega}_1=\frac16$.

\begin{example}[1D smooth problem] \label{example1Dsmooth}\rm
It  is used  to check the accuracy of the 1D PCP central DG methods.
	The initial data
	are taken as
$$
	\vec V(0,x)=\big(\rho(0,x),\vec v(0,x),p(0,x)\big)^{\rm T}
	=\big( 1+0.99999\sin(2 \pi x), 0.99, 10^{-2}\big)^{\rm T},
	\quad x\in [0,1),
$$
and  thus the exact solutions can be given by
$$\vec V(t,x)=\big( 1+0.99999\sin(x-0.99t), 0.99, 10^{-2} \big)^{\rm T}, \quad x\in[0,1),\ \ t\geq 0,
$$
which describes a RHD sine wave propagating periodically and quickly in the interval $[0,1)$ with low density and pressure.
%

The ideal EOS \eqref{eq:EOSideal} with  $\Gamma=\frac53$ is first considered.
Table \ref{tab:1Daccuracy}
lists the  $l^1$ and $l^2$-errors at $t=0.2$ and  corresponding orders
obtained  by using {\tt PCPRKCDGP2} and {\tt PCPMSCDGP2},
respectively.
The results show that the theoretical orders are obtained by both {\tt PCPRKCDGP2} and {\tt PCPMSCDGP2} and
the PCP limiting procedure does not destroy the accuracy. The error graphs  in Fig. \ref{fig:1Dsmooth}
  display the same phenomenon for {three different} EOS.

\begin{table}[htbp]
  \centering 
    \caption{\small Example \ref{example1Dsmooth}: Numerical    $l^1$- and $l^2$-errors and orders
    	at $t=0.2$ of  {\tt PCPRKCDGP2} and {\tt PCPMSCDGP2} for the ideal EOS with $\Gamma=5/3$.
  }
\begin{tabular}{|c||c|c|c|c||c|c|c|c|}
  \hline
\multirow{2}{8pt}{$N$}
 &\multicolumn{4}{c||}  {\tt PCPRKCDGP2}  &\multicolumn{4}{c|}  {\tt PCPMSCDGP2} \\
 \cline{2-9}
 &$l^1$ $ {\mbox{error}}$ &$l^1$ order &$l^2$ error &$l^2$ order  &$l^1$ error &$l^1$ order &$l^2$ error &$l^2$ order \\
 \hline
10 &2.402e-4& --     & 3.102e-4   &--    &1.987e-4& --    &2.430e-4 &--\\
20&3.439e-5&2.80  & 4.988e-5   &2.64   &2.290e-5& 3.12&2.948e-5 &3.04\\
40&5.031e-6&2.77  &9.328e-6   &2.42  &2.845e-6& 3.01&3.686e-6 &3.00\\
80&6.036e-7&3.06  &1.180e-6    &2.98 &3.564e-7& 3.00&4.611e-7 &3.00\\
160&4.458e-8&3.76 &5.767e-8  &4.35 &4.456e-8& 3.00&5.766e-8 &3.00\\
320&5.573e-9&3.00&7.209e-9  &3.00 &5.570e-9& 3.00&7.207e-9 &3.00\\
\hline
\end{tabular}\label{tab:1Daccuracy}
\end{table}

\begin{figure}[htbp]
  \centering
    \subfigure[EOS \eqref{eq:EOS:Mathews}]
    {\includegraphics[width=0.32\textwidth]{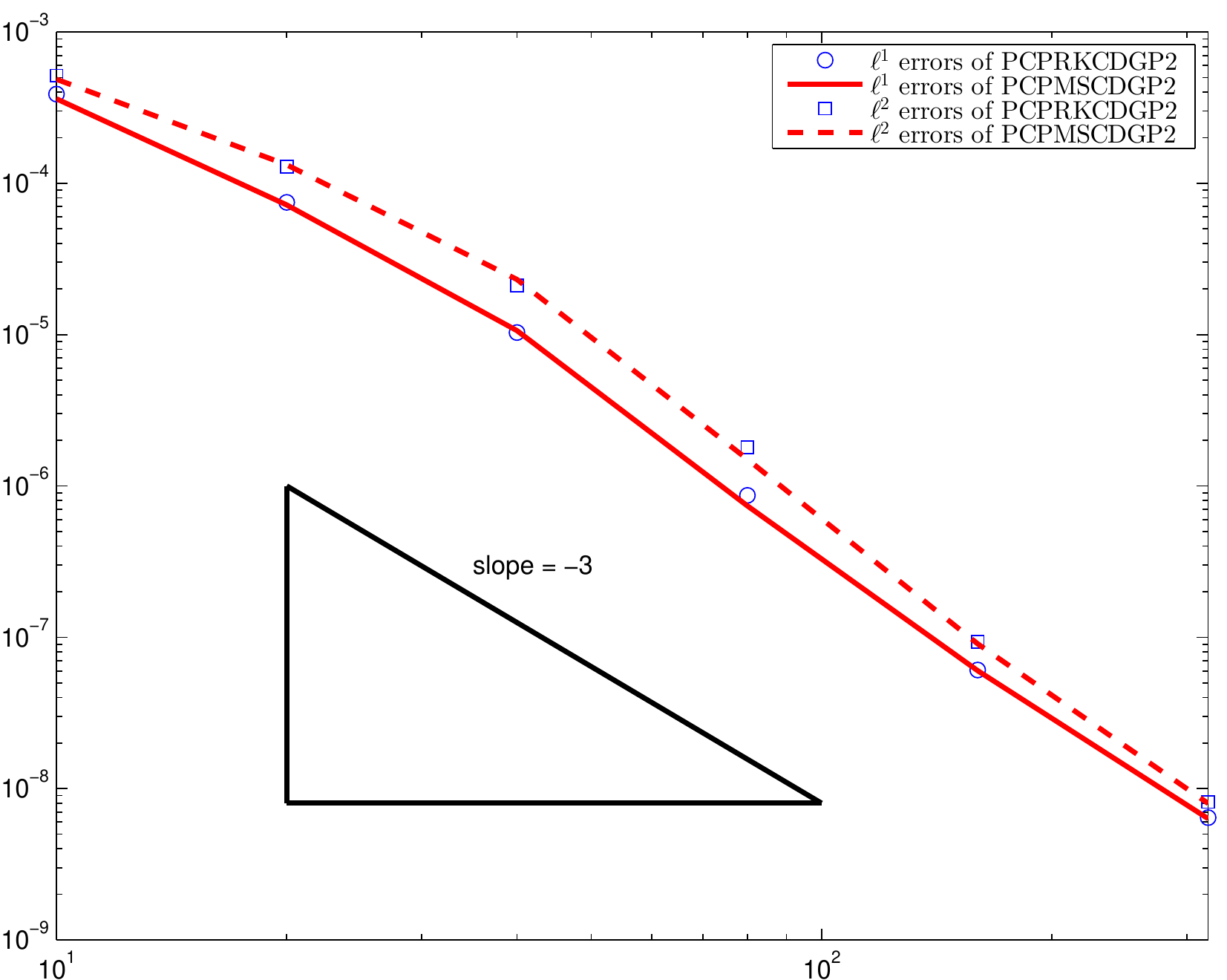}}
  \subfigure[EOS \eqref{eq:EOS:Sokolov}]
  {\includegraphics[width=0.32\textwidth]{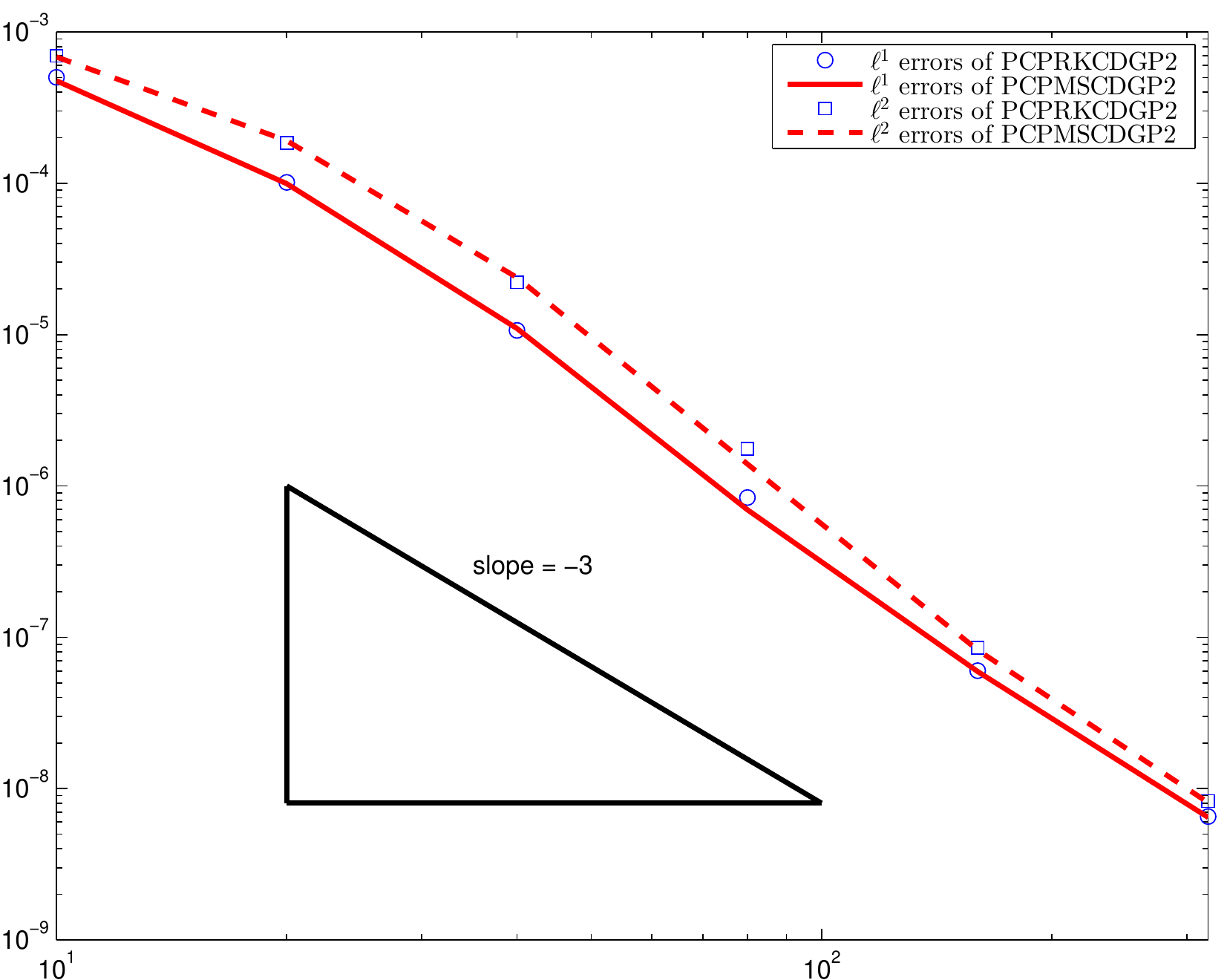}}
  \subfigure[EOS \eqref{eq:EOS:Ryu}]
  {\includegraphics[width=0.32\textwidth]{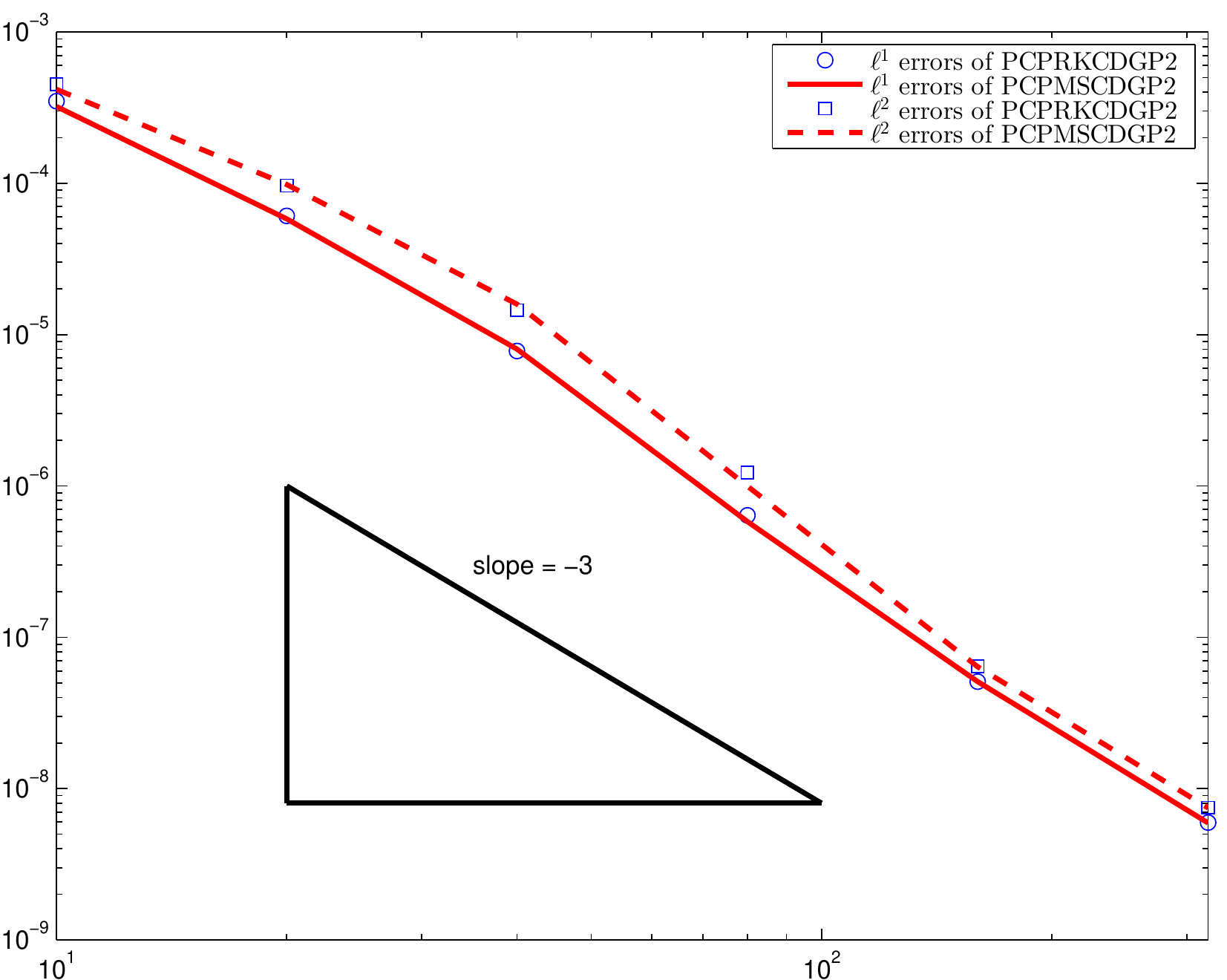}}
  \caption{\small Example \ref{example1Dsmooth}: Numerical  $l^1$- and $l^2$-errors
  	 at $t=0.2$ of  {\tt PCPRKCDGP2} and {\tt PCPMSCDGP2}.
 }
  \label{fig:1Dsmooth}
\end{figure}

\end{example}

To verify the capability of the proposed PCP central DG methods in resolving 1D ultra-relativistic wave configurations,  a Riemann problem, a shock heating problem,
and  a blast wave interaction problem will be solved and
only numerical  results of {\tt PCPMSCDGP2} will be presented in the following
since the results of  {\tt PCPRKCDGP2} are very similar to {\tt PCPMSCDGP2}.

\begin{figure}[htbp]
  \centering
  {\includegraphics[width=0.48\textwidth]{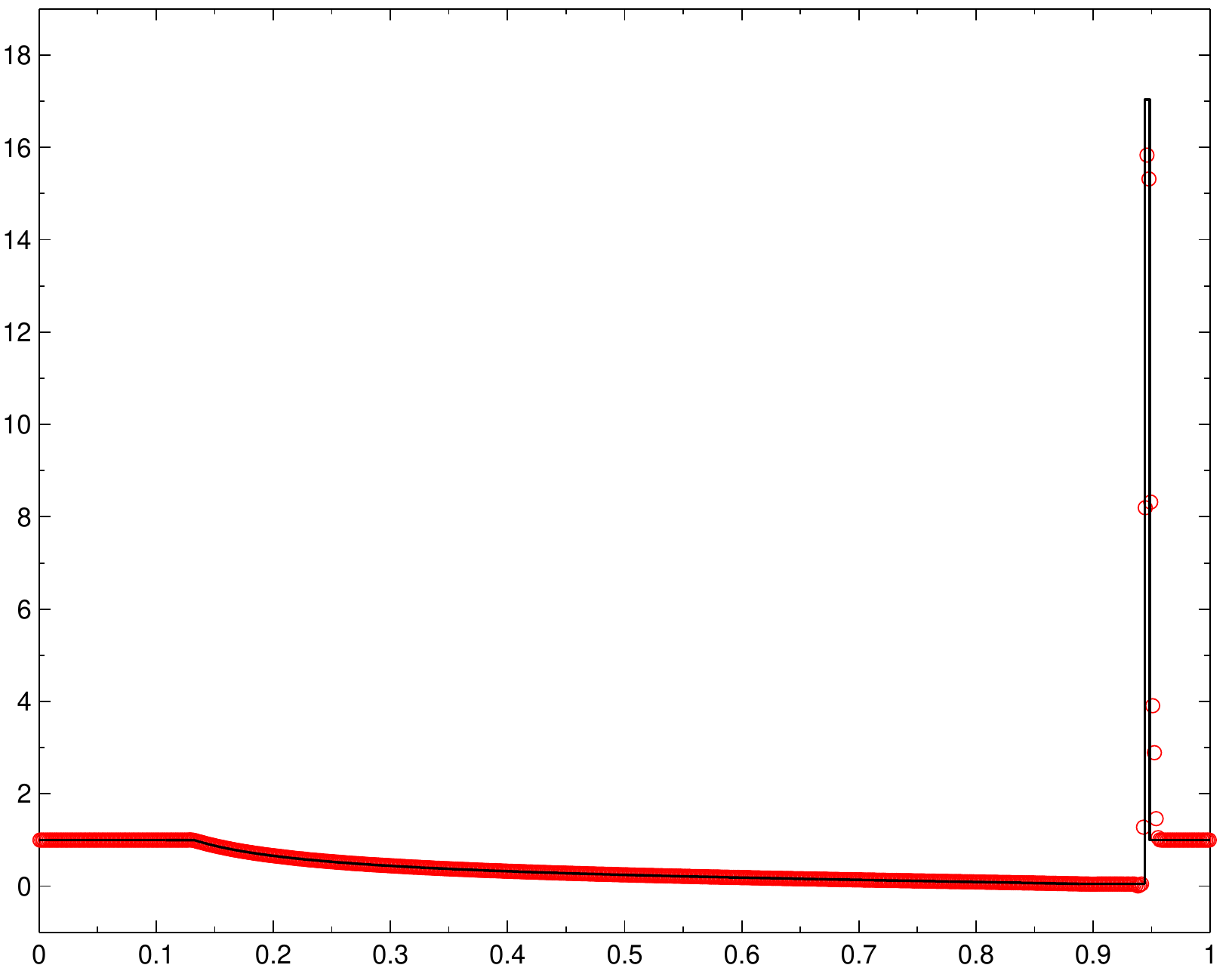}}
  {\includegraphics[width=0.48\textwidth]{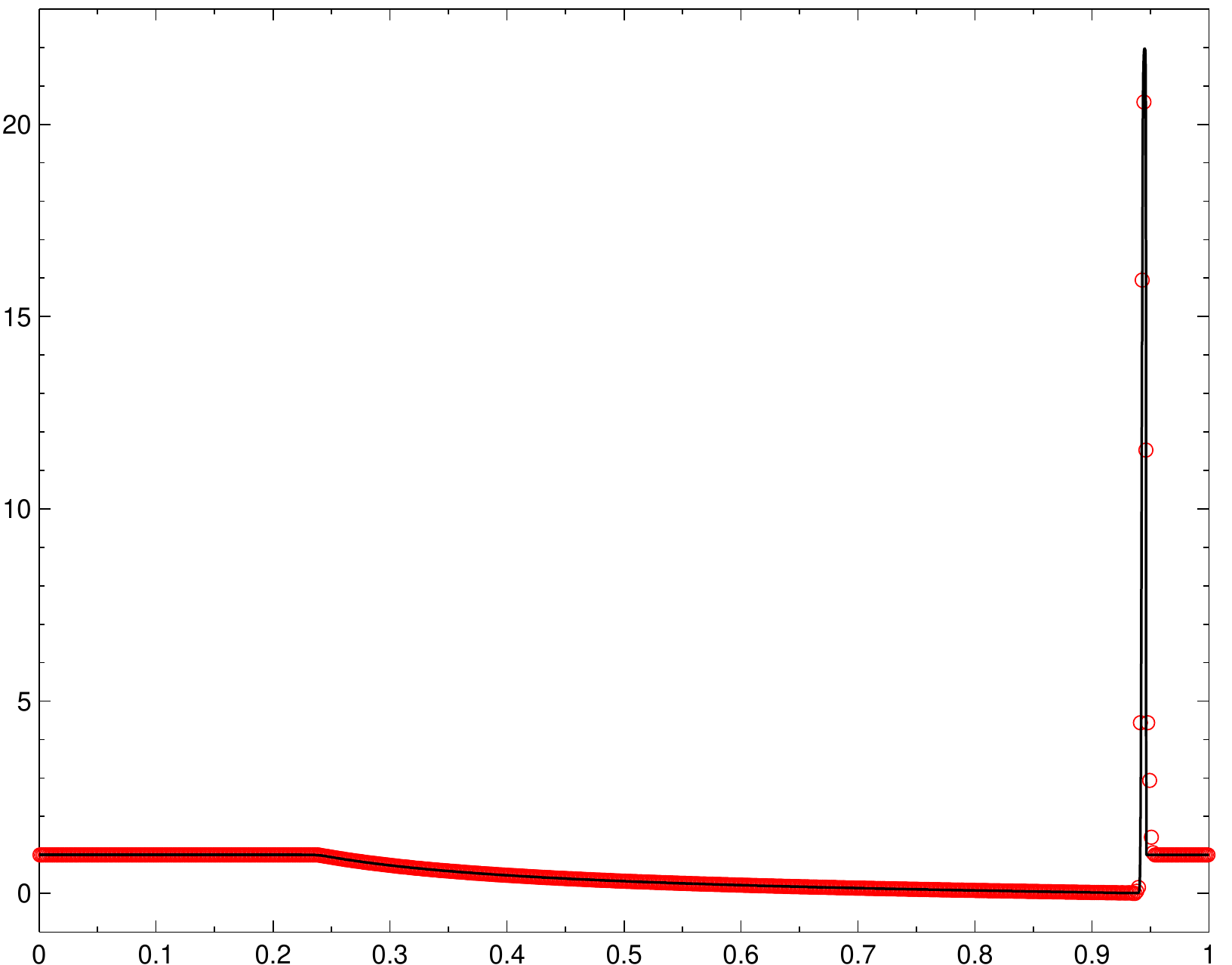}}
  {\includegraphics[width=0.48\textwidth]{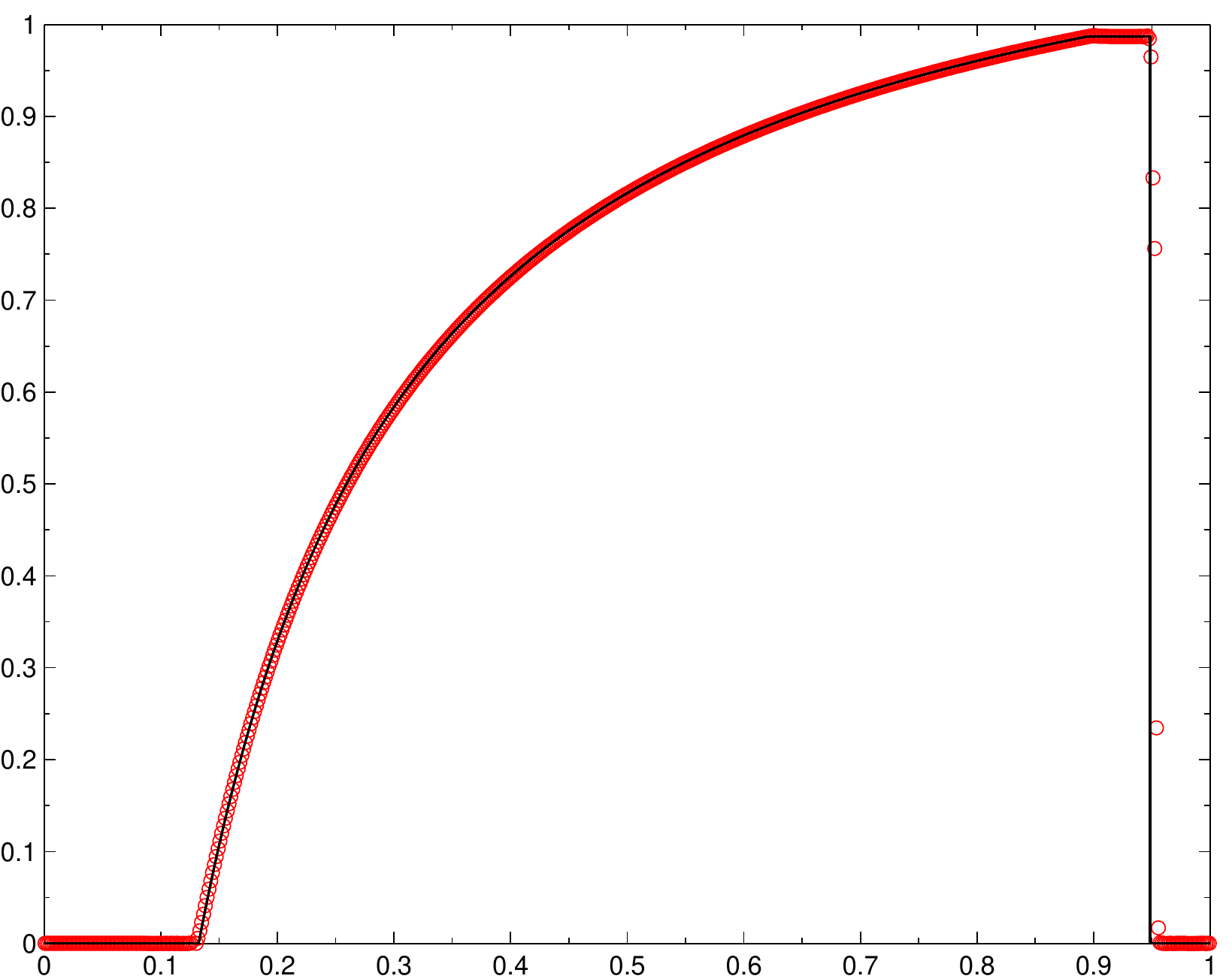}}
  {\includegraphics[width=0.48\textwidth]{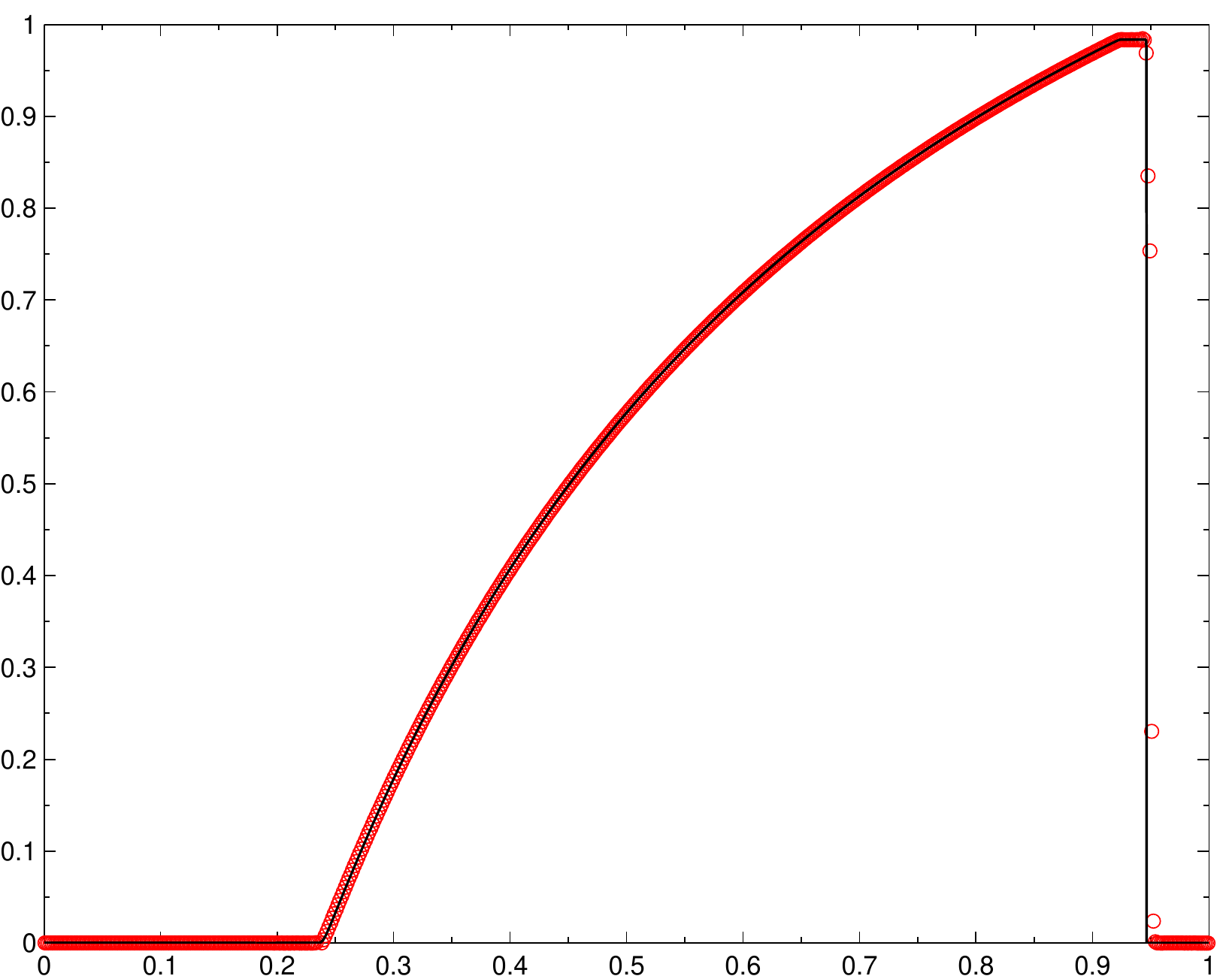}}
  {\includegraphics[width=0.48\textwidth]{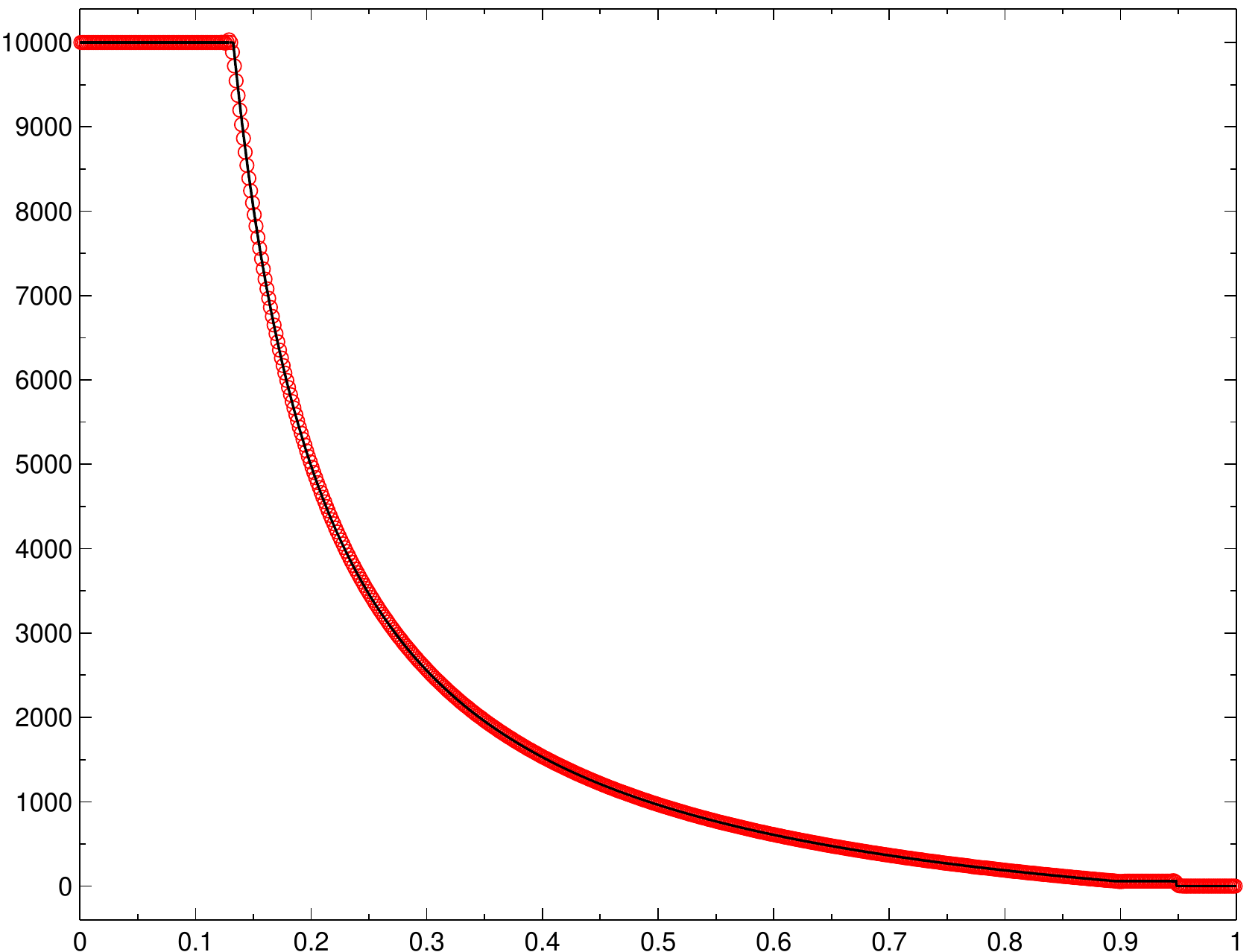}}
  {\includegraphics[width=0.48\textwidth]{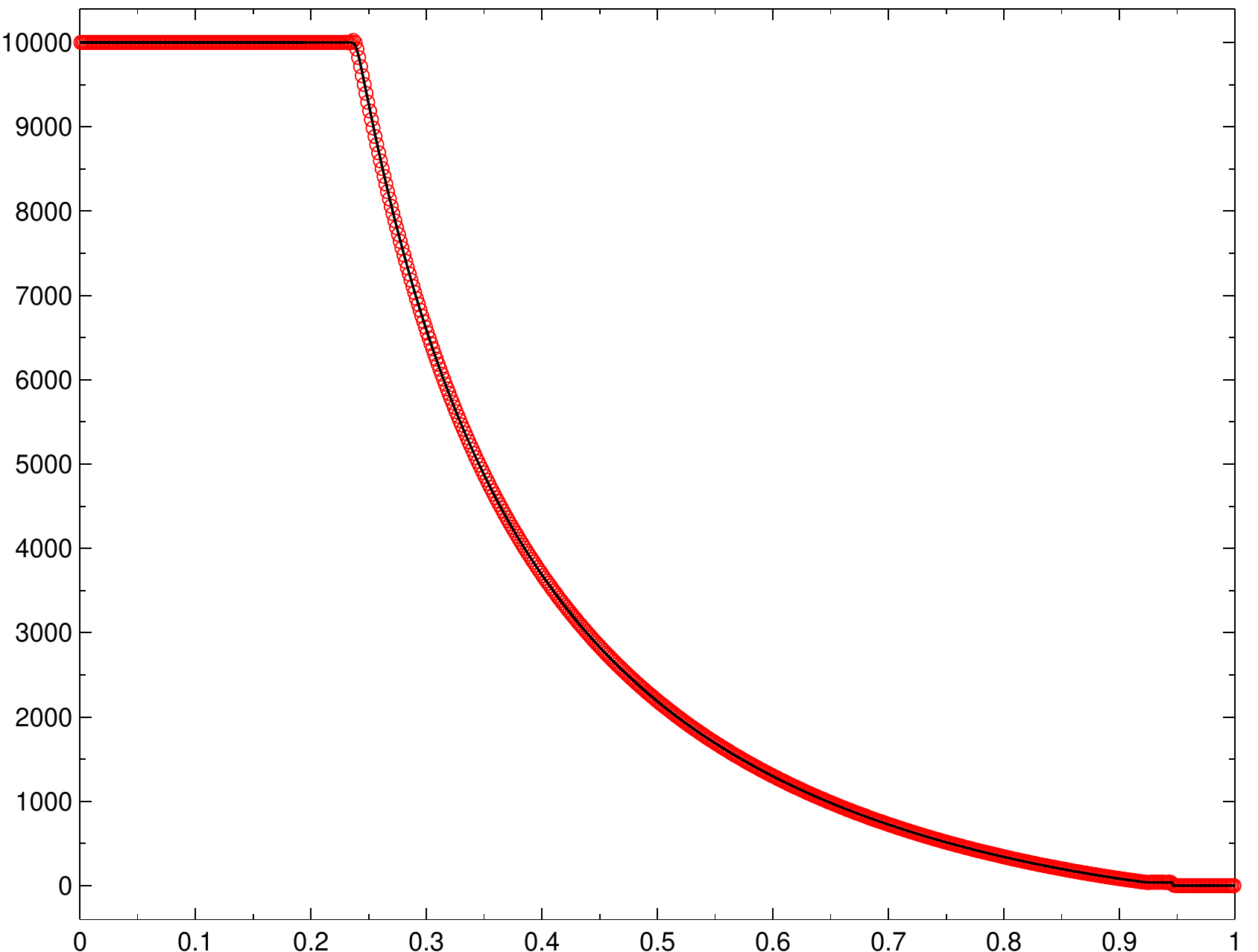}}
  \caption{\small Example \ref{example1DRiemann}: The density $\rho$,
  	velocity $v_1$, and  pressure $p$ at $t=0.45$ obtained by using  {\tt PCPMSCDGP2}
  	with 640 uniform cells.
  	Left: ideal EOS \eqref{eq:EOSideal} with $\Gamma=5/3$;
  	right: EOS \eqref{eq:EOS:Sokolov}.}
  \label{fig:1DRiemann}
\end{figure}

\begin{figure}[htbp]
  \centering
  {\includegraphics[width=0.48\textwidth]{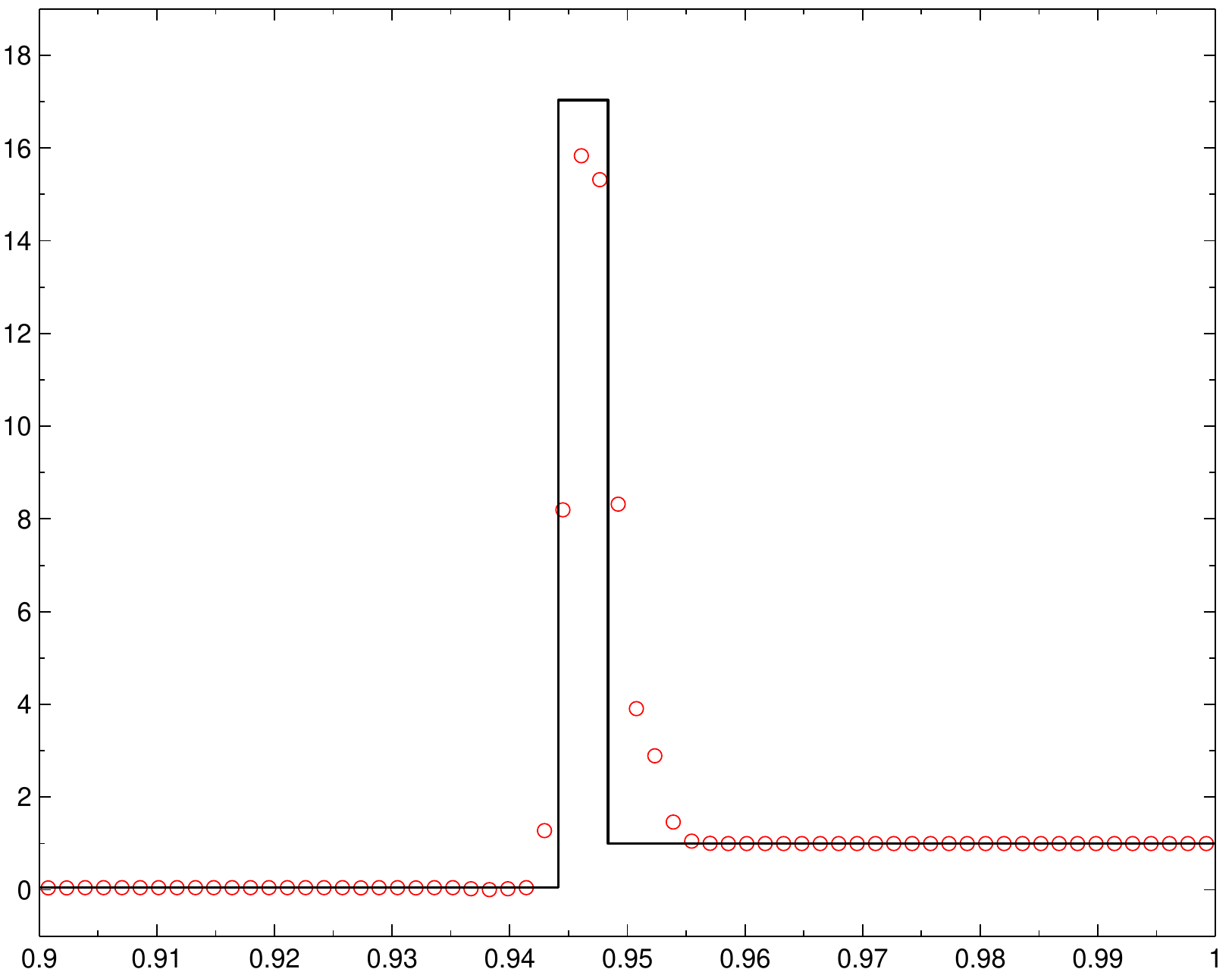}}
  {\includegraphics[width=0.48\textwidth]{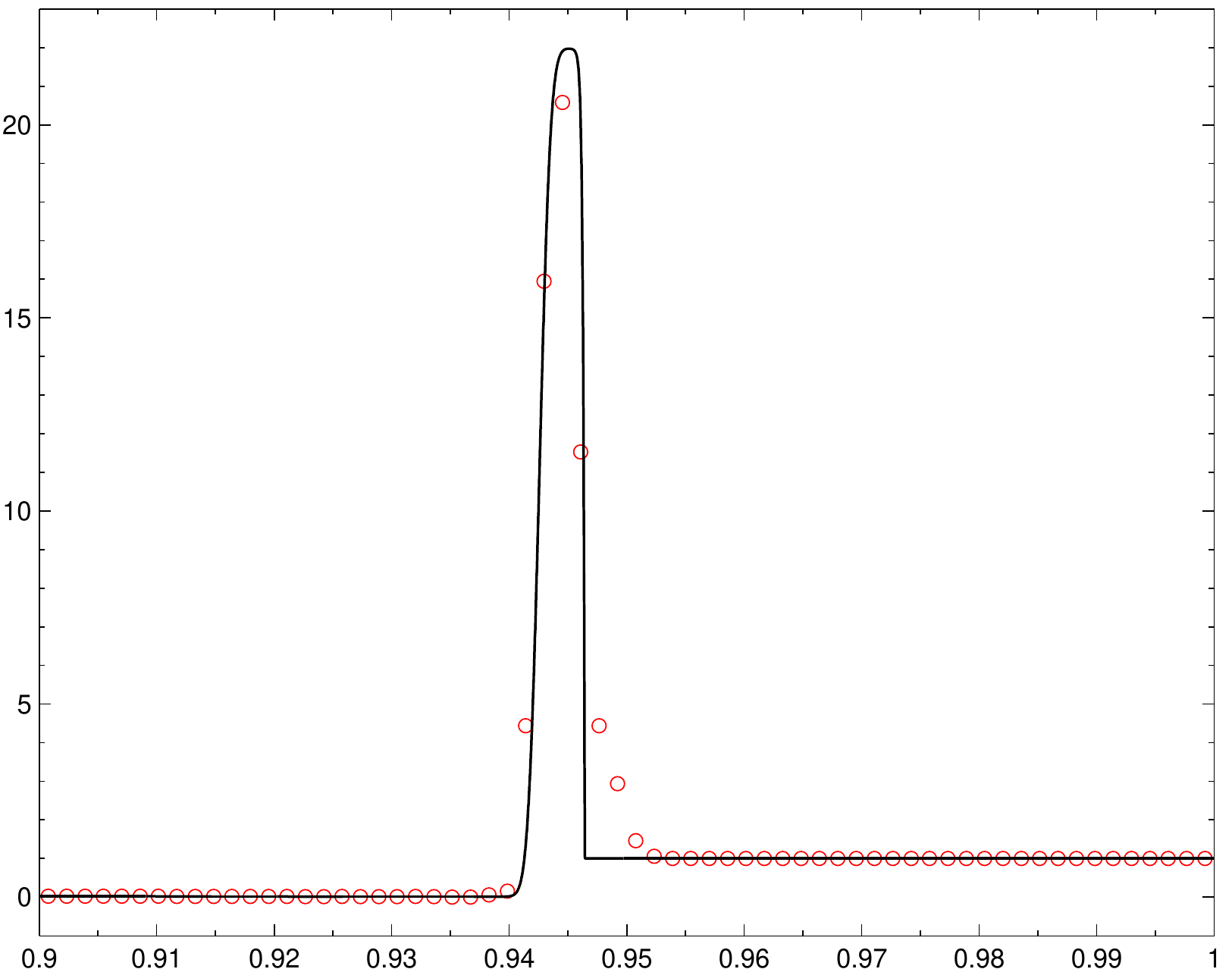}}
  \caption{\small Same as Fig. \ref{fig:1DRiemann} except for the close-up of rest-mass density.}
  \label{fig:1DRiemann:rho}
\end{figure}

\begin{example}[1D Riemann problem] \label{example1DRiemann}\rm
The  initial data of 1D RHD Riemann problem considered here are 
  \begin{equation}
    \label{eq:1DRiemann}
    \vec V(0,x)=
    \begin{cases}
      (1,0,10^4)^{\rm T},\ \ & x < 0.5,\\
      (1,0,10^{-8})^{\rm T},\ \ & x>0.5.
    \end{cases}
  \end{equation}
The initial discontinuity will evolve as a strong left-moving rarefaction wave,
 a quickly right-moving contact discontinuity  and a  shock wave.
The speeds of  the  contact discontinuity and shock wave are about 0.986956 and 0.9963757 respectively
for the ideal gas with $\Gamma=5/3$, see \citep{WuTang2015},
so that they are very close to the speed of light  and this test becomes  very ultra-relativistic.

Fig. \ref{fig:1DRiemann} displays the numerical results at $t=0.45$ obtained by using {\tt PCPMSCDGP2} (``{$\circ$}'') with  640 uniform cells
within the domain $[0,1]$, where the solid lines denote the exact solutions   \citep{exRP} for the
ideal EOS, and reference solutions for the EOS \eqref{eq:EOS:Sokolov}.
The close-ups of rest-mass densities are displayed in Fig. \ref{fig:1DRiemann:rho}.
Since it is difficult to get the exact solution  for a general EOS,
our reference solutions are numerically obtained  by using the Lax-Friedrichs scheme
over a very fine mesh of $100000$ uniform cells.
It is worth emphasizing that  the width of  region between
the contact discontinuity and  shock wave
at $t=0.45$ is about $4\times10^{-3}$
so that it is not easy to well resolve the contact discontinuity and  shock wave
with 640 uniform cells in the domain $[0,1]$.
From Figs. \ref{fig:1DRiemann} and \ref{fig:1DRiemann:rho},
we see that
{\tt PCPMSCDGP2} exhibits very good resolution and well
captures the wave configuration in the extremely narrow region between
the contact discontinuity and  shock wave, in comparison with
the fifth- and ninth-order accurate finite difference WENO schemes \citep{WuTang2015};
 the maximal densities  for {\tt PCPMSCDGP2} within the narrow region  between
the contact discontinuity and  shock wave  are about 92.98\%
of the analytic value for the ideal EOS \eqref{eq:EOSideal},
and 93.67\% of the reference value for the EOS \eqref{eq:EOS:Sokolov}, respectively;
the nonlinear addition of velocities yields a curved profile
for the rarefaction fan, as opposed to a linear one in the non-relativistic case.
and the wave configurations in Fig. \ref{fig:1DRiemann} for two EOS are different.
If  the PCP limiteing procedure is not employed,
then the high-order accurate central DG methods will break down  quickly
after  few time steps due to nonphysical numerical solutions.
\end{example}

\begin{figure}[htbp]
  \centering
  {\includegraphics[width=0.48\textwidth]{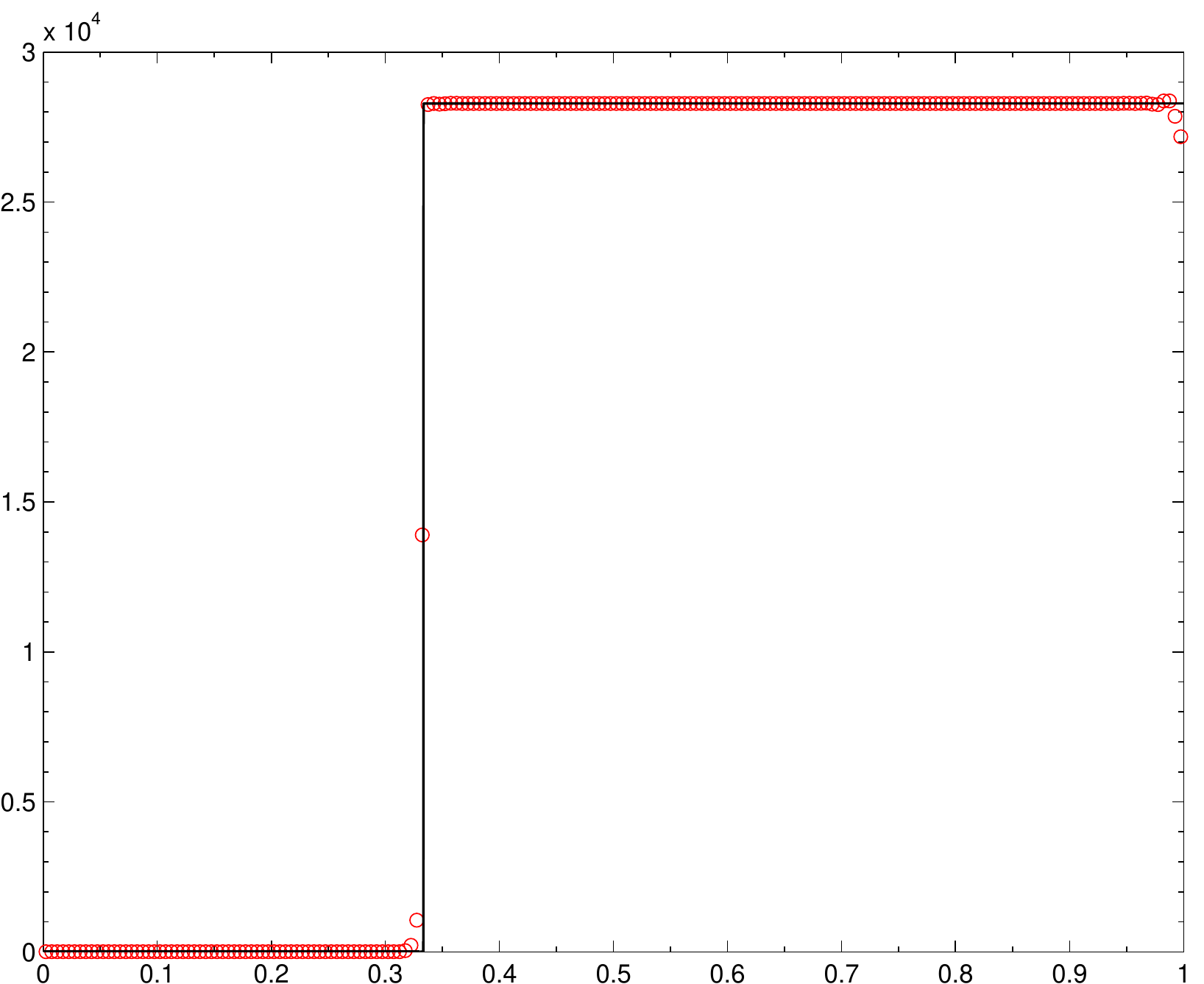}}
  {\includegraphics[width=0.48\textwidth]{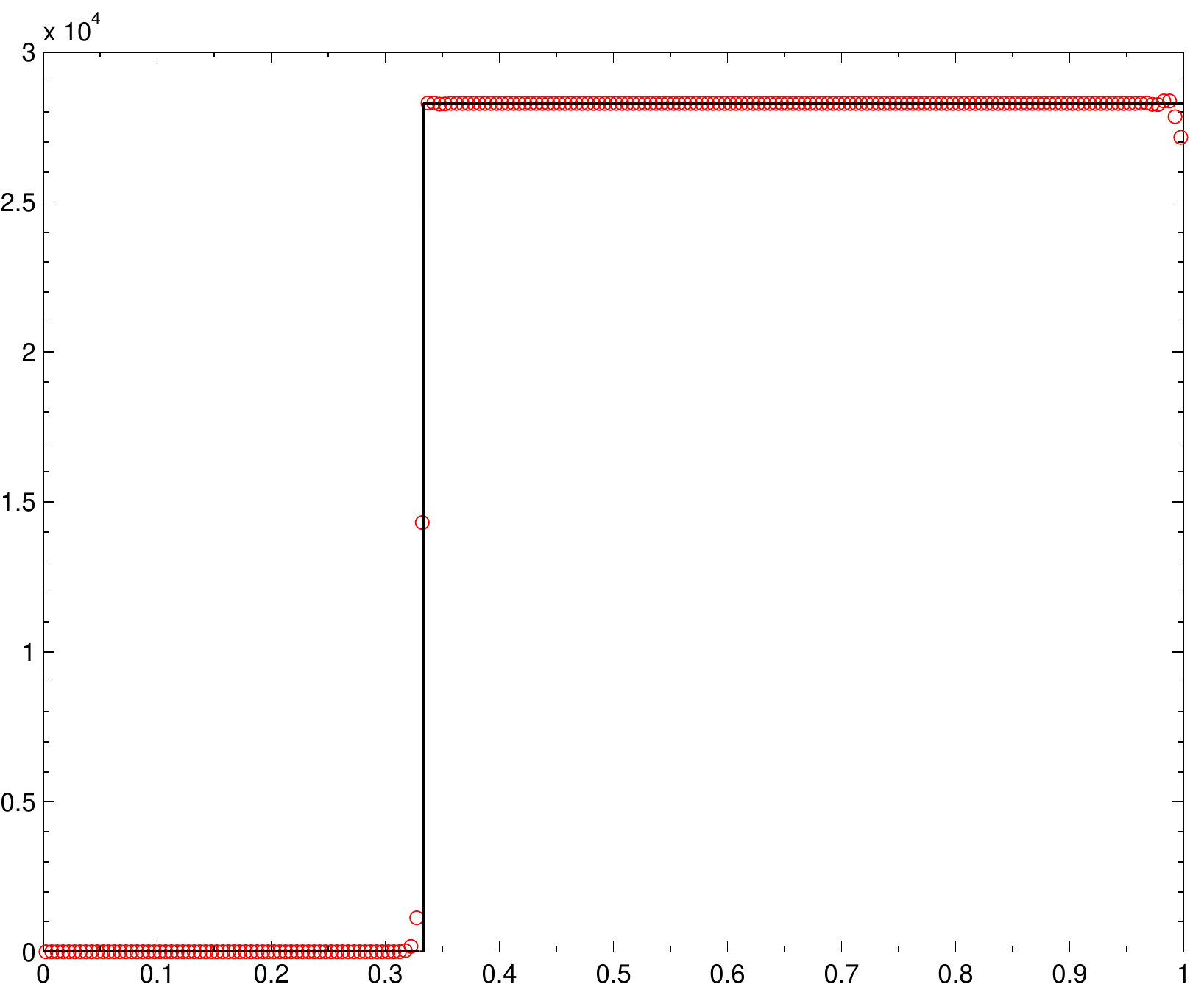}}
  {\includegraphics[width=0.48\textwidth]{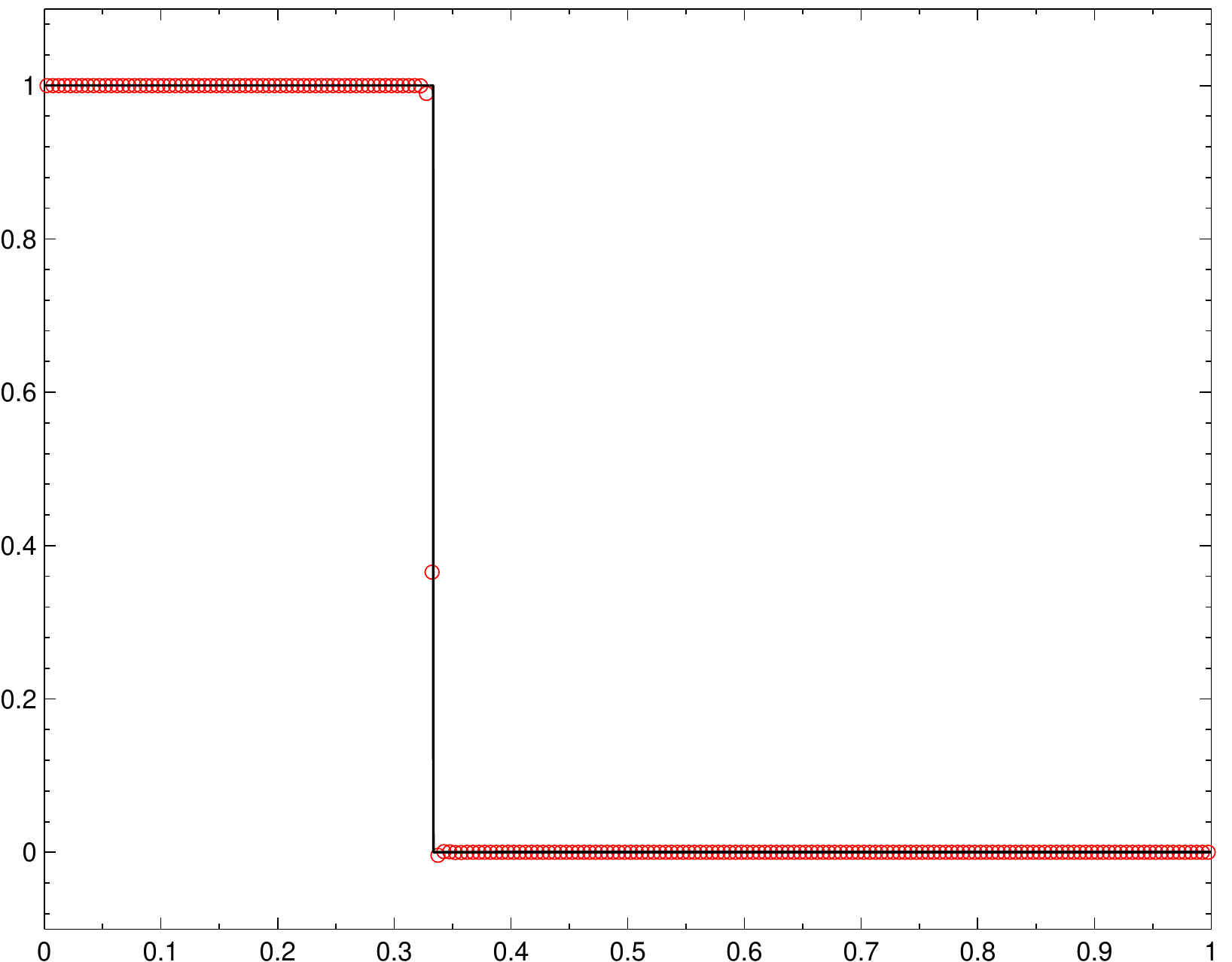}}
  {\includegraphics[width=0.48\textwidth]{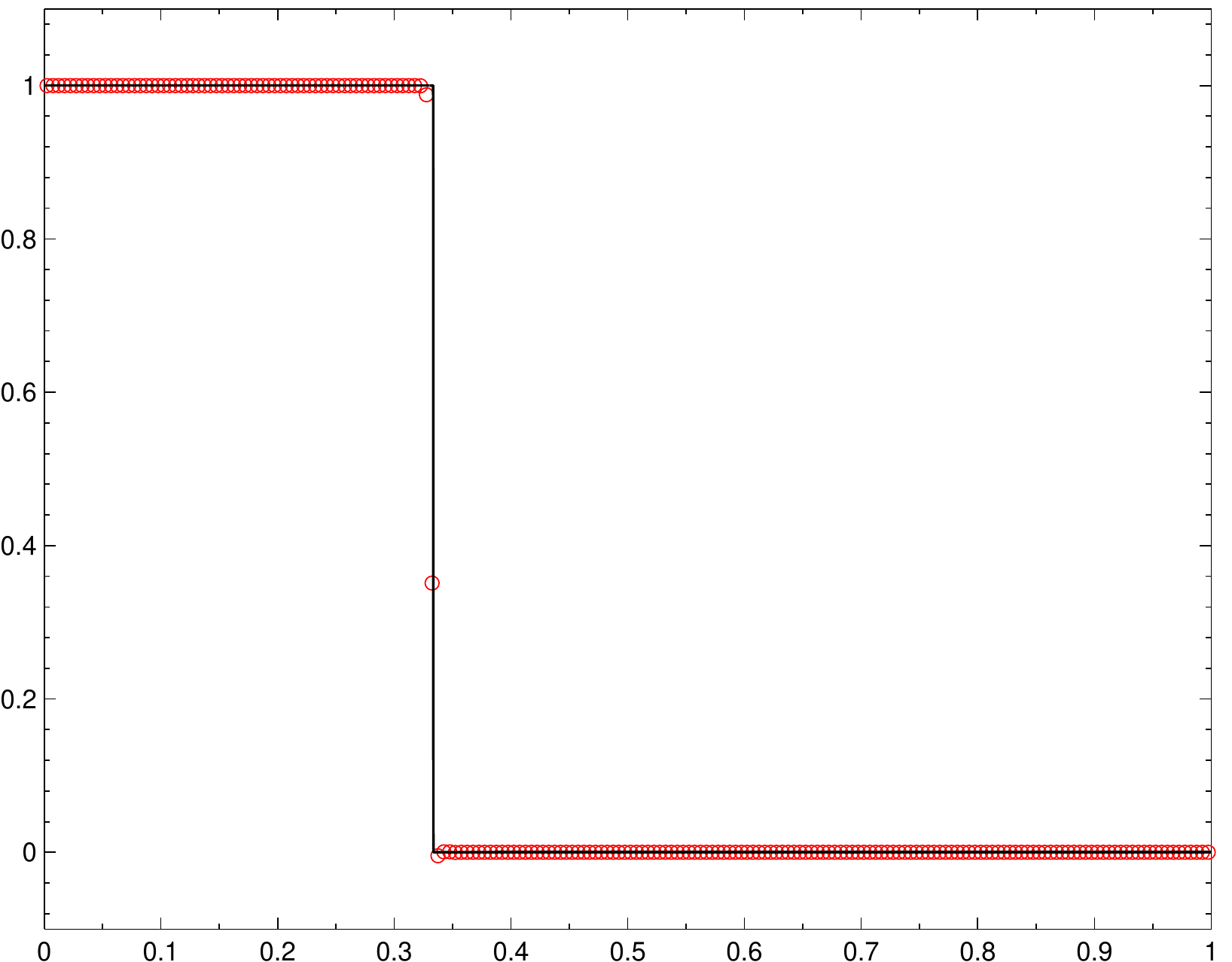}}
  {\includegraphics[width=0.48\textwidth]{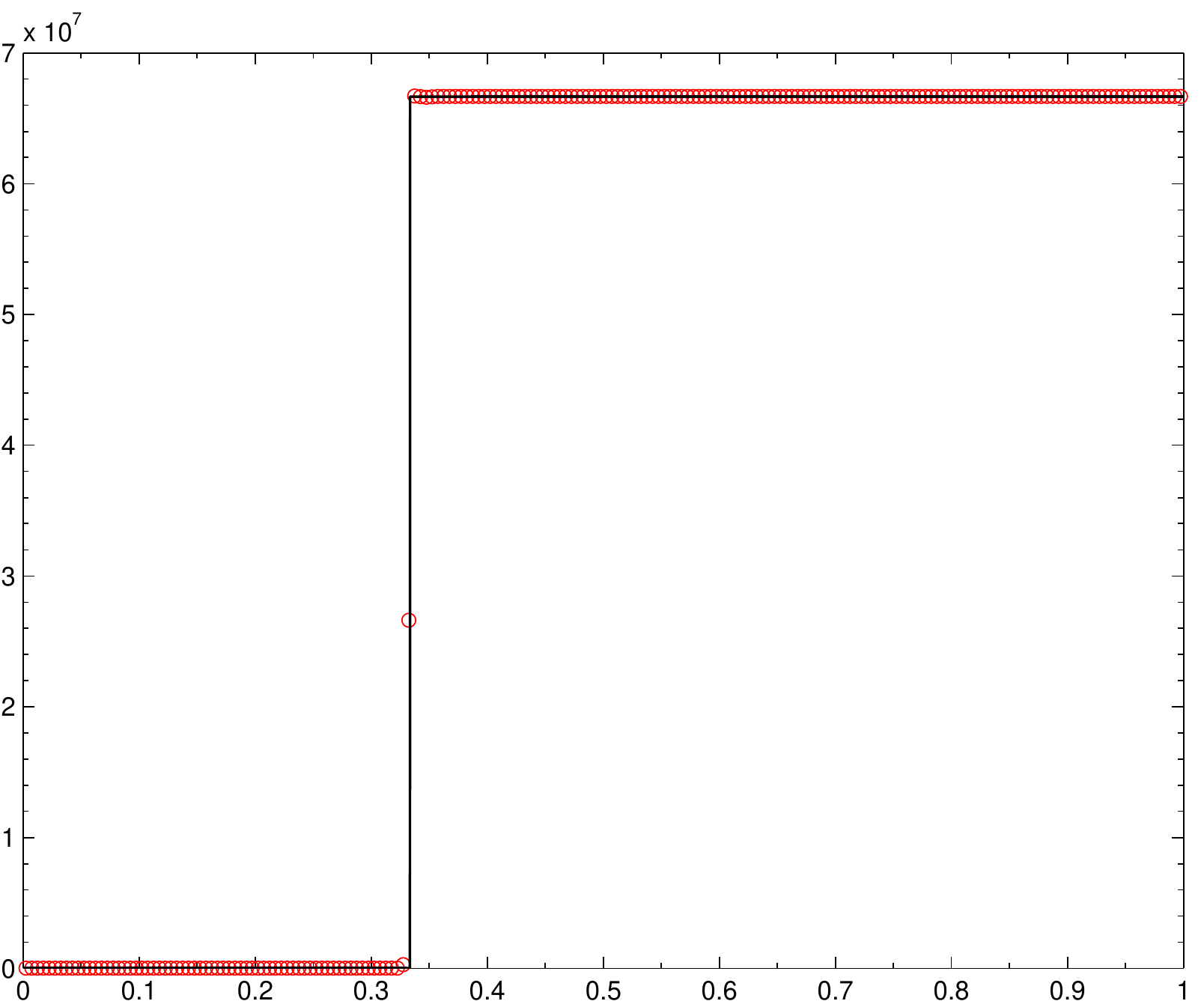}}
  {\includegraphics[width=0.48\textwidth]{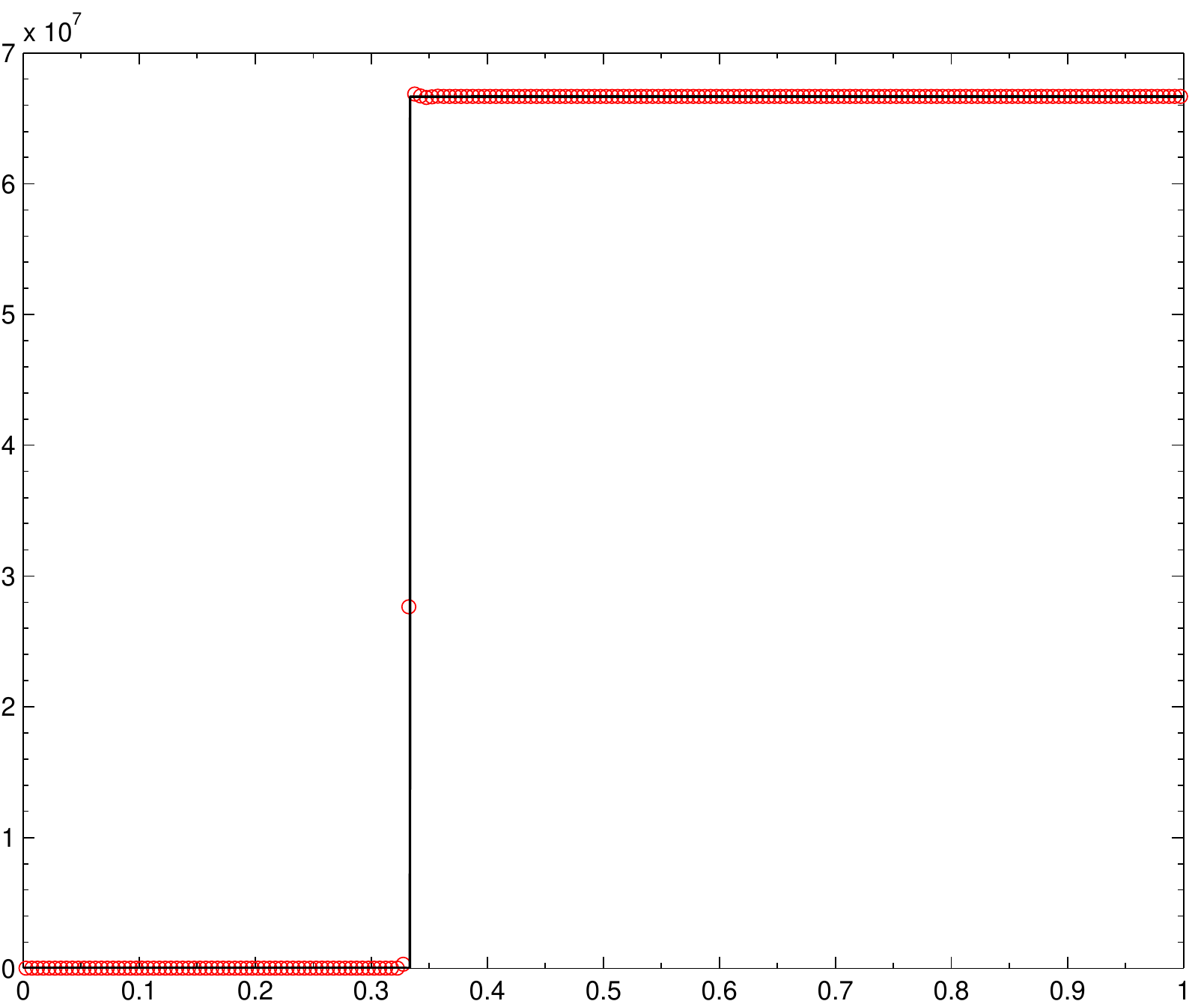}}
  \caption{\small Example \ref{exampleSH}: The density $\rho$,  velocity $v_1$, and
  	 pressure $p$ at $t=2$ obtained by using  {\tt PCPMSCDGP2} (``{$\circ$}'')
  	 with  200 uniform cells.
  	 The solid lines denote the exact solutions.
  	 Left: ideal EOS \eqref{eq:EOSideal} with $\Gamma=4/3$;
  	 right:  EOS \eqref{eq:EOS:Ryu}.}
  \label{fig:SH}
\end{figure}

\begin{example}[Shock heating problem] \label{exampleSH}\rm
The  test is to solve the shock heating problem \citep{Blandford1976}.
The computational domain $[0,1]$ with a  reflecting boundary
at $x=1$
is initially filled with a cold gas (the specific internal energy is nearly zero
and taken as 0.0001 in the computations),
which
 has an unit rest-mass density and the velocity $v_0$ of $1-10^{-8}$.
When the initial gas moves toward to the reflecting boundary, the gas is compressed and
  heated as the kinetic energy is converted into the internal energy.
After then, a reflected strong shock wave is formed and propagates to the left.
 Behind the reflected shock wave, the gas is at rest and has a specific internal energy of $W_0 - 1$  due to the energy conservation across the shock wave, $W_0=(1-v_0^2)^{-1/2}$ is about 7071.07. The compression ratio $\sigma$ across the relativistic shock wave is about $
\sigma \approx 4 W_0 + 3
\approx 28287.27,$
and grows linearly to the infinite  with the Lorentz factor $W_0$  when $v_0$ tends to speed of light $c$.
It is worth noting that the compression ratio
across the non-relativistic shock wave  is always
bounded, e.g. by ${(\Gamma+1)}/{(\Gamma -1)}$ for the ideal gas.

Here we will consider  the ideal EOS with the adiabatic index  $\Gamma$ of $4/3$
and the EOS \eqref{eq:EOS:Ryu}.
Fig. \ref{fig:SH} displays
the numerical solutions at $t=2$ obtained by using  {\tt PCPMSCDGP2}  (``{$\circ$}'')  with 200 uniform cells.
It is seen that {\tt PCPMSCDGP2}  exhibits good robustness for this ultra-relativistic problem
and  high resolution for the strong shock wave, 
even though there exists the well-known wall-heating phenomenon  near the reflecting boundary $x=1$.
The difference between two  different EOS is very small because of the very low specific internal energy.
In this test,  it is also necessary for  the successful  performance of  the high-order accurate central DG methods
to use the PCP limiting procedure.

\end{example}

\begin{example}[Blast wave interaction] \label{exampleBW}\rm
It is an initial-boundary-value problem for the RHD equations \eqref{eqn:coneqn3d}
with $d=1$
and  very severe  due to the strong relativistic shock waves and
interaction between blast waves  in a narrow region  \citep{Marti3,YangHeTang2011,WuTang2015}.
The  initial data are taken as
\begin{equation}
    \label{eq:BlastInteract}
    \vec V(0,x) =
    \begin{cases}
      (1,0,1000)^{\rm T},\ \ & 0<x< 0.1,\\
      (1,0,0.01)^{\rm T},\ \ & 0.1<x<0.9,\\
      (1,0,100)^{\rm T},\ \ &  0.9<x<1,
    \end{cases}
 \end{equation}
 and the outflow boundary conditions are specified at   two ends of the computational domain $[0,1]$.

\begin{figure}[htbp]
  \centering
  {\includegraphics[width=0.48\textwidth]{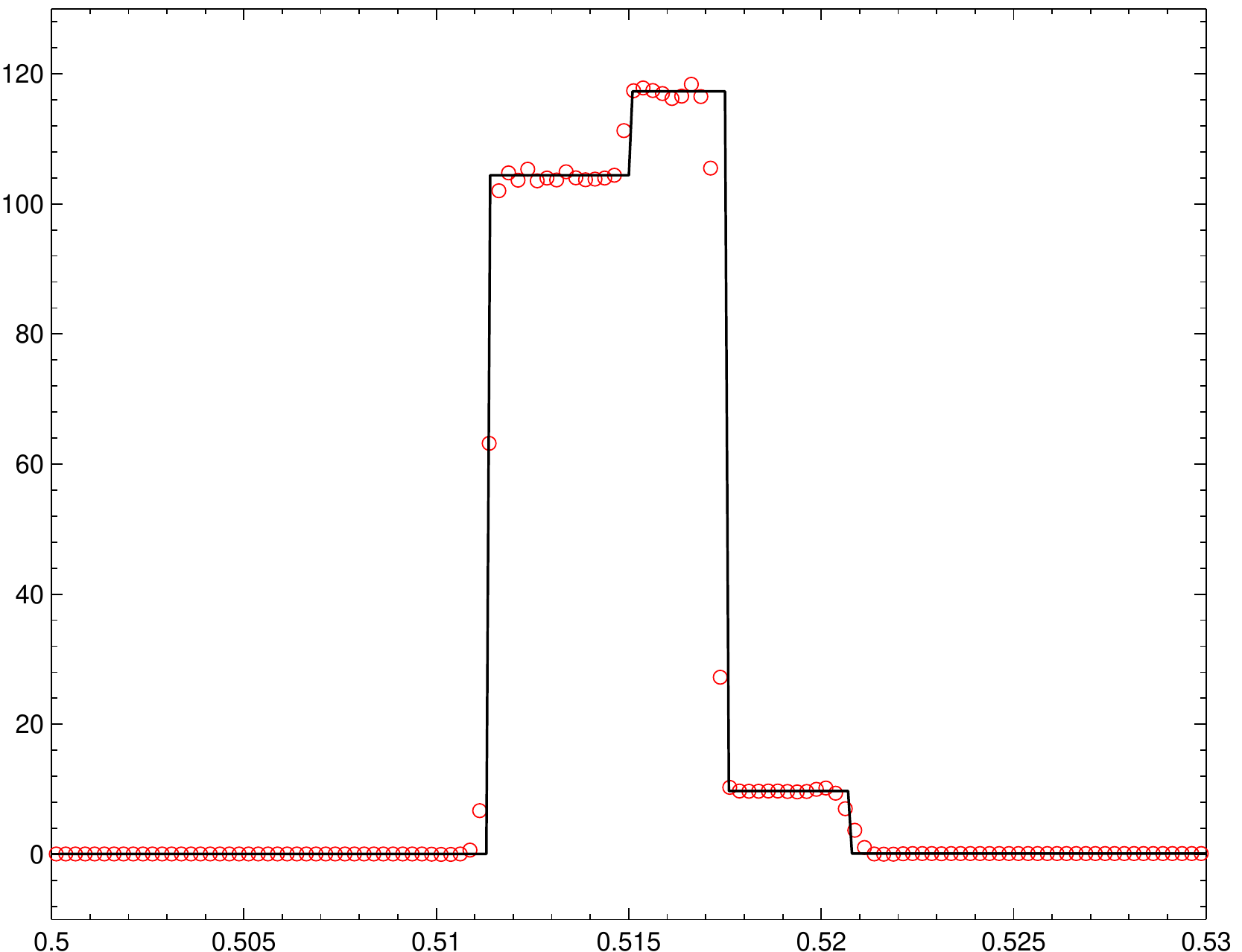}}
  {\includegraphics[width=0.48\textwidth]{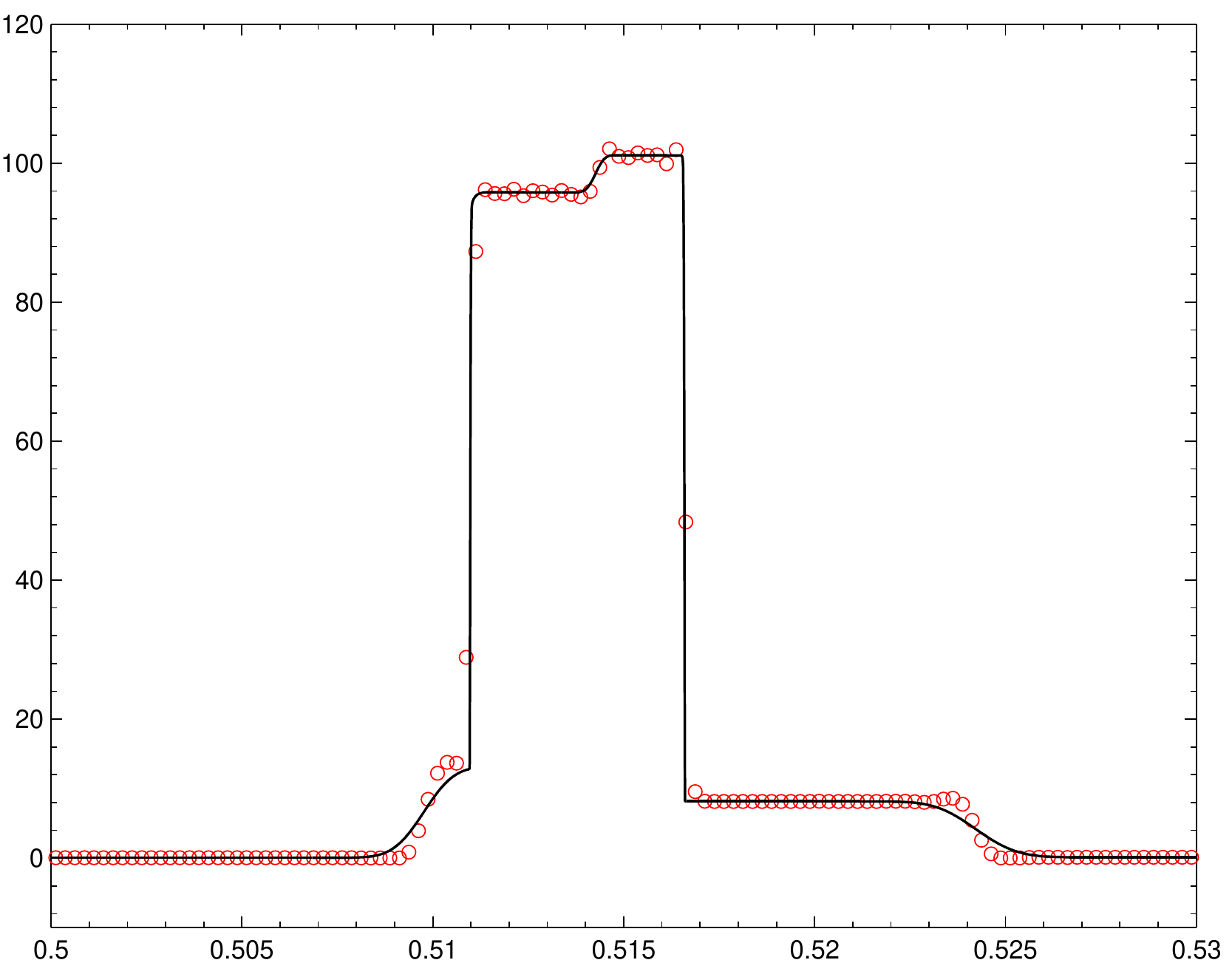}}
  {\includegraphics[width=0.48\textwidth]{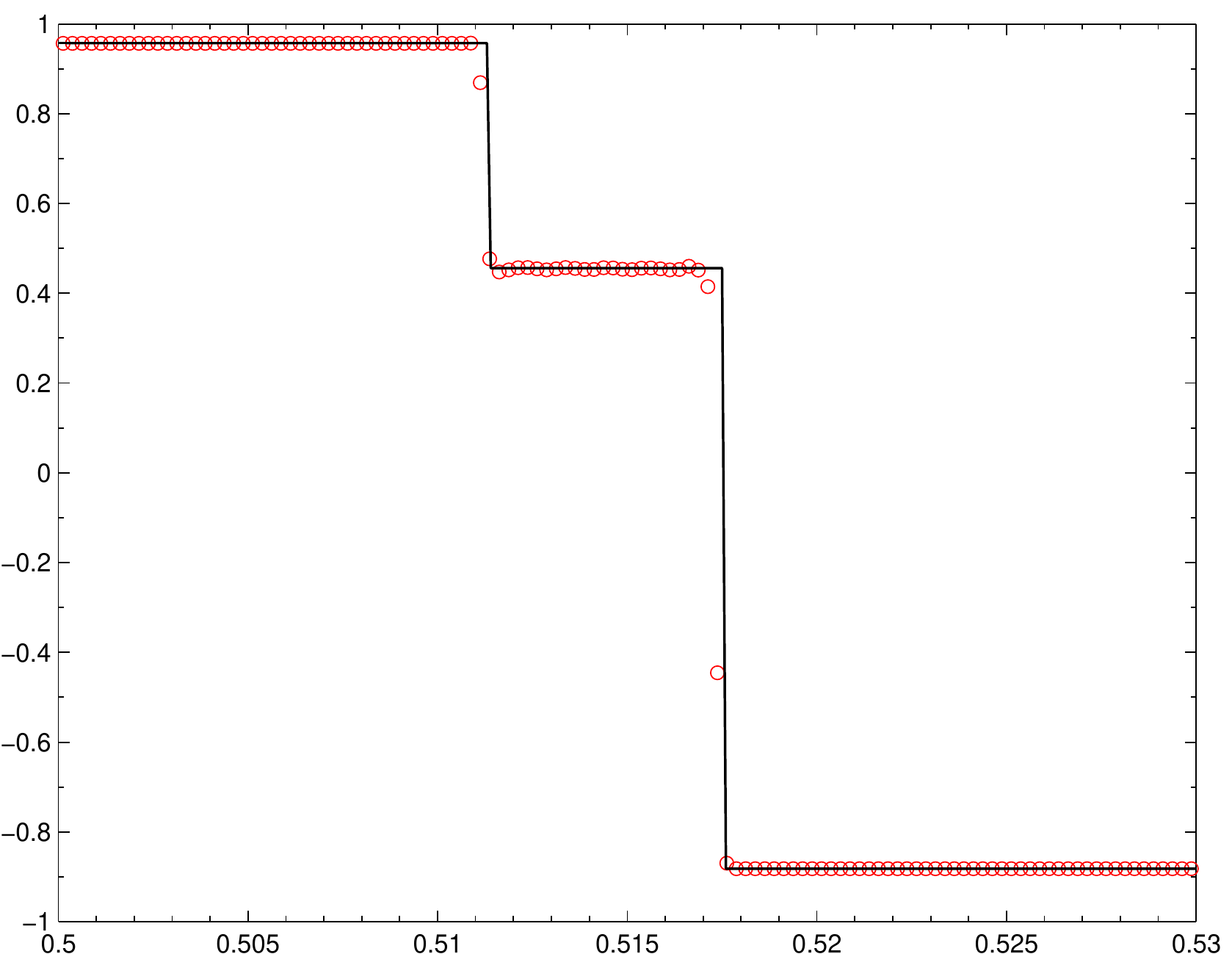}}
  {\includegraphics[width=0.48\textwidth]{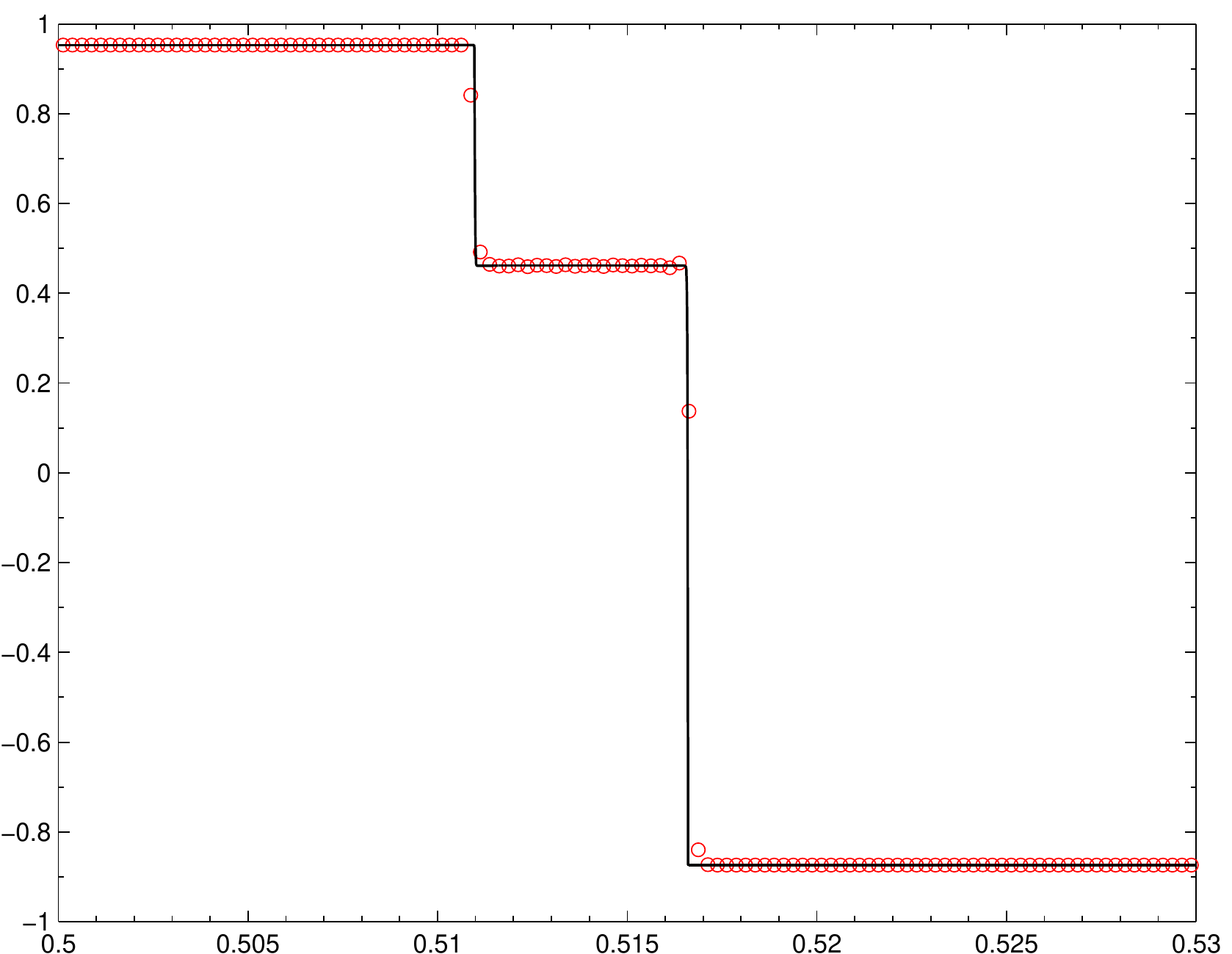}}
  {\includegraphics[width=0.48\textwidth]{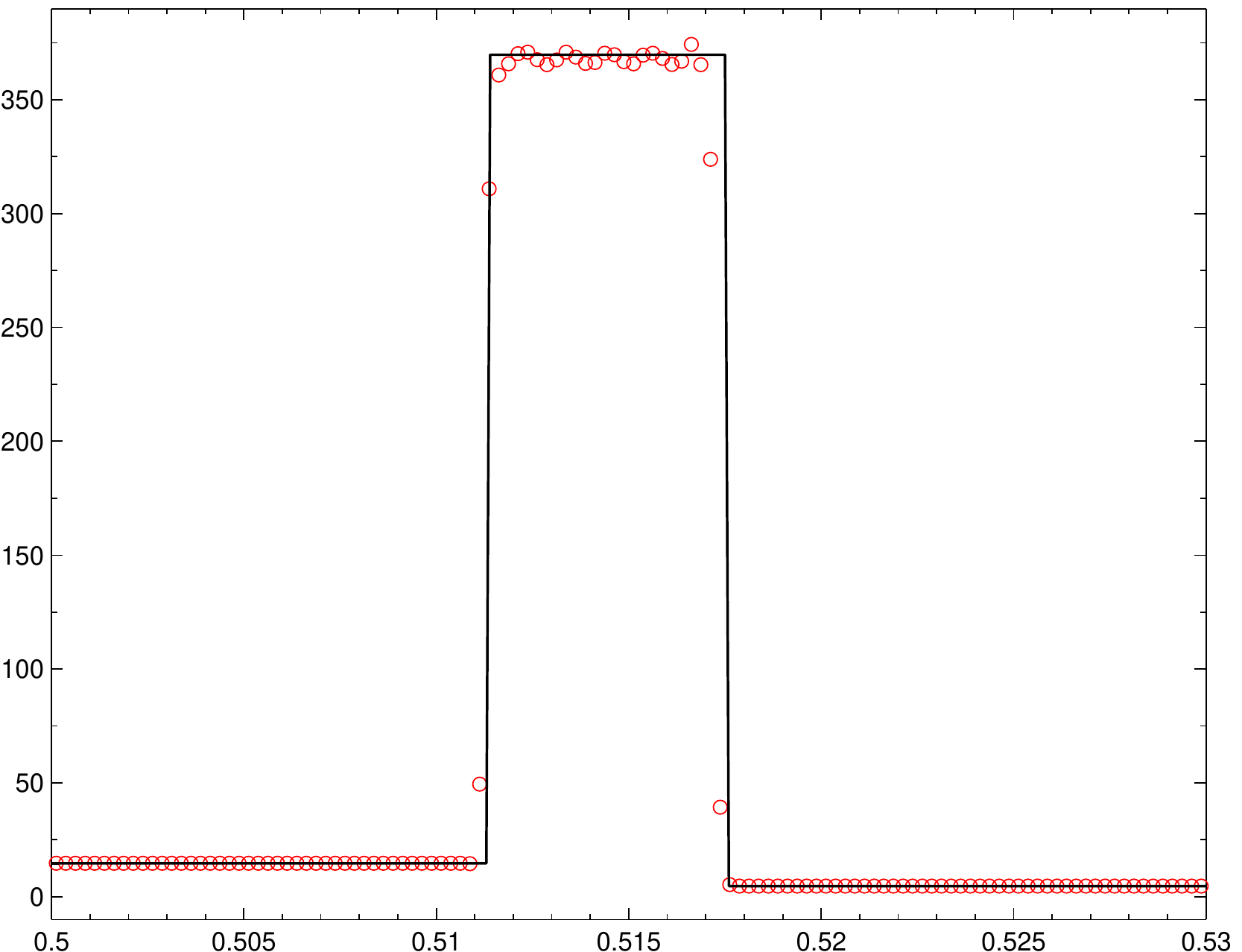}}
  {\includegraphics[width=0.48\textwidth]{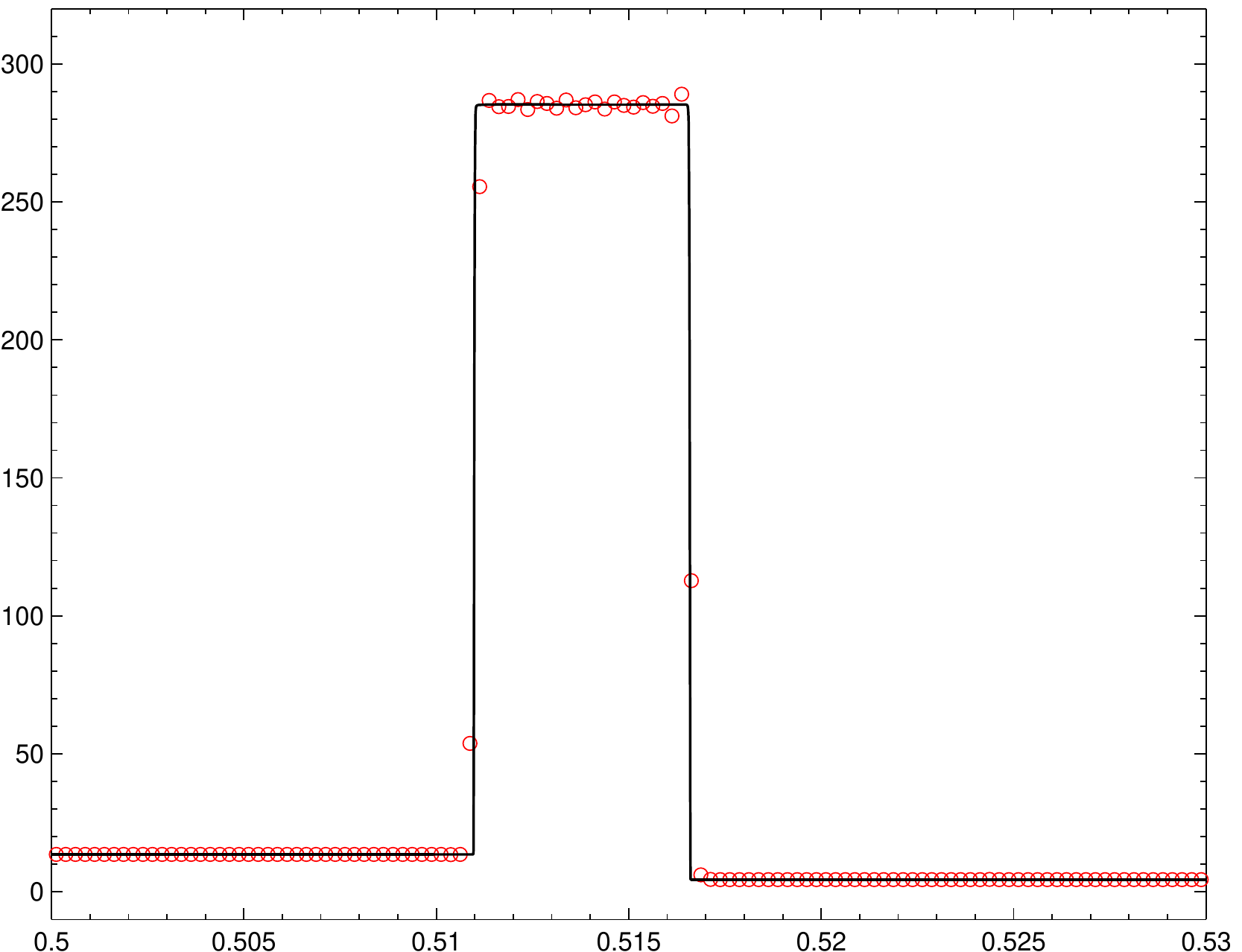}}
  \caption{\small Example \ref{exampleBW}: Close-up of the numerical solutions at $t=0.43$ obtained by using  {\tt PCPMSCDGP2} (``{$\circ$}'') with 4000 uniform cells. Left: ideal EOS \eqref{eq:EOSideal} with $\Gamma=1.4$; right: EOS \eqref{eq:EOS:Mathews}.
 }
  \label{fig:BWI}
\end{figure}


Fig.~\ref{fig:BWI} gives close-up of the solutions at $t=0.43$ obtained by using   {\tt PCPMSCDGP2}  (``{$\circ$}'')
with 4000 uniform cells within the domain $[0,1]$, where the solid lines denote the exact solutions
for  the ideal EOS \eqref{eq:EOSideal} with $\Gamma=1.4$, see \citep{Marti3},
and the reference solutions for the EOS \eqref{eq:EOS:Mathews} obtained
by using the Lax-Friedrichs scheme over a very fine mesh of $400000$ uniform cells.
It is found that  there are
two shock waves and two contact discontinuities
in the solutions at  $t=0.43$ within the interval $[0.5,0.53]$
since both initial discontinuities evolve and two blast waves collide each other;
and  
the proposed central DG methods  may well resolve
those discontinuities and clearly capture the complex relativistic wave configuration
except for small oscillations  between the left shock wave and  contact discontinuity.
The oscillations may be suppressed by locally using
the nonlinear limiter, e.g. the WENO limiter \citep{QiuWENOlimiter,Zhao},
see Fig. \ref{fig:BWIWENO}.  

\begin{figure}[htbp]
  \centering
  {\includegraphics[width=0.48\textwidth]{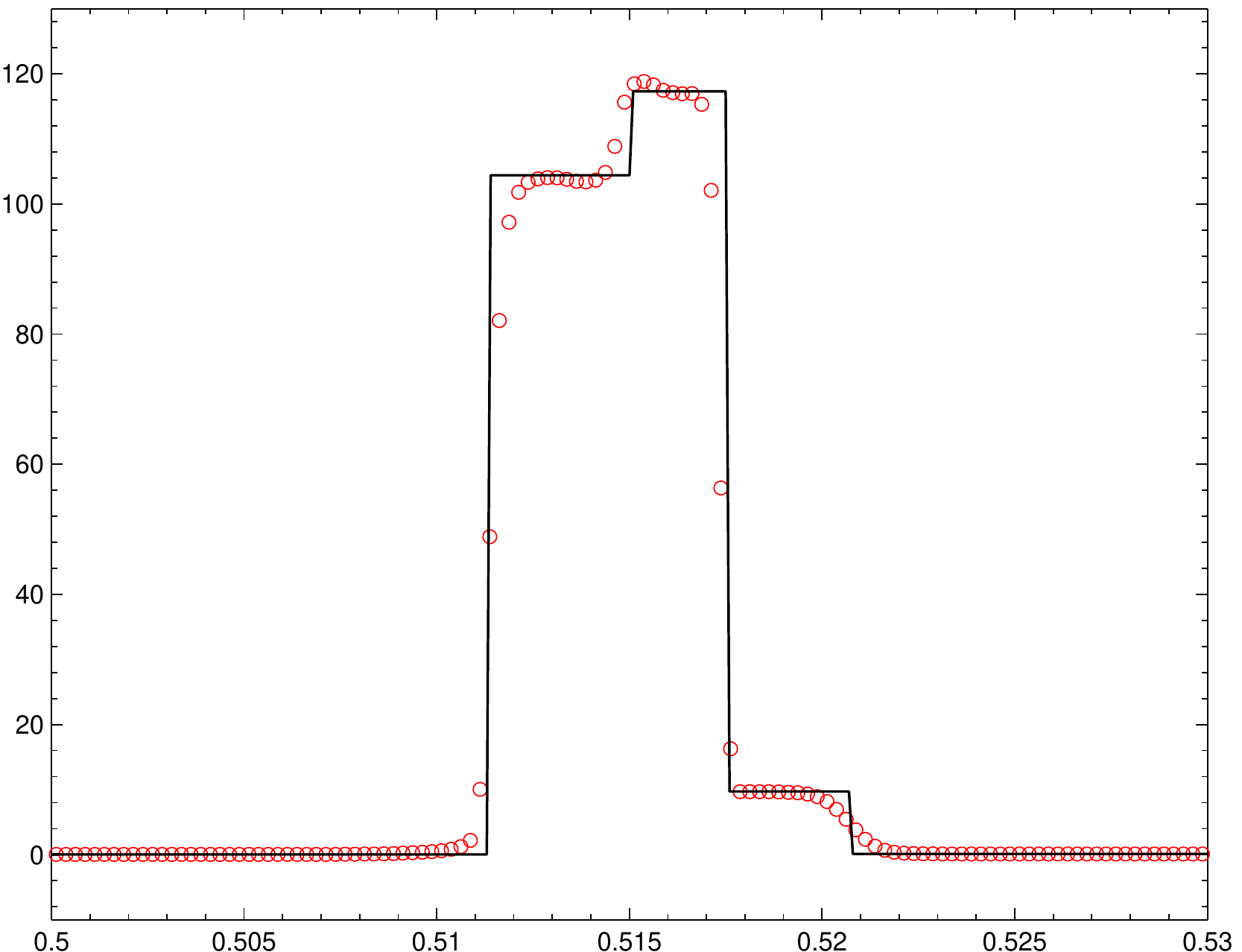}}
  {\includegraphics[width=0.48\textwidth]{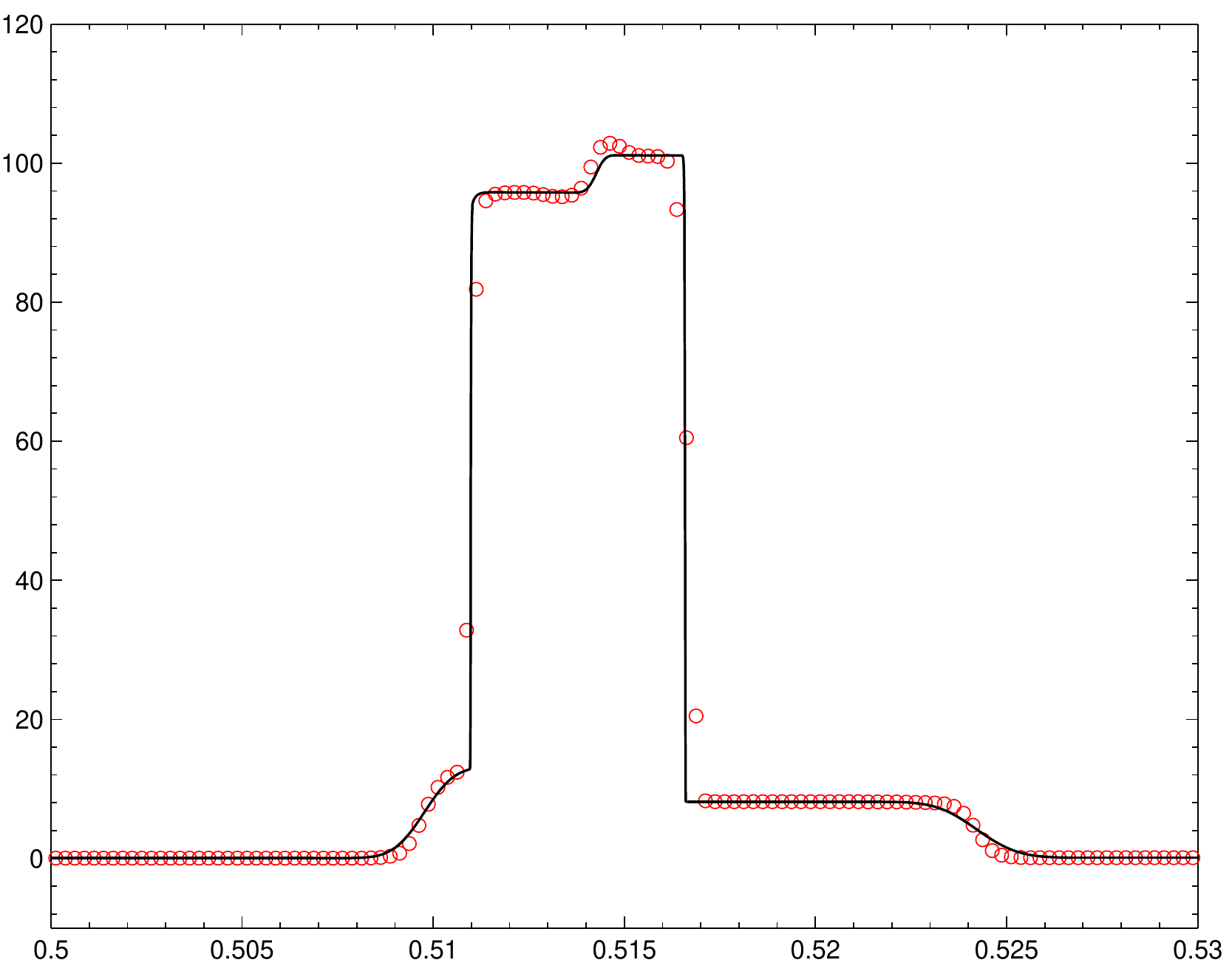}}
  {\includegraphics[width=0.48\textwidth]{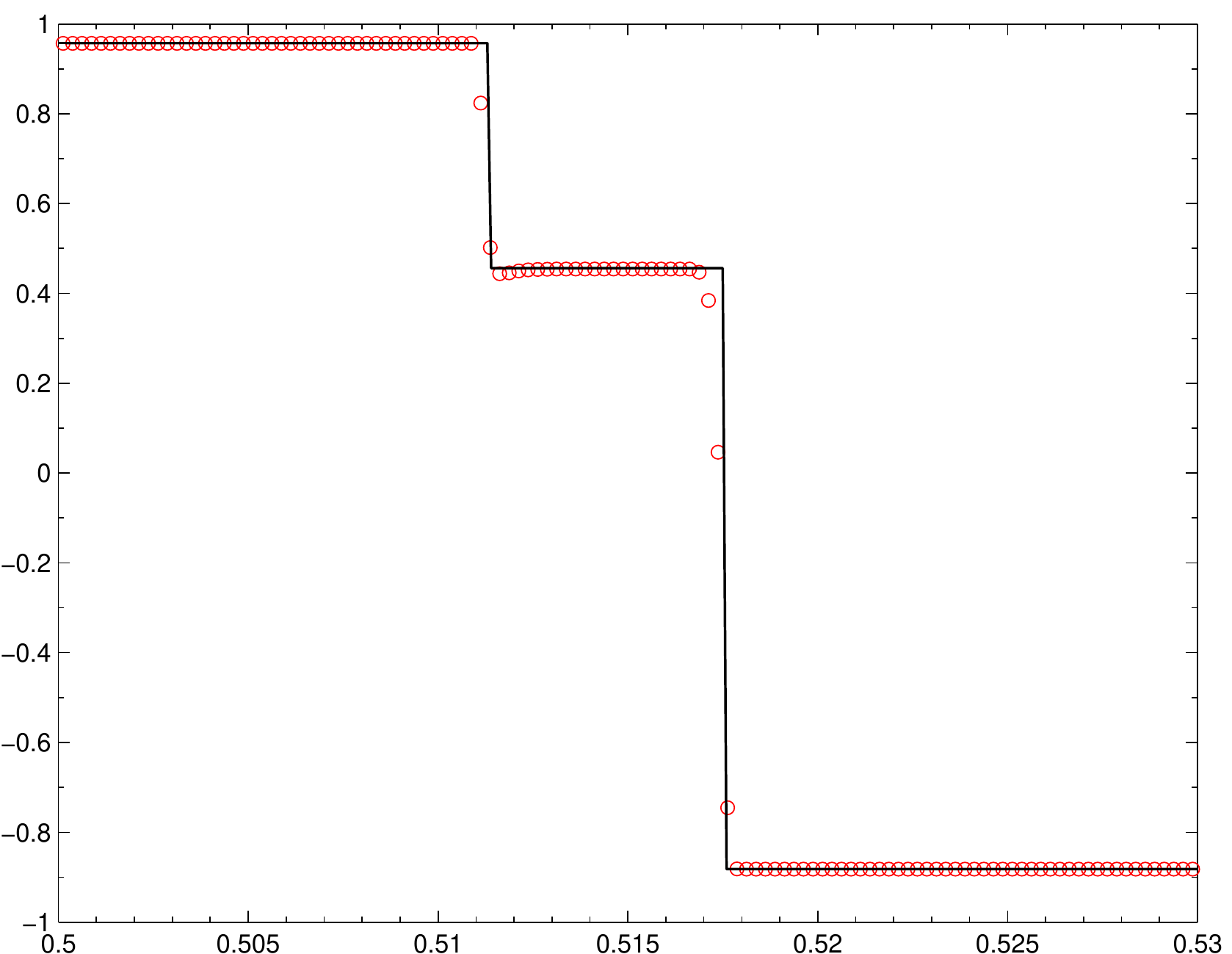}}
  {\includegraphics[width=0.48\textwidth]{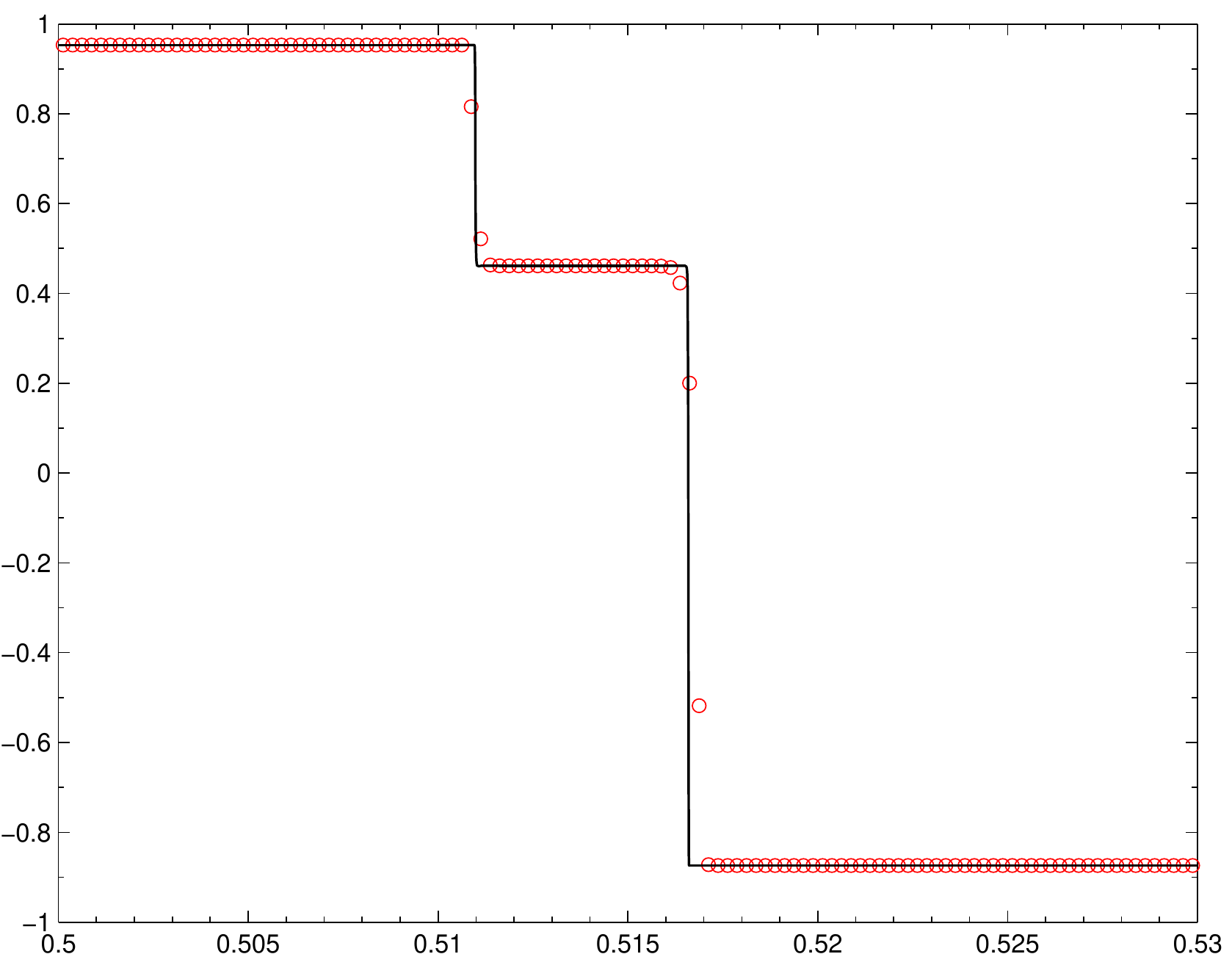}}
  {\includegraphics[width=0.48\textwidth]{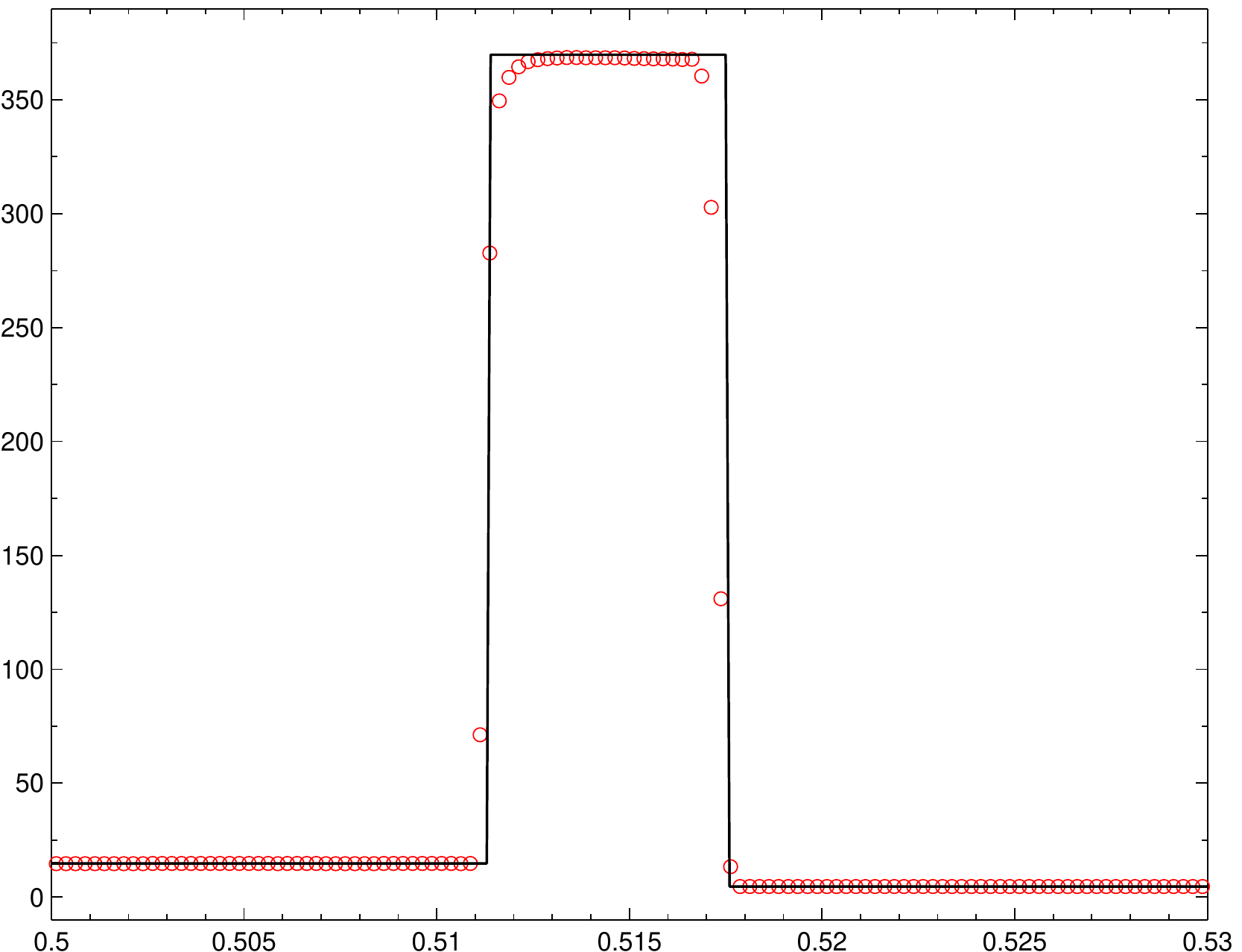}}
  {\includegraphics[width=0.48\textwidth]{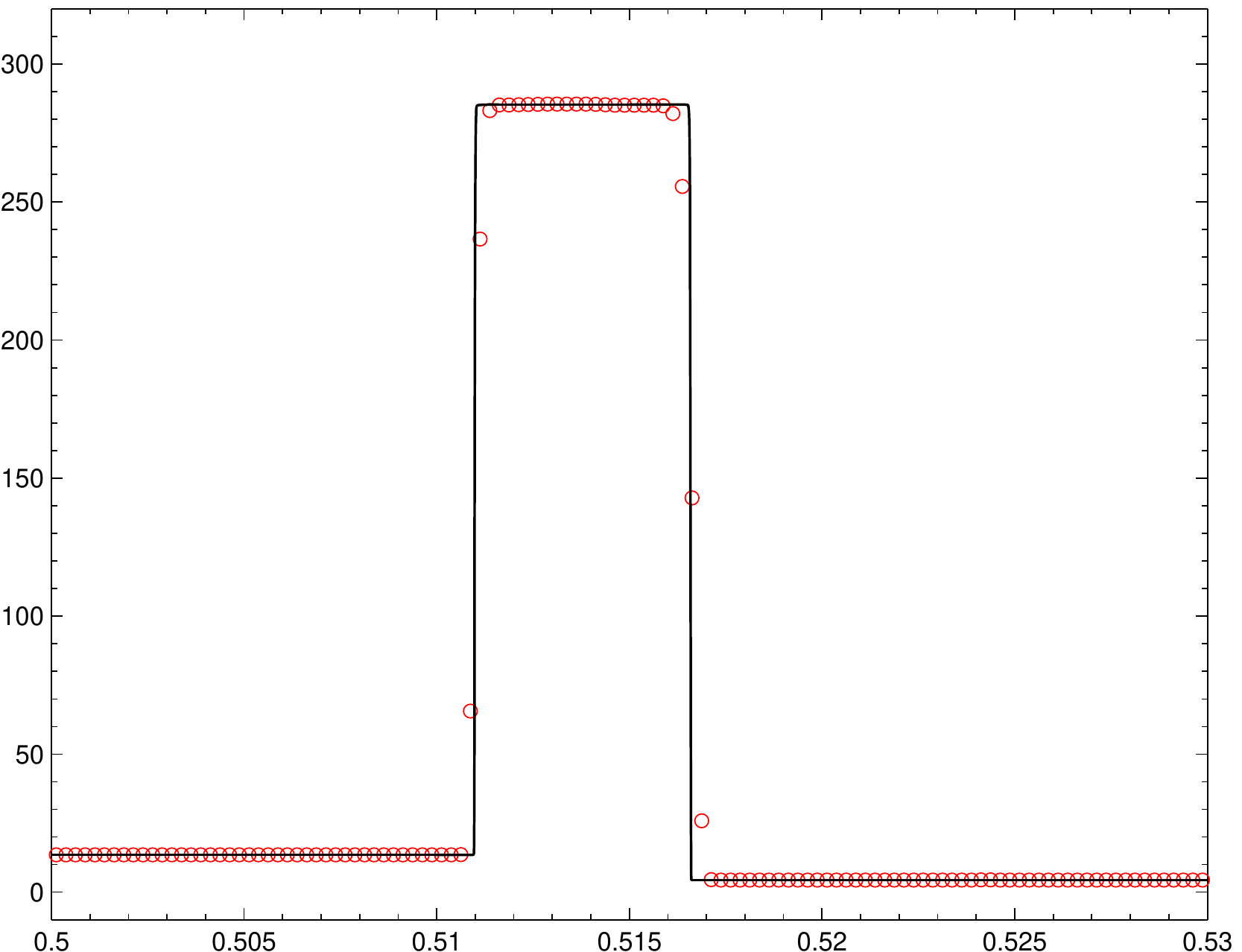}}
  \caption{\small Same as Fig. \ref{fig:BWI} except for locally using the WENO limiter.
 }
  \label{fig:BWIWENO}
\end{figure}

\end{example}

\subsection{2D case}

This section is to  conduct five 2D numerical experiments
on a smooth problem, two Riemann problems, and two relativistic jet flows.
Because the strong shock waves as well as their interaction appear in the last four problems,
the WENO limiter will be  implemented prior to the PCP limiting procedure
with the aid of the local characteristic decomposition  \citep{Zhao}.
Besides it may suppress spurious oscillations,
it can enhance the numerical  stability of  high-order accurate (central) DG methods.
Specially, when the WENO limiter is locally used,
 a larger time stepsize is allowed.
In all computations, the time stepsize  $\Delta t$
will be taken  as
 $ \frac{ {\varpi} \theta} {2c\left( 1/\Delta x + 1/\Delta y \right)}$
 with $\varpi = \hat{\omega}_1=\frac16$ for the first problem and  $ \varpi =1$
 for other problems.

\begin{example}[2D smooth problem] \label{example2Dsmooth}\rm
Similar to Example \ref{example1Dsmooth},
this smooth problem  is
used to check the accuracy of  proposed 2D PCP central DG methods.
 The initial data are taken as
$$
	\vec V(0,x,y)=\big( 1+0.99999\sin(2 \pi (x+y) ), 0.99/\sqrt{2},0.99/\sqrt{2}, 10^{-2}\big)^{\rm T},
$$
so that the exact solutions are
$$
	\vec V(t,x,y)=\big( 1+0.99999\sin(2 \pi (x+y-0.99 \sqrt{2} t), 0.99/\sqrt{2},0.99/\sqrt{2}, 10^{-2}\big)^{\rm T},
$$
which describe a RHD sine wave propagating periodically in the domain $\Omega=[0,1] \times [0,1]$
at an angle $45^\circ$ with the $x$-axis.
The  domain $\Omega$ is divided into $N\times N$ uniform cells and
the periodic boundary conditions are specified on $\partial \Omega$.

The ideal EOS \eqref{eq:EOSideal} with $\Gamma=\frac53$ is first considered.
Table \ref{tab:2Daccuracy}
lists  the $l^1$ and $l^2$-errors at $t=0.2$ and  corresponding orders
obtained  by using {\tt PCPRKCDGP2} and {\tt PCPMSCDGP2}, respectively.
The results show that the theoretical orders are obtained by both {\tt PCPRKCDGP2} and {\tt PCPMSCDGP2}
and the PCP limiting procedure does not destroy the accuracy.
Plots of numerical errors in Fig. \ref{fig:2Dsmooth} further validate  the accuracy of
 both {\tt PCPRKCDGP2} and {\tt PCPMSCDGP2}  for the general EOS.

\begin{table}[htbp]
  \centering 
    \caption{\small Example \ref{example2Dsmooth}: Numerical  $l^1$- and $l^2$-errors
    	and orders   at $t=0.2$ of  {\tt PCPRKCDGP2} and {\tt PCPMSCDGP2} for the
    	ideal EOS with $\Gamma=5/3$.
  }
\begin{tabular}{|c||c|c|c|c||c|c|c|c|}
  \hline
\multirow{2}{8pt}{$N$}
 &\multicolumn{4}{c||}  {\tt PCPRKCDGP2}  &\multicolumn{4}{c|}  {\tt PCPMSCDGP2} \\
 \cline{2-9}
 &$l^1$ $ {\mbox{error}}$ &$l^1$ order &$l^2$ error &$l^2$ order  &$l^1$ error &$l^1$ order &$l^2$ error &$l^2$ order \\
 \hline
10 & 2.462e-3 & --     & 3.091e-3   &--    & 2.456e-3 & --    & 3.083e-3 &--\\
20& 2.573e-4 & 3.26 & 3.446e-4   & 3.17   & 2.568e-4 & 3.26 & 3.442e-4 & 3.16 \\
40& 3.131e-5 & 3.04  & 4.261e-5   & 3.02  & 3.054e-5 & 3.07 & 4.227e-5 & 3.03 \\
80& 3.785e-6 & 3.05  & 5.278e-6    & 3.01 & 3.769e-6 & 3.02 & 5.276e-6 &3.00\\
160& 4.707e-7 & 3.01 & 6.594e-7  & 3.00 & 4.707e-7 & 3.00 & 6.594e-7 &3.00\\
\hline
\end{tabular}\label{tab:2Daccuracy}
\end{table}

\begin{figure}[htbp]
  \centering
    \subfigure[EOS \eqref{eq:EOS:Mathews}]
    {\includegraphics[width=0.32\textwidth]{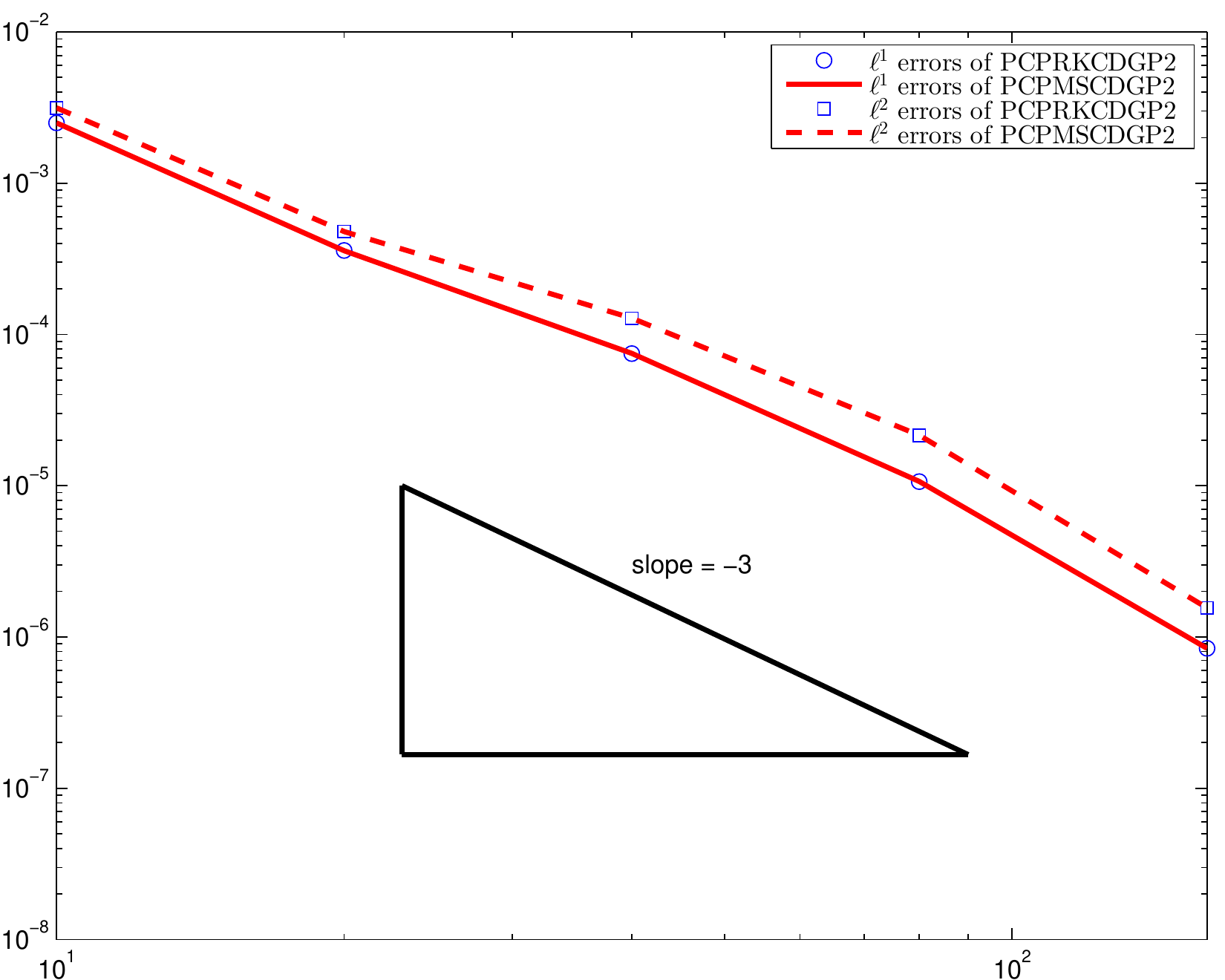}}
  \subfigure[EOS \eqref{eq:EOS:Sokolov}]
  {\includegraphics[width=0.32\textwidth]{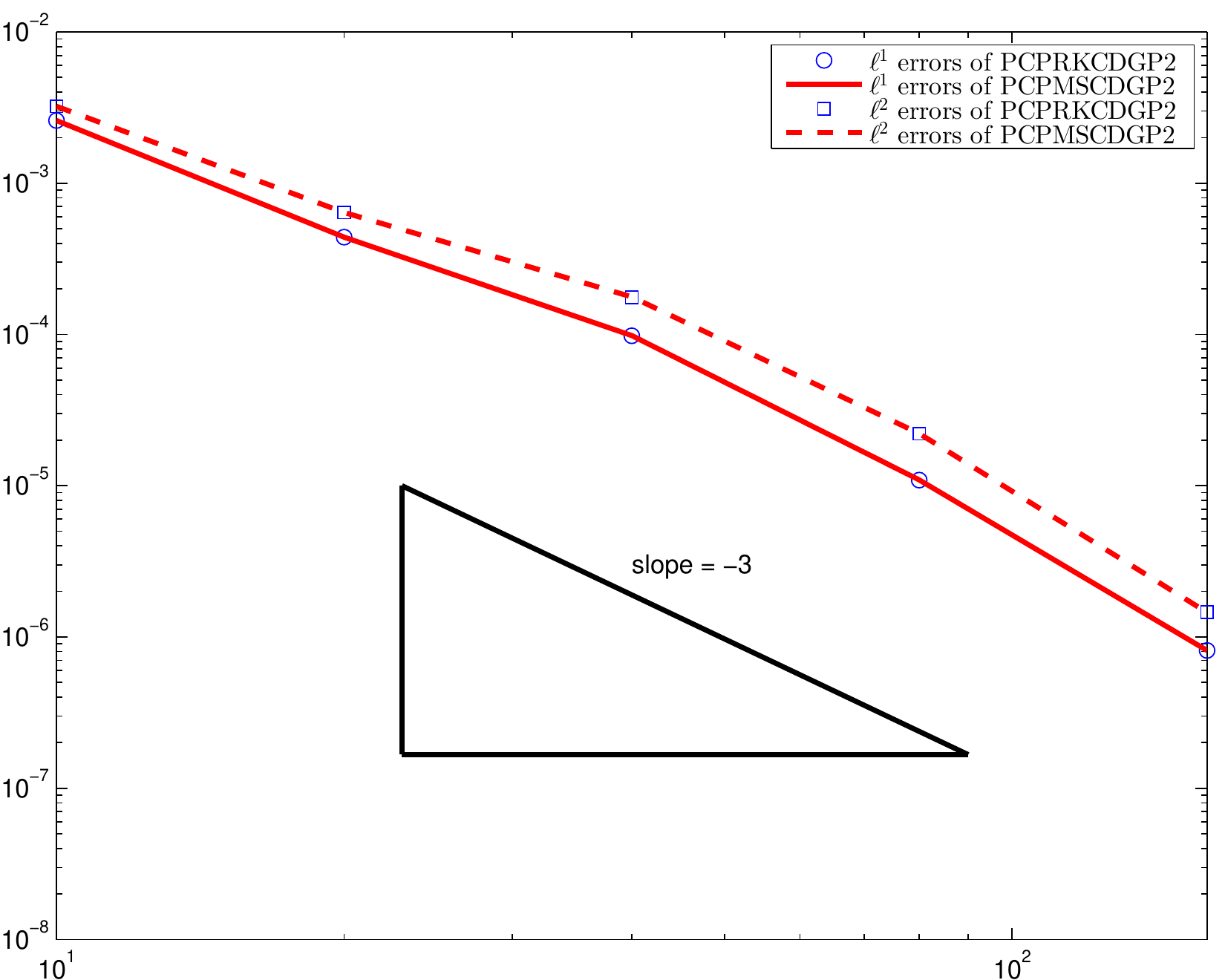}}
  \subfigure[EOS \eqref{eq:EOS:Ryu}]
  {\includegraphics[width=0.32\textwidth]{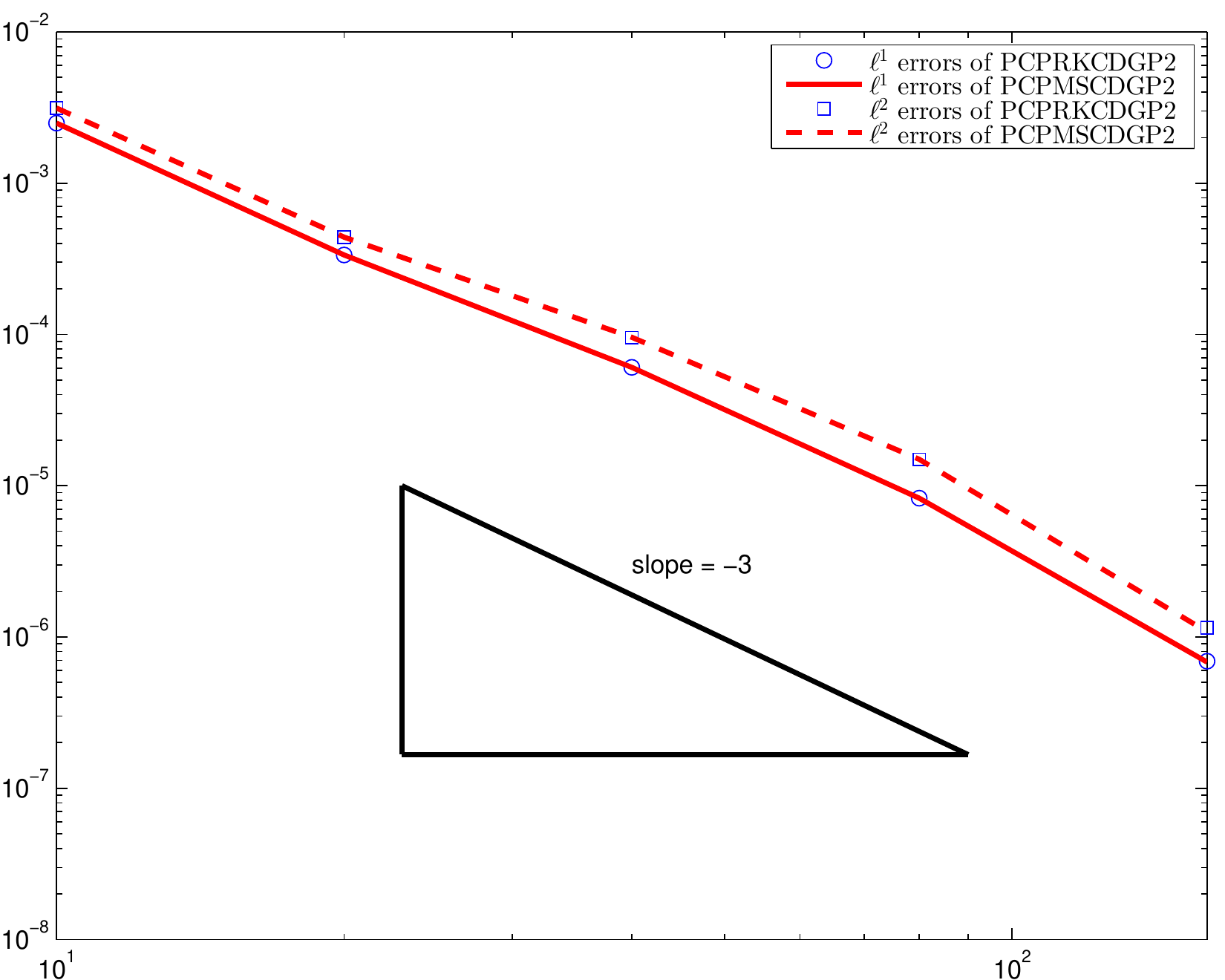}}
  \caption{\small Example \ref{example2Dsmooth}: Numerical  $l^1$- and $l^2$-errors  at $t=0.2$
  	of  {\tt PCPRKCDGP2} and {\tt PCPMSCDGP2}.
 }
  \label{fig:2Dsmooth}
\end{figure}

\end{example}

%

\begin{example}[2D Riemann problems] \label{example2DRPs}\rm
Initial data of two Riemann problems of 2D RHD equations  \eqref{eqn:coneqn3d} considered here comprise four different constant states in the unit square $\Omega=[-1,1]\times[-1,1]$, while initial discontinuities parallel to both coordinate axes respectively.   In our computations,
the uniform mesh of $400 \times 400$ cells is used,
the output time is set as $0.8$, and
$\Gamma=\frac53$ in  the ideal EOS.
Moreover, it is also necessary for  the successful  performance of  the high-order accurate central DG methods
to use the PCP limiting procedure.

The initial data of the first Riemann problem \citep{ZannaBucciantini:2002,Lucas-Serrano2004} are
$$\vec V(0,x,y)=
\begin{cases}(0.1,0,0,0.01)^{\rm T},& x>0,y>0,\\
  (0.1,0.99,0,1)^{\rm T},&    x<0,y>0,\\
  (0.5,0,0,1)^{\rm T},&      x<0,y<0,\\
  (0.1,0,0.99,1)^{\rm T},&    x>0,y<0,
  \end{cases}$$
where both the left and lower discontinuities are the contact waves with a jump in the transverse velocity
and rest-mass density, while both the right and upper  are non-simple waves.

Fig.~\ref{fig:2DRP1} gives the contours of the density logarithm $\ln \rho$  obtained by using
{\tt PCPMSCDGP2} for the ideal EOS \eqref{eq:EOSideal} and the EOS \eqref{eq:EOS:Sokolov}.
The results obtained by {\tt PCPRKCDGP2} are omitted here and hereafter because they very similar to {\tt PCPMSCDGP2}.
It is found that four initial discontinuities interact each other
and form two reflected curved shock waves,
an elongated jet-like spike approximately between two points (0.4,0.4) and (0.8,0.8) on the line
$x=y$ when $t = 0.8$,
and a complex mushroom structure starting from the point (0,0)
and  expanding to the bottom-left region;
{\tt PCPMSCDGP2} exhibits good robustness and well captures those complex
wave configurations;
 the results  for the ideal EOS case agrees well with those given
 by the high-order accurate PCP finite difference WENO
in \citep{WuTang2015};
the wave configurations depend on the EOS;
and  the velocities of the reflected curved shock waves in the case of EOS \eqref{eq:EOS:Sokolov}
are smaller than the ideal EOS.
It is worth mentioning that the high-order accurate
central DG methods  fail in the first time step
if the PCP limiting procedure is not employed.

\begin{figure}[htbp]
  \centering
  {\includegraphics[width=0.48\textwidth]{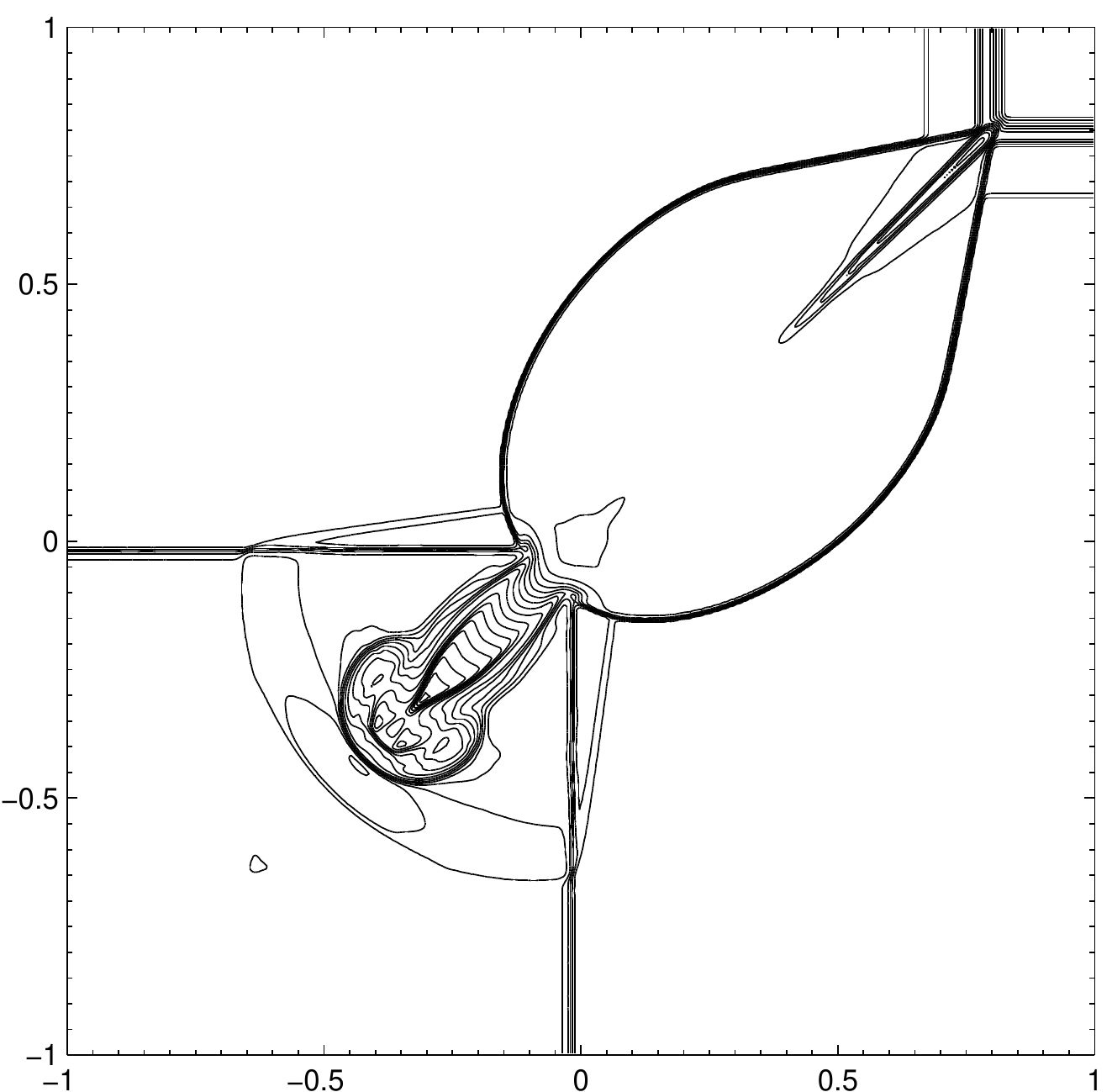}}
  {\includegraphics[width=0.48\textwidth]{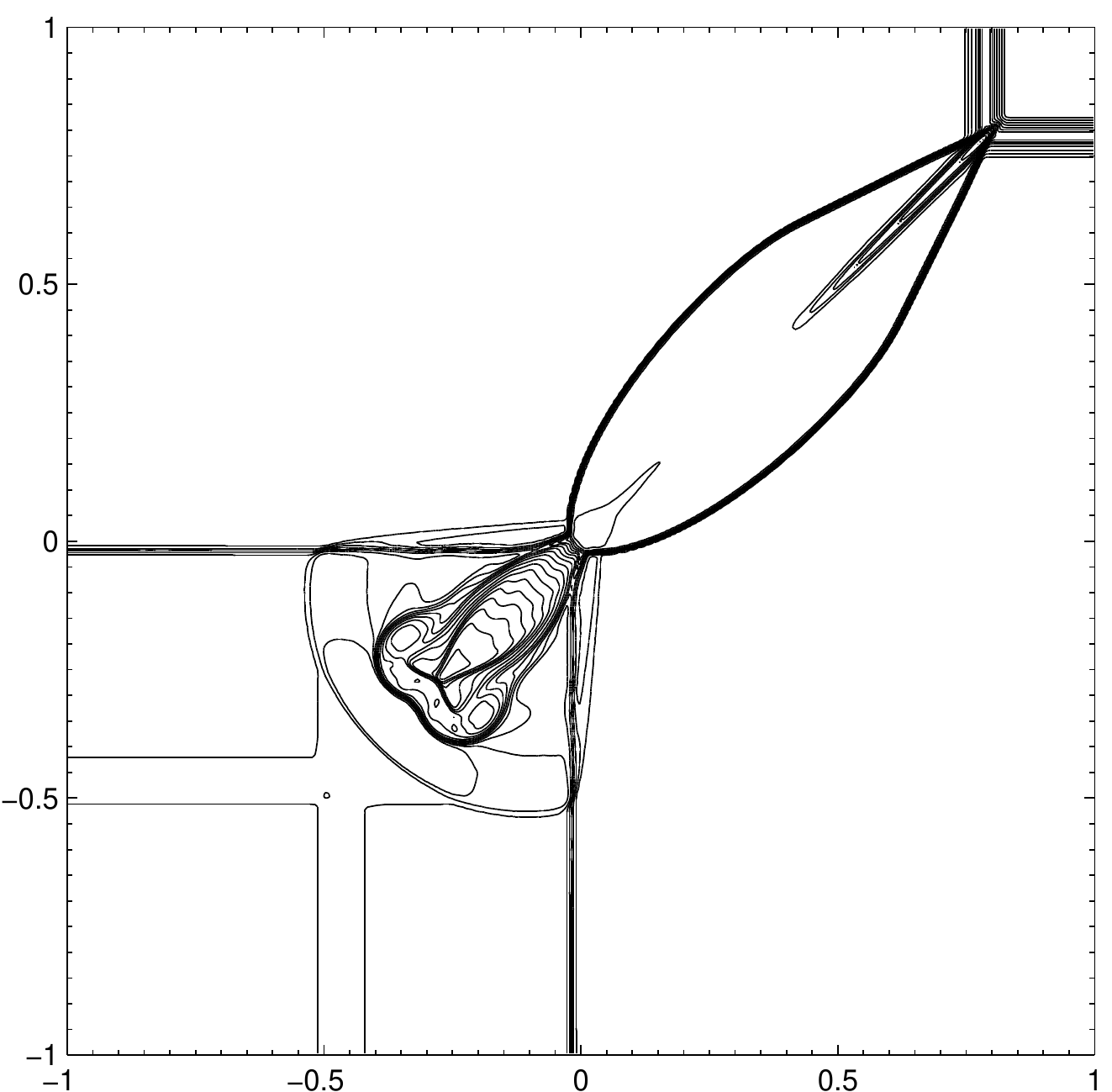}}
  \caption{\small The first 2D Riemann problem in Example \ref{example2DRPs}:
  The contours of  density
       logarithm $\ln \rho$ at $t=0.8$  obtained by using {\tt PCPMSCDGP2}.
         25 equally spaced contour lines are used.
         Left: ideal EOS \eqref{eq:EOSideal} with $\Gamma=5/3$; right: EOS \eqref{eq:EOS:Sokolov}.
 }
  \label{fig:2DRP1}
\end{figure}

The initial data of the second 2D Riemann problem \citep{WuTang2015} are
$$\vec V(0,x,y)=
\begin{cases}(0.1,0,0,20)^T,& x>0.5,y>0.5,\\
  (0.00414329639576,0.9946418833556542,0,0.05)^{\rm T},&    x<0.5,y>0.5,\\
  (0.01,0,0,0.05)^T,&      x<0.5,y<0.5,\\
  (0.00414329639576,0,0.9946418833556542, 0.05)^{\rm T},&    x>0.5,y<0.5,
  \end{cases}$$
  in which the left and lower  initial discontinuities are the contact discontinuities,
while the upper and right are the shock waves with the speed of $-0.66525606186639$ only
for  the ideal EOS.
In this test,   the EOS \eqref{eq:EOS:Mathews}  will also be considered
and the maximal value of the fluid velocity  becomes very close to the speed of light
 as the time increases.

\begin{figure}[htbp]
  \centering
  {\includegraphics[width=0.48\textwidth]{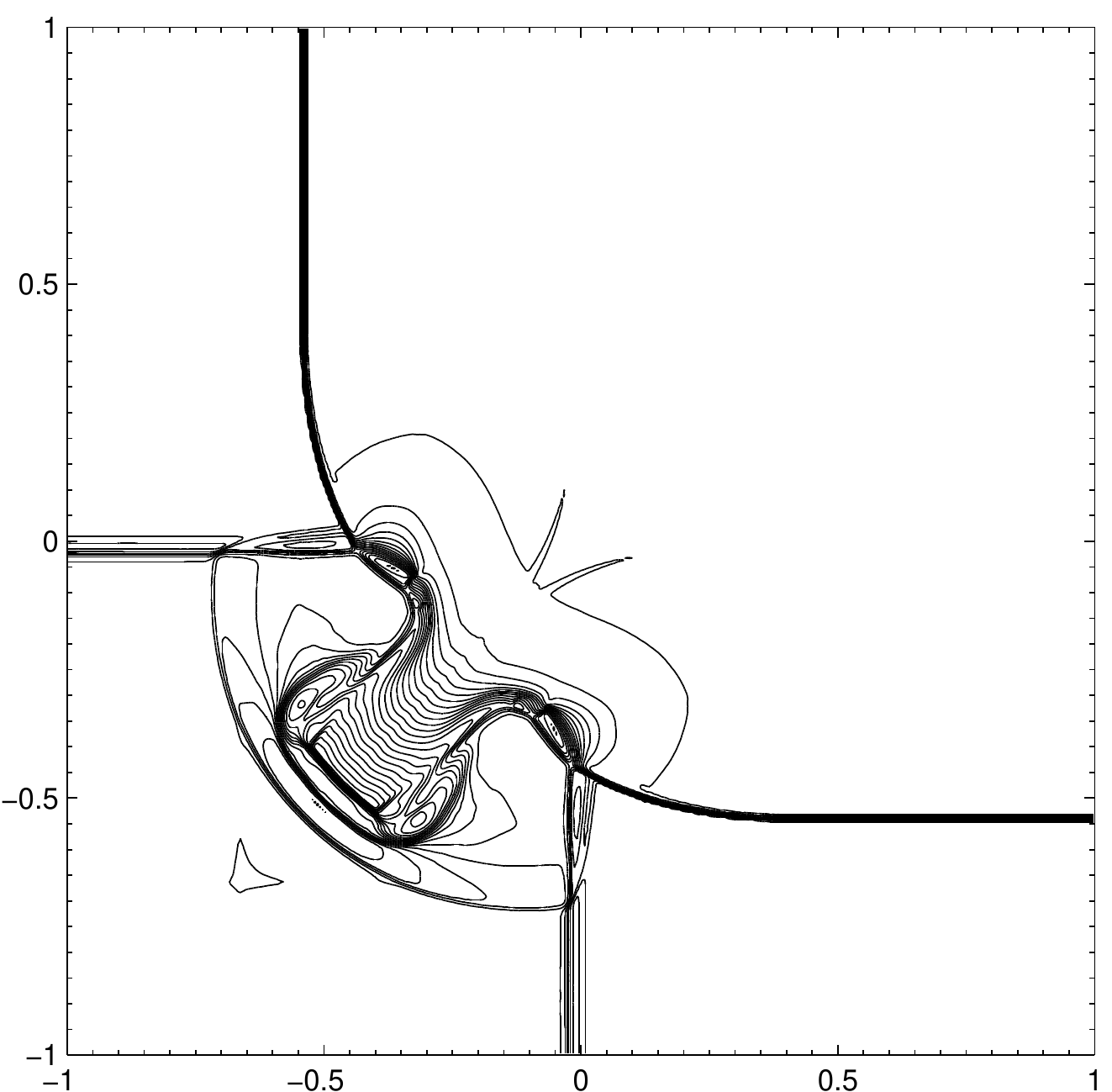}}
  {\includegraphics[width=0.48\textwidth]{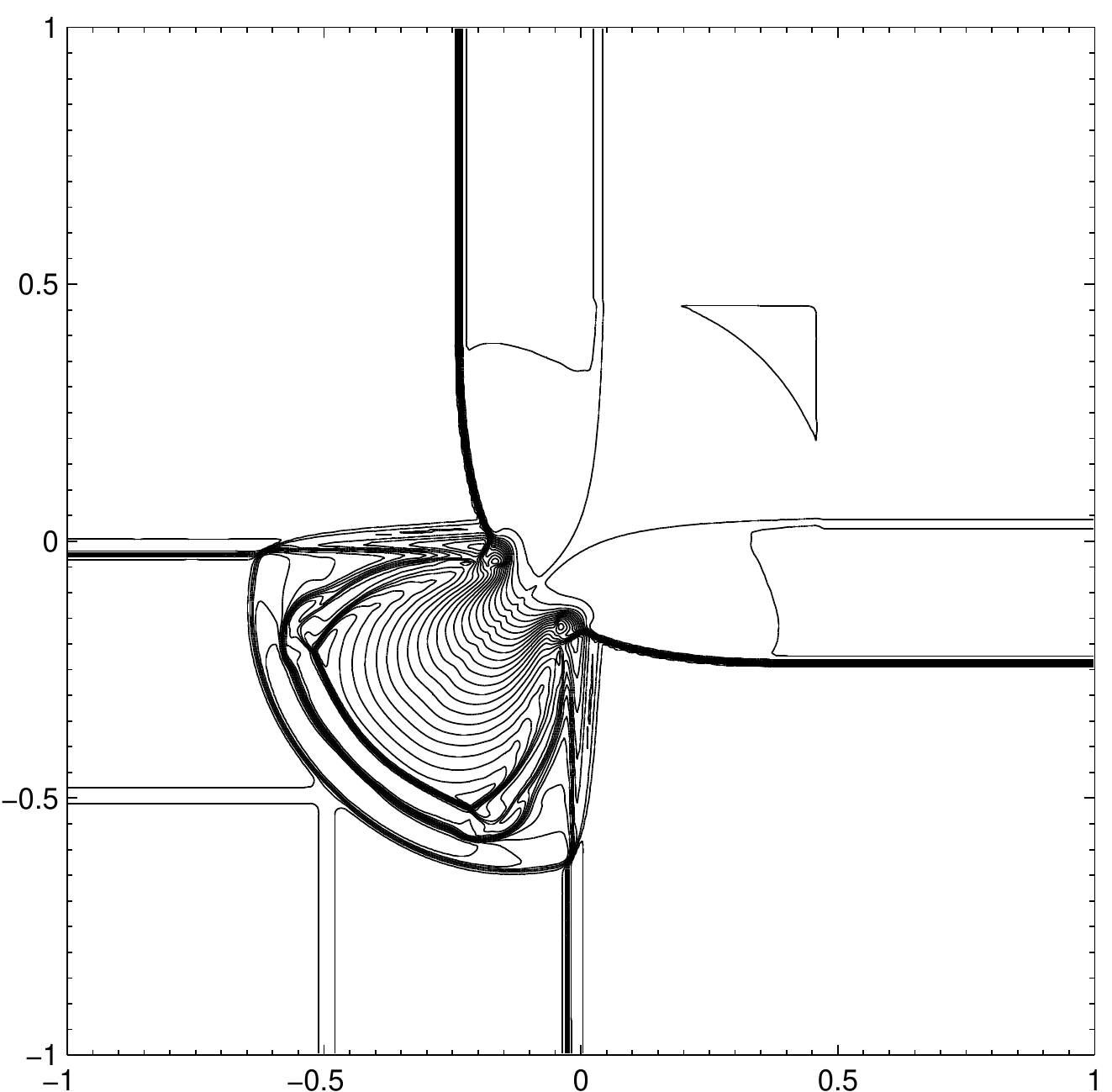}}
  \caption{\small The second 2D Riemann problem in Example \ref{example2DRPs}:
  The contours of  density
       logarithm $\ln \rho$ at $t=0.8$  obtained by using {\tt PCPMSCDGP2}
         25 equally spaced contour lines are used.
         Left: ideal EOS \eqref{eq:EOSideal} with $\Gamma=5/3$;
         right: EOS \eqref{eq:EOS:Mathews}.
 }
  \label{fig:2DRP2}
\end{figure}

Fig.~\ref{fig:2DRP2} displays the contours of the density logarithm $\ln \rho$
 obtained by using {\tt PCPMSCDGP2}.
It is seen that the
interaction of four initial discontinuities leads to the distortion of
the initial shock waves and the formation
of a ``mushroom cloud'' starting from the point $(0,0)$ and expanding to the left bottom region.
The present methods have good performance and  robustness
in simulating such ultra-relativistic flow.
The flow structures of ``mushroom cloud'' for the ideal EOS \eqref{eq:EOSideal}
and EOS \eqref{eq:EOS:Mathews} are obviously different,
and the former agrees well with that given in \citep{WuTang2015}
by high-order accurate PCP finite difference WENO schemes.

\end{example}

\begin{example}[Relativistic jets] \label{exampleJet}\rm
The last 2D example is to simulate two high-speed relativistic jet flows.
The jet flows with  high speed
are ubiquitous in the  extragalactic radio sources associated with the active galactic nuclei
and the most compelling case for a special relativistic phenomenon.
It is very challenging  to simulate such jet flows
since there may appear  the strong relativistic shock wave,  shear wave,  interface instabilities,
and the ultra-relativistic region, etc. besides the high-speed jet,
 see e.g. \citep{Marti1994,Duncan:1994,Marti1997,Komissarov1998,ZhangMacfadyen2006}.

\begin{figure}[htbp]
  \centering
  {\includegraphics[width=0.3\textwidth]{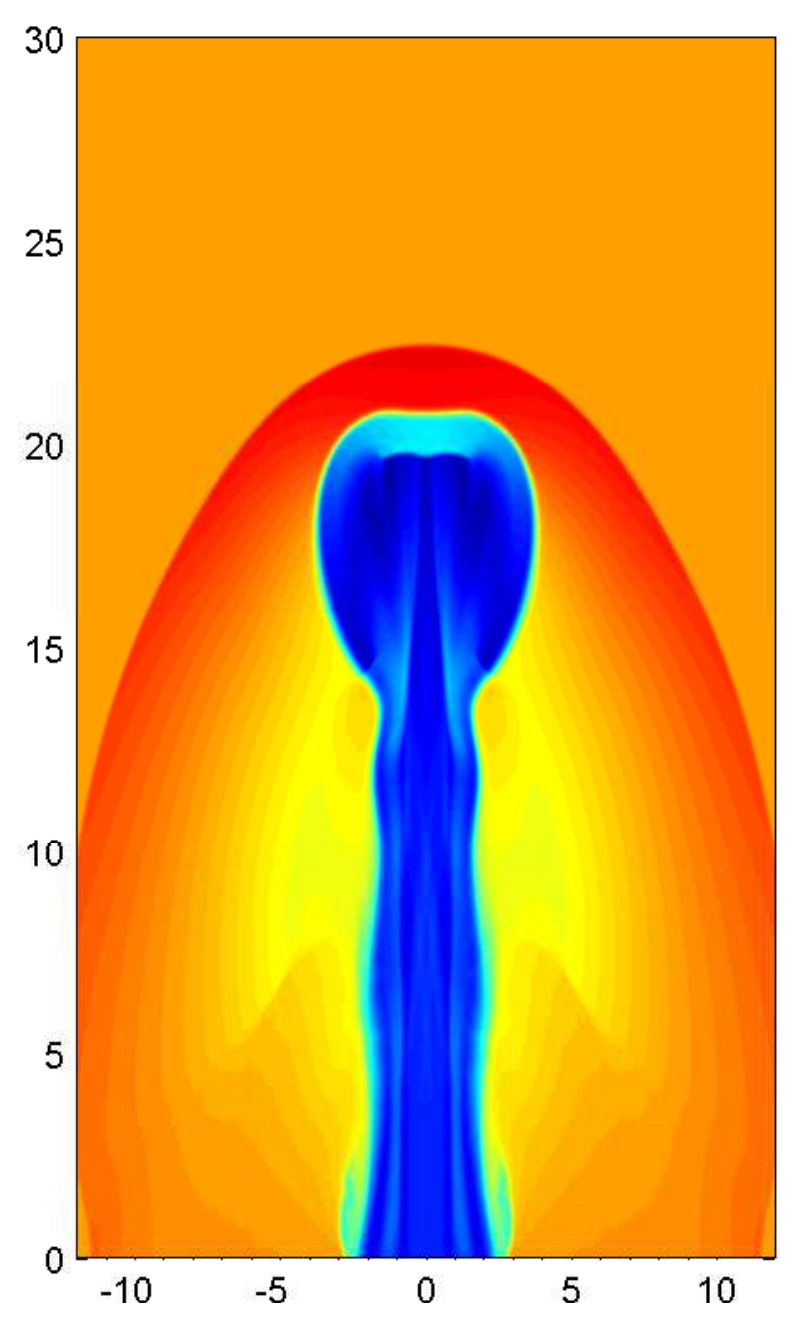}}
  {\includegraphics[width=0.3\textwidth]{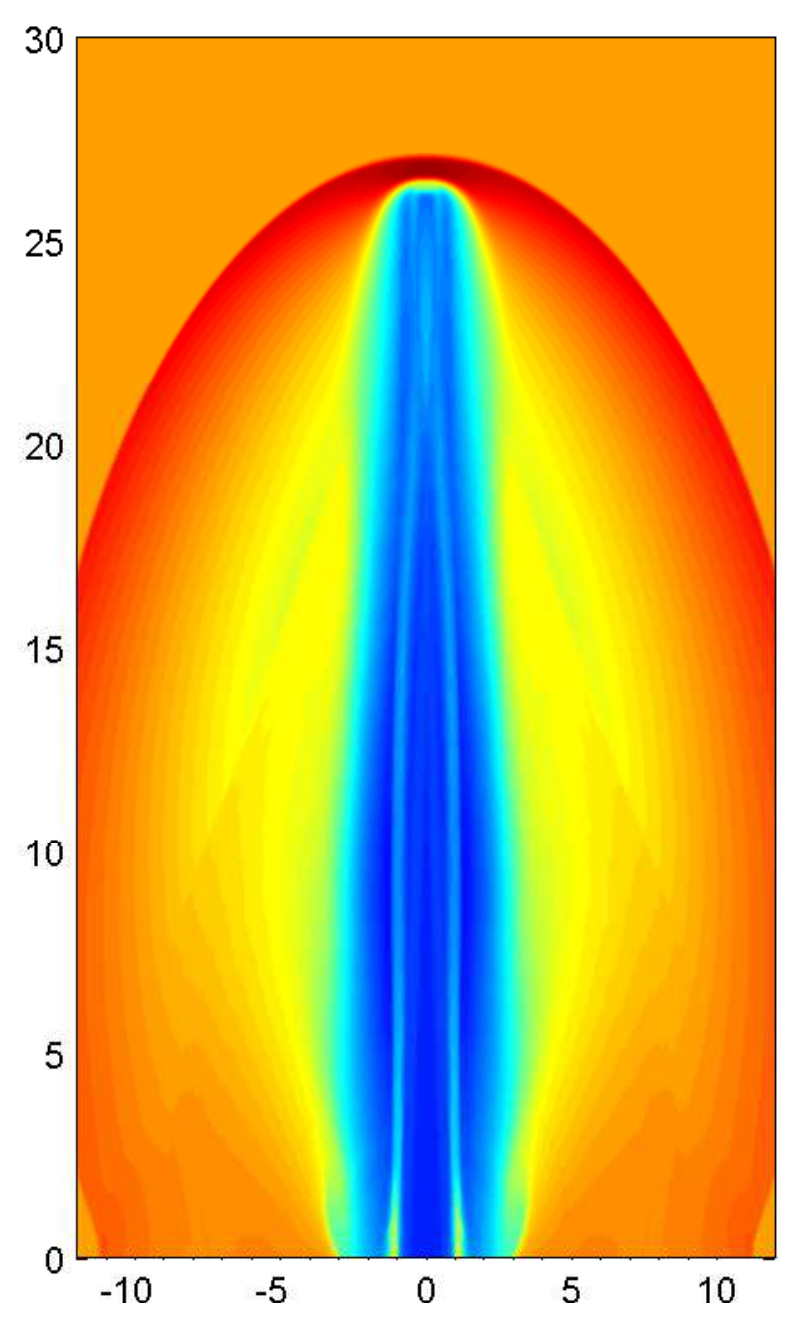}}
  {\includegraphics[width=0.3\textwidth]{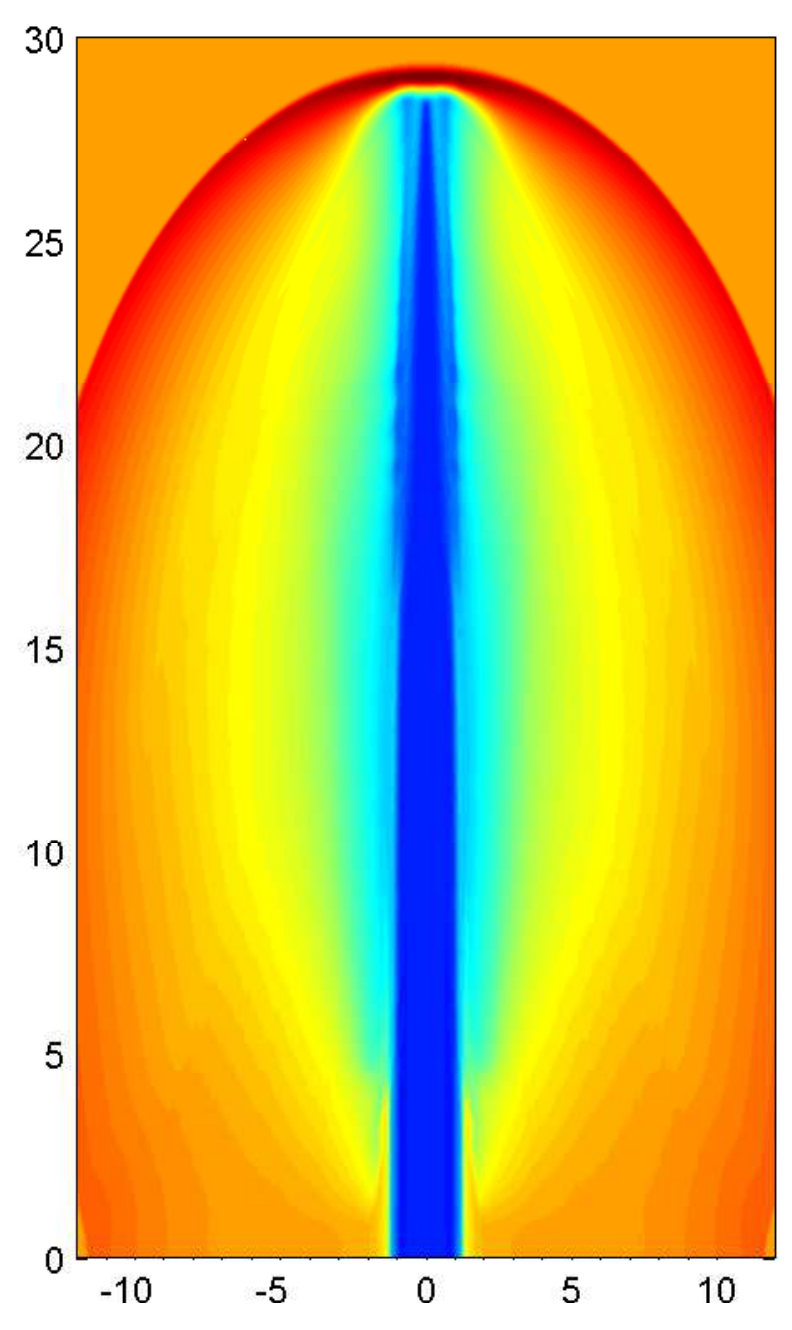}}
  \caption{\small Hot jet models in Example \ref{exampleJet}:
  Schlieren images of the rest-mass density logarithm
  $\ln \rho$ at $t=30$ obtained by {\tt PCPMSCDGP2}
        on the mesh of $240 \times 600$ uniform cells.
     From left to right: configurations (i), (ii), and (iii).
 }
  \label{fig:jetA1}
\end{figure}

\begin{figure}[htbp]
  \centering
  {\includegraphics[width=0.3\textwidth]{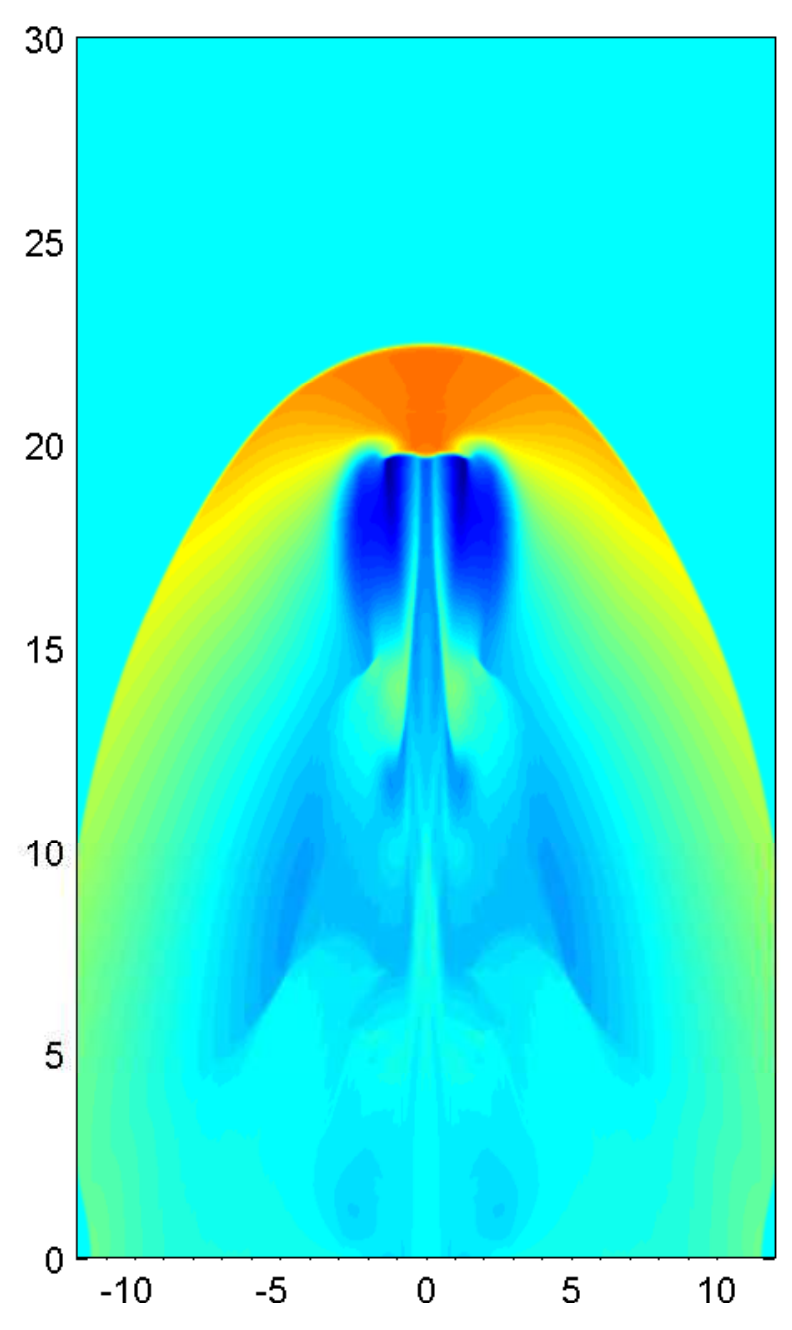}}
  {\includegraphics[width=0.3\textwidth]{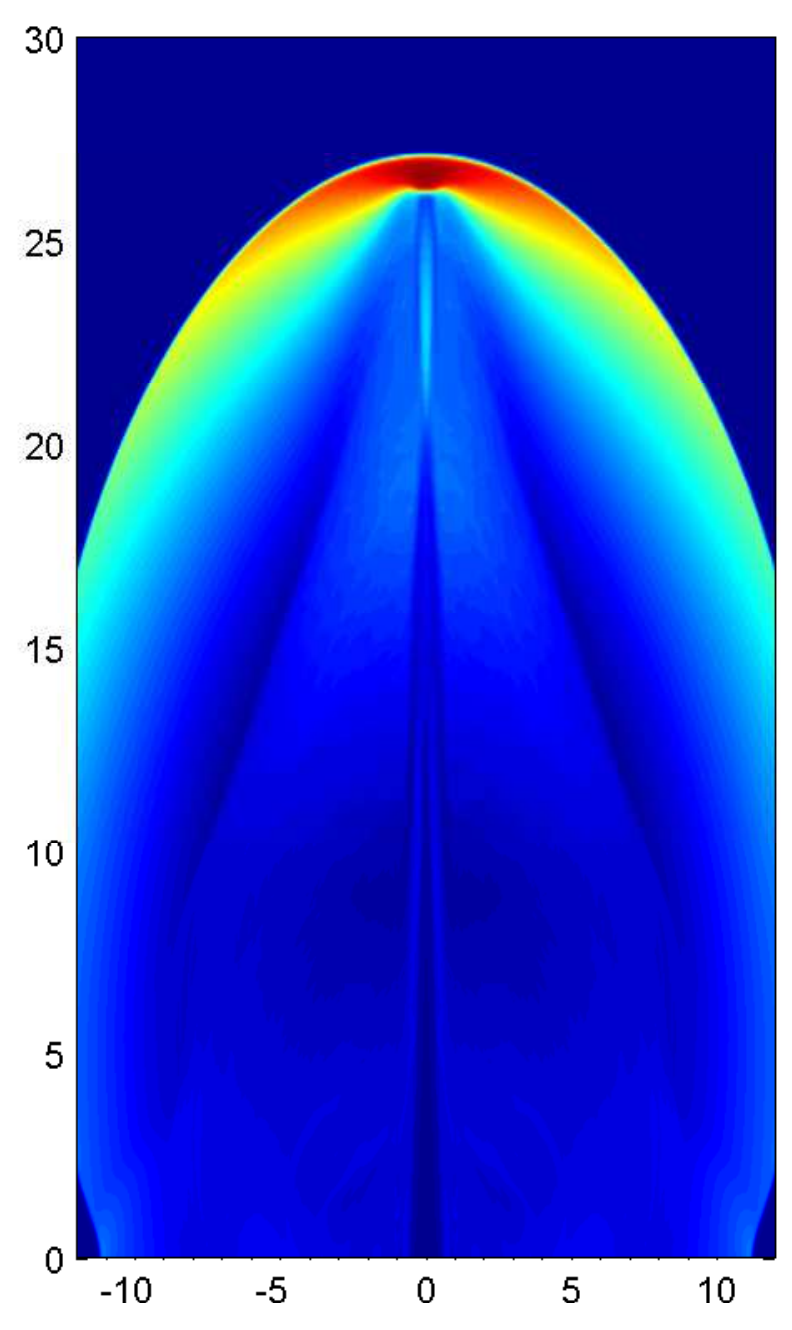}}
  {\includegraphics[width=0.3\textwidth]{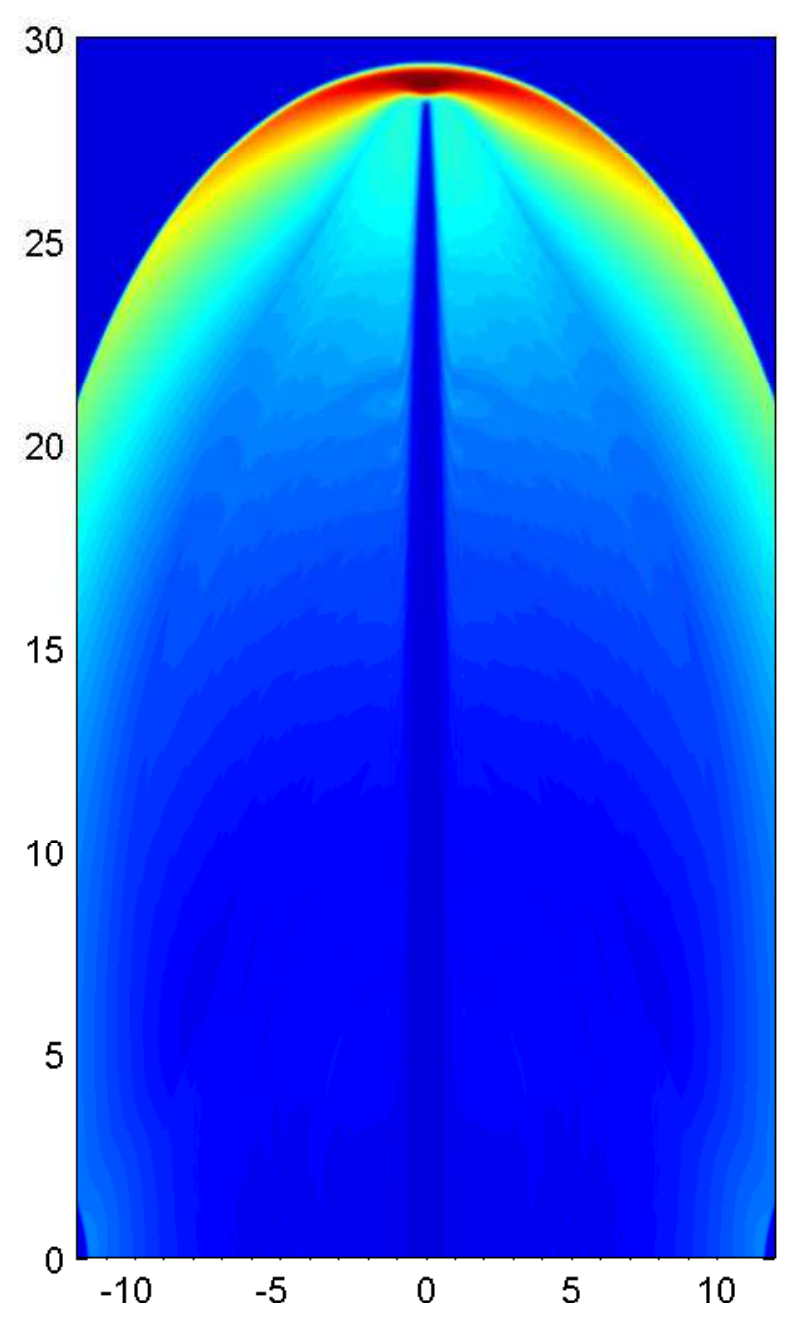}}
  \caption{\small    Same as Fig. \ref{fig:jetA1} except for the schlieren images of
  	 pressure logarithm $\ln p$.
 }
  \label{fig:jetA1p}
\end{figure}

The first test is a pressure-matched hot jet model,
 in which the beam is moving at a speed $v_b$,
 the classical beam Mach number $M_b$ is near the minimum Mach number
 for given $v_b$, and the relativistic effects from large beam internal energies
are important and comparable to the effects from the fluid velocity near the speed of light.
Initially, the computational domain $[0,12]\times [0,30]$
is filled with a static uniform medium with
an unit rest-mass density.
A light relativistic jet is
injected in the $y$--direction through
the inlet part ($|x| \le 0.5$)
on the bottom boundary ($y=0$) with a density of 0.01,
a pressure equal to the ambient pressure, and a speed of $v_b$.
The reflecting boundary condition is specified at $x=0$,
 the fixed inflow beam condition is specified on the nozzle $\{y=0, |x|\le 0.5$\},
while the outflow boundary conditions are on other boundaries.
The EOS is taken as \eqref{eq:EOS:Ryu} and three different configurations are considered as follows:
\\
(i) $v_b=0.99$ and $M_b=1.72$,   corresponding to the case of  Lorentz factor  $W\approx 7.09$ and
relativistic Mach number $M_r :=M_b W /W_s\approx 9.97$,
where   $W_s=1/\sqrt{1-c_s^2}$  is the Lorentz factor associated with the local sound speed;
\\
(ii) $v_b=0.999$ and $M_b=1.74$,  corresponding to the case of  $W\approx 22.37$
and  $M_r\approx  38.88$;
\\
(iii) $v_b=0.9999$ and $M_b=1.74$,  corresponding to the case of   $W\approx   70.71$
and  $M_r\approx  123.03$.
As $v_b$ becomes more close to the speed of light, the simulation of the jet becomes more challenging.

Figs. \ref{fig:jetA1} and \ref{fig:jetA1p}  display respectively
the schlieren images of  rest-mass density logarithm $\ln \rho$ and  pressure logarithm $\ln p$
within the domain $[-12,12]\times [0,30]$ at $t=30$
obtained by using {\tt PCPMSCDGP2}       on the uniform mesh of $240 \times 600$ cells in
the computational domain $[0,12]\times [0,30]$.
It is seen that the  Mach shock wave at the jet head and  the beam/cocoon interface
are well captured during the whole simulation and
the proposed PCP methods exhibit good performance and robustness.

\begin{figure}[htbp]
  \centering
  {\includegraphics[width=0.3\textwidth]{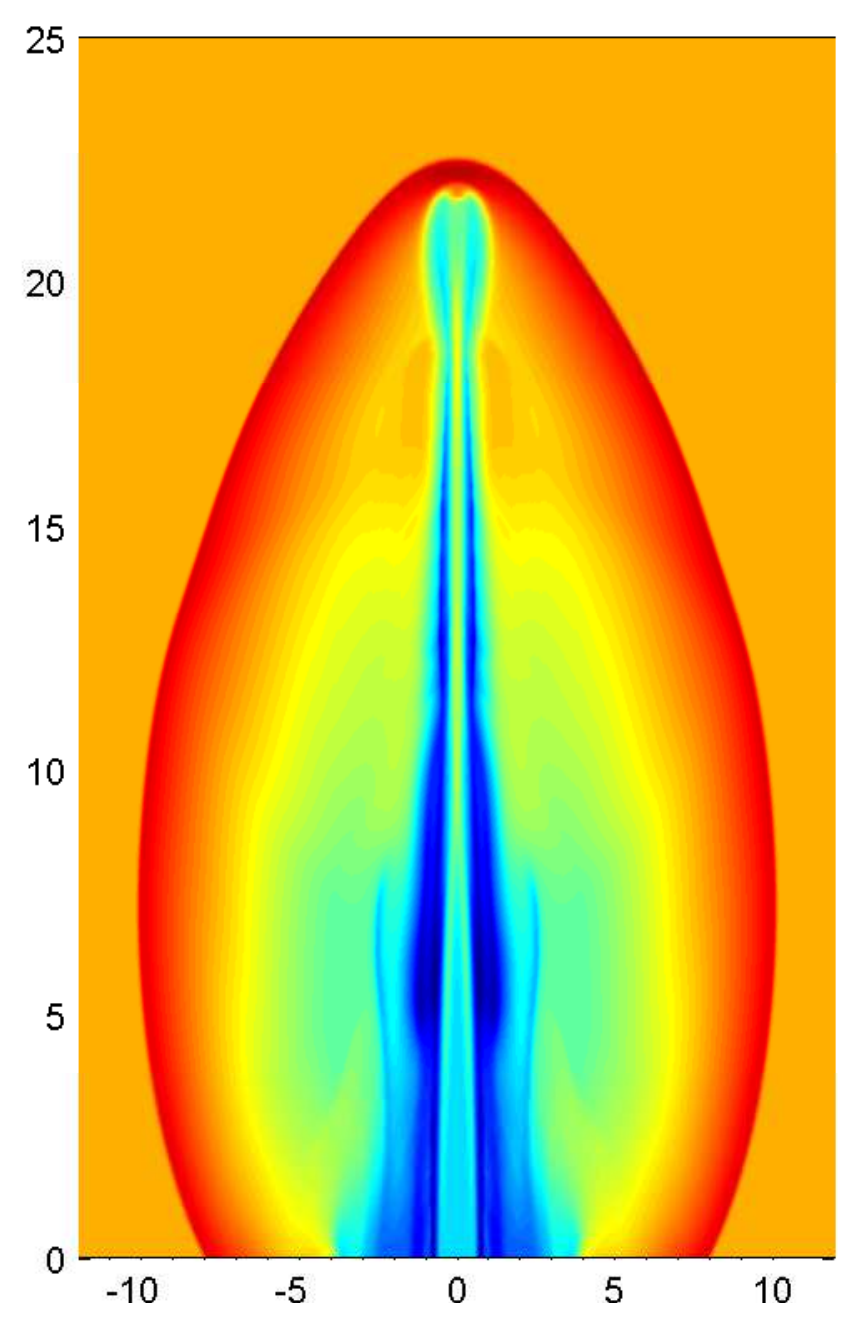}}
  {\includegraphics[width=0.3\textwidth]{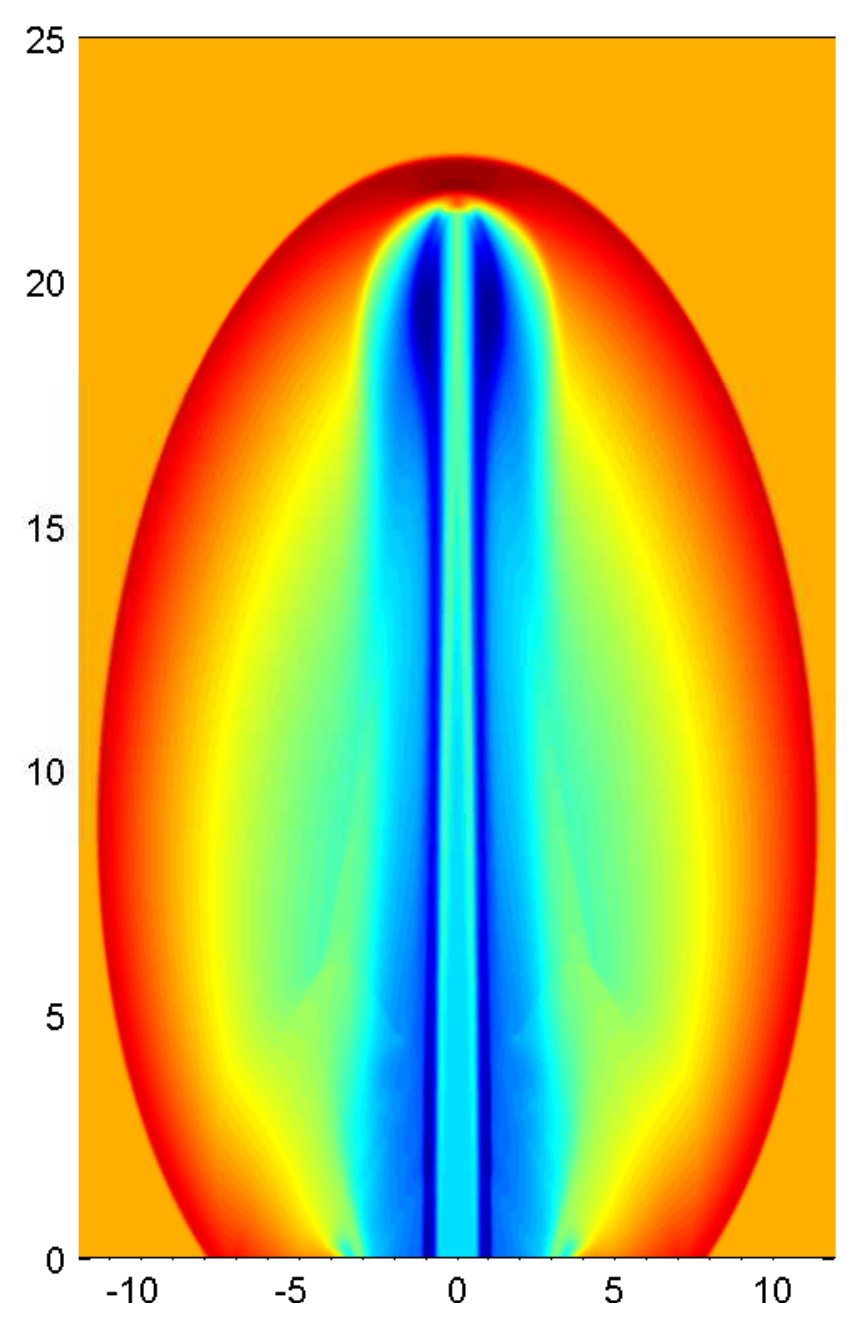}}
  {\includegraphics[width=0.3\textwidth]{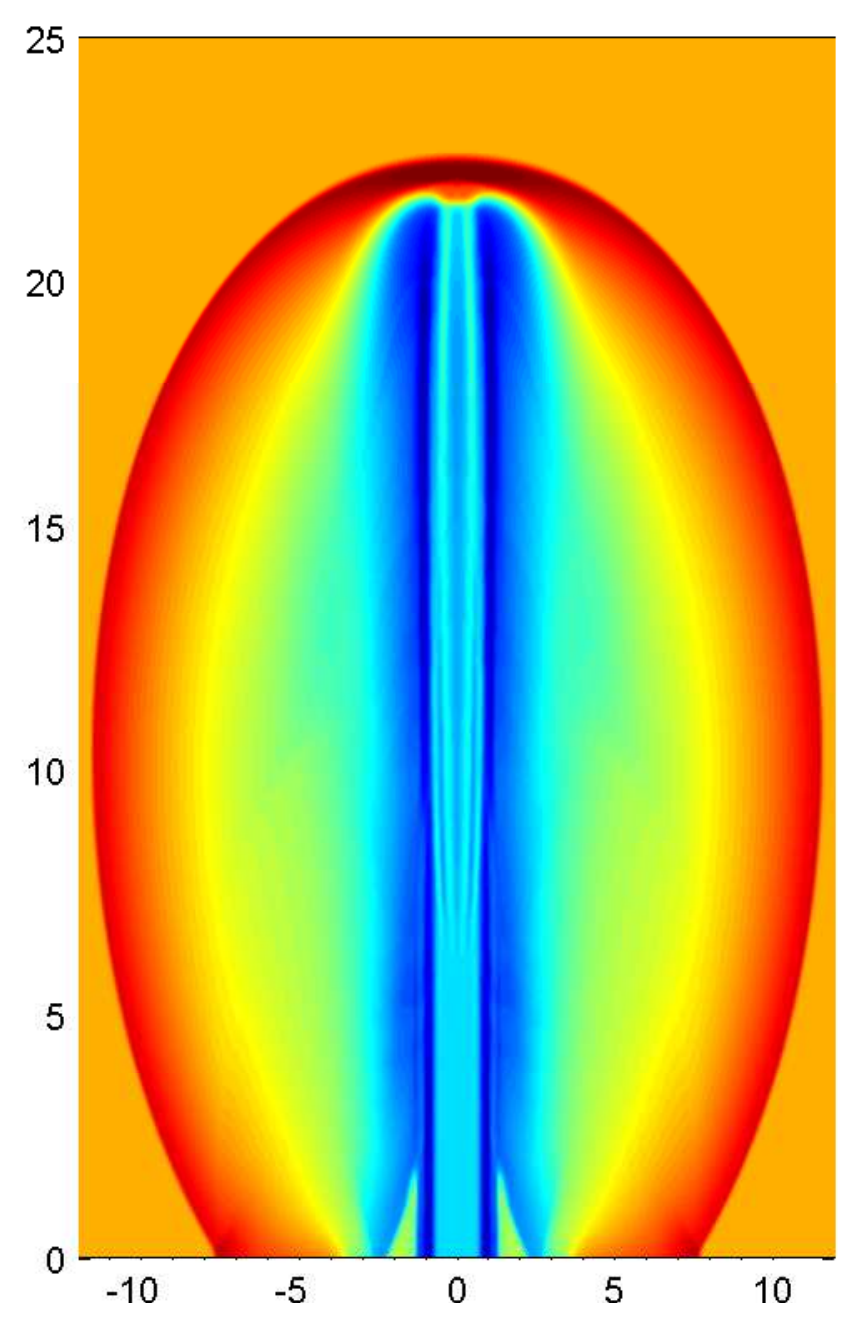}}
  \caption{\small Cold jet models in Example \ref{exampleJet}:
  Schlieren images of the rest-mass density logarithm $\ln \rho$
  obtained by {\tt PCPMSCDGP2}   on the uniform mesh of $240 \times 500$  cells.
        From left to right: configurations (i) at $t=30$,  (ii)
      at $t=25$,  and (iii) at $t=23$.
 }
  \label{fig:jetC2}
\end{figure}

\begin{figure}[htbp]
  \centering
  {\includegraphics[width=0.3\textwidth]{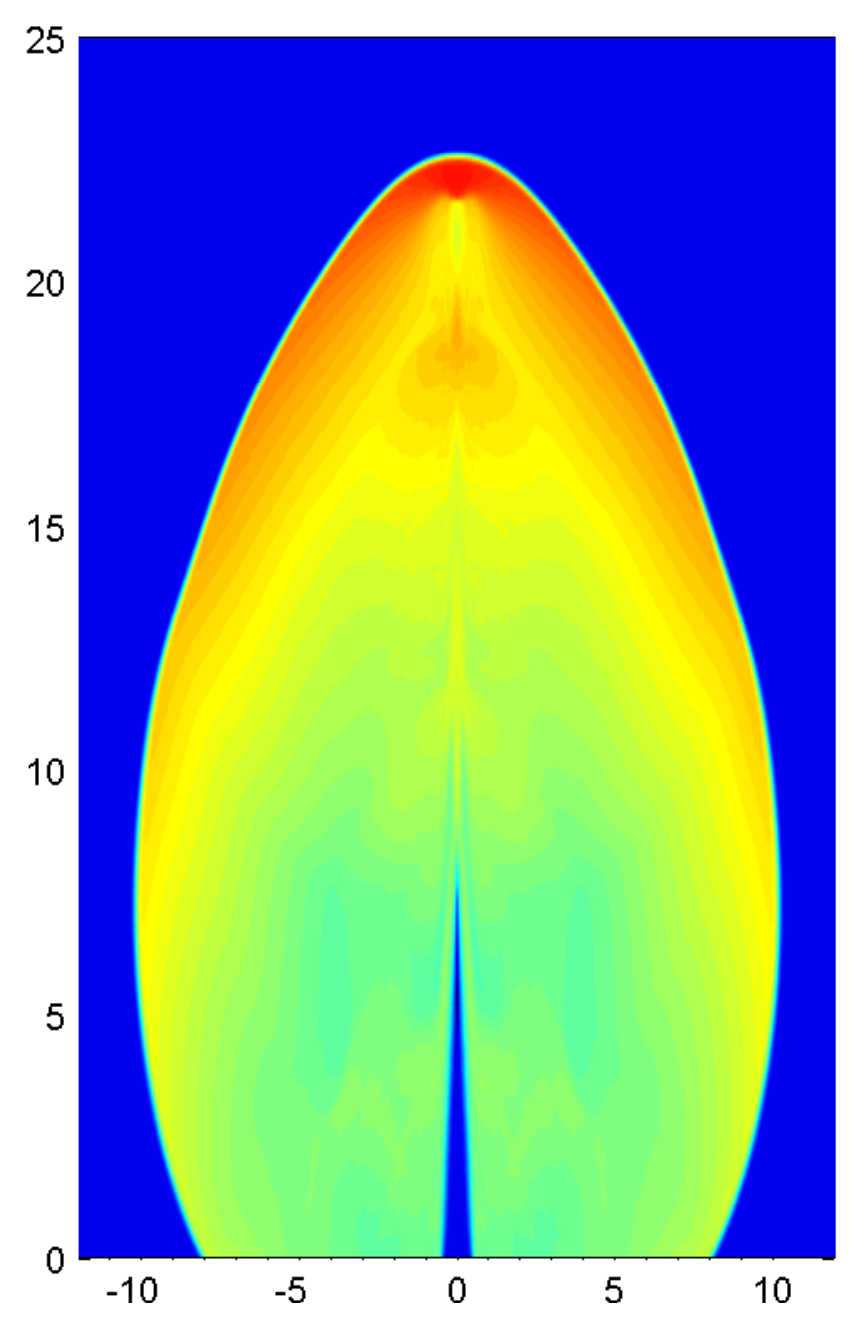}}
  {\includegraphics[width=0.3\textwidth]{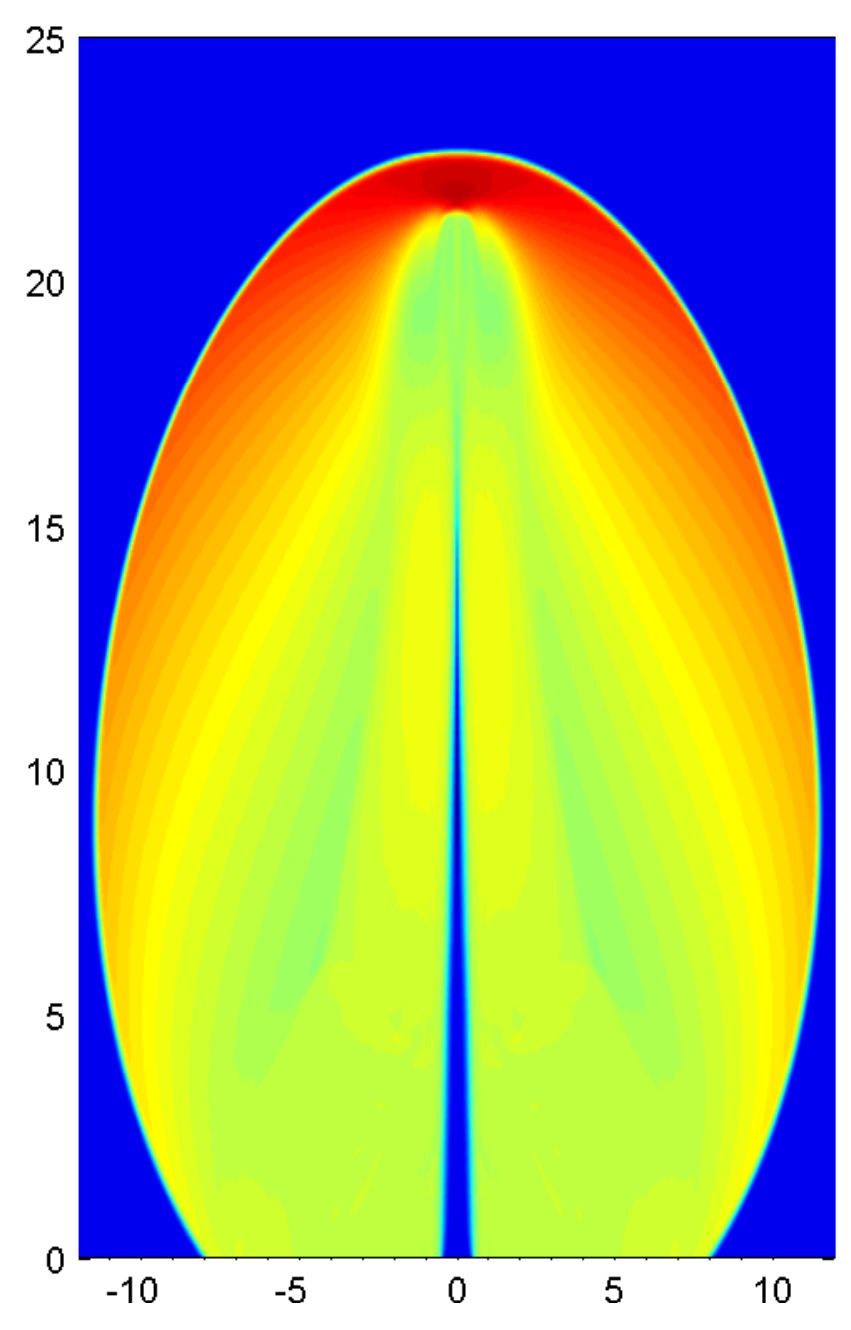}}
  {\includegraphics[width=0.3\textwidth]{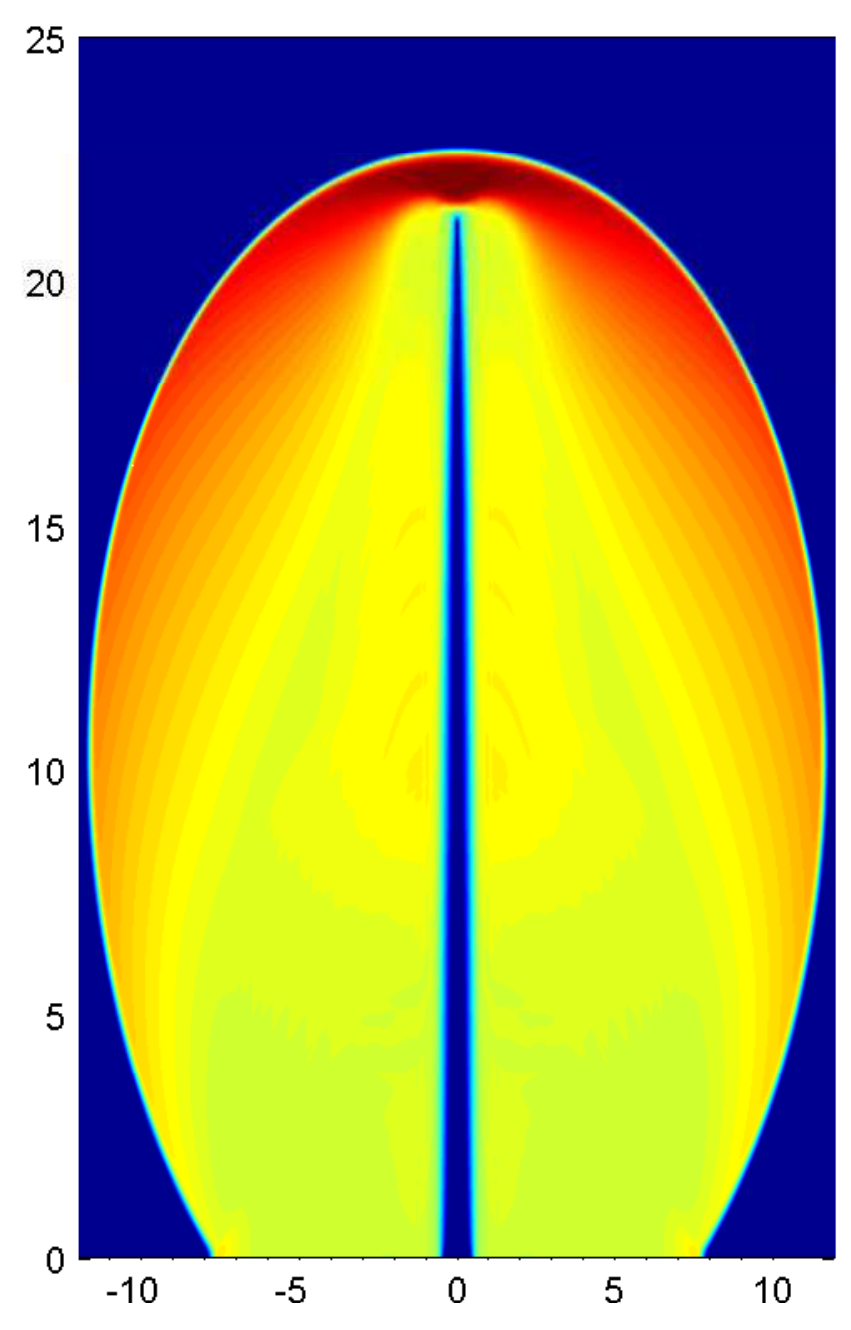}}
  \caption{\small
  Same as Fig. \ref{fig:jetC2} except for the schlieren images of  pressure logarithm $\ln p$.
 }
  \label{fig:jetC2p}
\end{figure}

The second test is the pressure-matched highly supersonic jet model.
Highly supersonic jet model  is also referred to the cold model, in which
 the relativistic effects from the large beam speed dominate
so that there exists an important difference between the hot and cold relativistic jets.
The setups are the same as the above hot jet model except for that the density of
 inlet jet becomes 0.1,  the EOS is taken as the ideal EOS with $\Gamma = \frac 5 3$,
 and the computational domain is $[0,12]\times [0,25]$.
 Three different configurations are  considered as follows:
 \\
 (i) $v_b=0.99$ and $M_b=50$,   corresponding to the case of  Lorentz factor  $W\approx 7.09$ and
 relativistic Mach number $M_r \approx 354.37$;
 \\
 (ii) $v_b=0.999$ and $M_b=50$,  corresponding to the case of  $W\approx 22.37$
 and  $M_r\approx  1118.09$;
 \\
 (iii) $v_b=0.9999$ and $M_b=500$,  corresponding to the case of   $W\approx   70.71$
 and  $M_r\approx  35356.15$.

Figs. \ref{fig:jetC2} and \ref{fig:jetC2p} display respectively  the schlieren images of
 rest-mass density logarithm $\ln \rho$ and  pressure logarithm $\ln p$
 within the domain $[-12,12]\times [0,25]$ obtained by using {\tt PCPMSCDGP2}
 on the uniform mesh of $240 \times 500$ cells in the computational domain $[0,12]\times [0,25]$.
 It is observed that the flow structures are different from
 those of the hot jet model, and the bow shock expends wider for larger beam velocity
 and  our PCP central DG methods exhibit very strong robustness during the whole simulations.
\end{example}

\section{Conclusions}\label{sec:conclude}

The paper developed high-order accurate physical-constraints-preserving (PCP)
central discontinuous Galerkin (DG) methods for the 1D and 2D special relativistic
hydrodynamic (RHD) equations with a general equation of state (EOS).
The main contribution was  proving several key properties of the admissible state set,
including the convexity, scaling and orthogonal invariance, and Lax-Friedrichs splitting property.
It was done with the aid of the equivalent form of the admissible state set
and nontrivial due to the inherent nonlinearity of the RHD equations and
no explicit expressions  of the primitive variables and the
flux vectors with respect to the conservative vector.
Built on the analysis of the admissible state set,
the PCP limiting procedure was designed to enforce the admissibility
of the central DG solutions.
The fully-discrete high-order PCP central DG methods
with the PCP limiting procedure and strong stability preserving time discretization
were proved to
preserve positivity of the density, pressure and specific internal
energy and the bound of the fluid velocity under a CFL type condition,
maintain high-order accuracy, and be $L^1$-stable.
Several 1D and 2D numerical examples were used to demonstrate
 the accuracy, robustness and effectiveness of the proposed PCP methods in solving
several 1D and 2D relativistic fluid flow problems with large Lorentz factor, strong discontinuities, or low
rest-mass density or pressure, etc.
The present PCP limiting procedure and analyses
could be used to develop  high-order accurate PCP finite volume or finite difference schemes
for the RHD equations with a general EOS.

\section*{Acknowledgements}

This work was partially supported by
the National Natural Science Foundation
of China (Nos.  91330205 \& 11421101).

\begin{appendices}

\section{Derivation of (2.6) by the kinetic theory}
\label{Appendix-A}

 Only  the   case of $d=3$ is discussed here.  
 According to the kinetic theory \citep{Cercignani:book,Rezolla-ZanottiBook}, %
 one has
$$
D = \hat m\int_{{\mathbb{R}}^3 } { \hat f {\rm d}\vec {\hat p}},
\  m_i  = \int_{{\mathbb{R}}^3 } { \hat p^i \hat f{\rm d}\vec {\hat p}},\  E = \int_{{\mathbb{R}}^3 } {\hat p^0 \hat f {\rm d}\vec {\hat p}}, \ \ ~i=1,2,3,
$$
where $\hat m$ is the rest mass of the gas particle
and $\hat f(t,\vec x, \vec {\hat p})\in L^2([0,+\infty)\times \mathbb{R}^6)$
is nonnegative and denotes the equilibrium distribution function depending on the space-time and the particle
momentum coordinates $(\hat p^0,\vec{\hat p} )$ with $\hat p^0 = \sqrt{ |\vec {\hat p}|^2 + \hat m^2 }$. It follows that
\begin{align}\nonumber
D^2 + |\vec m|^2 - E^2 & =
\left( \hat m\int_{{\mathbb{R}}^3 } { \hat f {\rm d}\vec {\hat p}} \right )^2 +
\sum\limits_{i = 1}^3  \left ( \int_{{\mathbb{R}}^3 } { \hat p^i \hat f{\rm d}\vec {\hat p}}  \right)^2
- \left(  \int_{{\mathbb{R}}^3 } {\hat p^0 \hat f {\rm d}\vec {\hat p}} \right)^2\\ \nonumber
&\le
\left( \hat m\int_{{\mathbb{R}}^3 } { \hat f {\rm d}\vec {\hat p}} \right )^2 +
\sum\limits_{i = 1}^3  \left ( \int_{{\mathbb{R}}^3 } { |\hat p^i| \hat f{\rm d}\vec {\hat p}}  \right)^2
- \left(  \int_{{\mathbb{R}}^3 } {\hat p^0 \hat f {\rm d}\vec {\hat p}} \right)^2 \\ \nonumber
&
=
\left( \int_{{\mathbb{R}}^3 } { \left( \hat m^2 \hat f^2 \right)^{\frac{1}{2}} {\rm d}\vec {\hat p}} \right )^2 +
\sum\limits_{i = 1}^3  \left ( \int_{{\mathbb{R}}^3 } { \left( |\hat p^i|^2 \hat f^2 \right)^{\frac{1}{2}}  {\rm d}\vec {\hat p}}  \right)^2
- \left(  \int_{{\mathbb{R}}^3 } {\hat p^0 \hat f {\rm d}\vec {\hat p}} \right)^2\\  \label{eq:gEOSproof:Ieq}
& \le
\left( \int_{{\mathbb{R}}^3 } { \left(  \hat m^2 \hat f^2 + \sum\limits_{i = 1}^3 |\hat p^i|^2 \hat f^2  \right)^{\frac{1}{2}} {\rm d}\vec {\hat p}} \right )^2
- \left(  \int_{{\mathbb{R}}^3 } {\hat p^0 \hat f {\rm d}\vec {\hat p}} \right)^2\\ \nonumber
& =
\left( \int_{{\mathbb{R}}^3 } {  \sqrt{ |\vec {\hat p}|^2 + \hat m^2 }   \hat f {\rm d}\vec {\hat p}} \right )^2
- \left(  \int_{{\mathbb{R}}^3 } {\hat p^0 \hat f {\rm d}\vec {\hat p}} \right)^2 = 0,
\end{align}
where the reverse Minkowski inequality 
$$
\sum\limits_{i=0}^3  \left( \int_{{\mathbb{R}}^3 } { \left(  g_i(\vec {\hat p}) \right)^{\frac{1}{2}} {\rm d}\vec {\hat p}} \right )^2
\le \left( \int_{{\mathbb{R}}^3 } { \left( \sum\limits_{i=0}^3  g_i(\vec {\hat p}) \right)^{\frac{1}{2}} {\rm d}\vec {\hat p}} \right )^2,
$$
 has been used, 
$g_0= \hat f^2 \hat m^2$, and $g_i=\hat f^2  |\hat p^i|^2$, $i=1,2,3$.


The equal sign in \eqref{eq:gEOSproof:Ieq} does not work, in other words,
it always holds that $D^2 + |\vec m|^2 < E^2$.
Otherwise, one has that (i) $\hat f$ is equal to zero in the $\hat{\vec p}$ space
almost everywhere for $\vec{\hat p}$,
or (ii) there exist three nonnegative real numbers $\{a_i\}_{i=1}^3$ independent on $\vec{\hat p}$ such that
\begin{align}\label{EQ:appedix-A2}
g_0-a_i g_i = 0, \  i=1,2,3.
\end{align}
  for almost $\hat{\vec p} \in {\mathbb{R}}^3$.
The case (i) conflicts with  the fact that
$$
 \int_{{\mathbb{R}}^3 } \hat f(t,\vec x, \vec {\hat p}) {\rm d}\vec {\hat p } = D/\hat m>0,
$$
while the case (ii) also implies that $\hat f$ is equal to zero almost everywhere
for $\vec{\hat p}$ such that  the same contradiction is met. In fact, if
$\hat f(t,\vec x, \hat{\vec p}) \neq 0$ for fixed $t$ and $\vec x$,
then using \eqref{EQ:appedix-A2} gives $|\hat p^i| = \frac{\hat m}{a_i}$, where
 $a_i \neq 0$ since $g_0 >0$ and \eqref{EQ:appedix-A2}.
It implies that for fixed $t$ and $\vec x$, $\hat f(t,\vec x, \hat{\vec p}) \neq 0$ only when
 $\vec {\hat p}=(\pm {\hat m}/{a_1},\pm {\hat m}/{a_2},\pm {\hat m}/{a_3})$,
thus $\hat f(t,\vec x, \hat{\vec p})$ is equal to zero  in the $\hat{\vec p}$ space almost everywhere.

For any $\rho,p \in \mathbb{R}^+$ and $\vec v\in {\mathbb {R}}^3$ satisfying $v=|\vec v| < 1$, it holds that
$$0< E^2-(D^2+|\vec m|^2)=\frac{1}{1-v^2} \left[ \rho^2(1+e)^2-\rho^2-p^2v^2 \right].$$
The arbitrary of $\vec v \in {\mathbb {R}}^3$ with $ v < 1$ yields
$$\inf_{v < 1} \big( \rho^2(1+e)^2-\rho^2-p^2v^2 \big) \ge 0.$$
Thus one has $ \rho^2(1+e(p,\rho))^2-\rho^2-p^2\ge0 $, which is equivalent to \eqref{eq:hcondition1} by noting $e>0$ and \eqref{eq:h}.
\end{appendices}

\allauthors

\listofchanges

\end{document}